\documentclass[11pt,letterpaper]{amsart}
 
\usepackage[T1]{fontenc}
\usepackage{lmodern}
\usepackage[protrusion=true,expansion=false]{microtype}
\usepackage{amsmath,amssymb,amsthm,mathtools,mathrsfs,esint}
\usepackage{aliascnt}
\usepackage{enumitem}
\usepackage{hyperref,cleveref,tikz}

\numberwithin{equation}{section}
\newtheorem{thm}{Theorem}[section]
\newtheorem{lem}[thm]{Lemma}
\newtheorem{prop}[thm]{Proposition}

\newtheorem{cor}[thm]{Corollary}

\theoremstyle{remark}
\newtheorem{rem}[thm]{Remark}

\newcommand{\R}{\mathbb R}
\newcommand{\dd}{\,d}
\newcommand{\loc}{\mathrm{loc}}
\newcommand{\dist}{\operatorname{dist}}
\newcommand{\tr}{\operatorname{tr}}
\newcommand{\Id}{\operatorname{Id}}
\newcommand{\Acal}{\mathcal A}
\newcommand{\Xbf}{\mathbf X}
\newcommand{\Lscr}{\mathscr L}
\newcommand{\Pscr}{\mathscr P}
\newcommand{\Mscr}{\mathscr M}
\newcommand{\Wolff}{\mathbf W}
\newcommand{\weakto}{\rightharpoonup}

\title[Liouville classification and Harnack inequalities]
{Critical $p$-Laplace equations with monotone coefficients: Liouville classification and a Schoen-type Harnack inequality}
\author{Yi Ru-Ya Zhang}
\date{\today}

\address{State Key Laboratory of Mathematical Sciences, Academy of Mathematics and Systems Science, Chinese Academy of Sciences, Beijing 100190, China}
\address{Institute of Mathematics, Academy of Mathematics and Systems Science, Chinese Academy of Sciences, Beijing 100190, China}
\email{yzhang@amss.ac.cn}

\thanks{The author was supported by the National Key R\&D Program of China (Grant Nos. 2025YFA1018400 and 2021YFA1003100), the NSFC (Grant Nos. 12288201 and 12571128), the Chinese Academy of Sciences, and the CAS Project for Young Scientists in Basic Research (Grant No. YSBR-031).}


\begin{document}

\begin{abstract}
We investigate positive weak solutions of the critical $p$-Laplace equation
$$
 -\Delta_p u = u^{p^*-1} h(u), \qquad 1 < p < n,
$$
where $h$ is a positive, bounded, continuous, and nonincreasing function. Our first main result is a complete classification of normalized, bounded, positive entire solutions for every equation in a compact family determined by $h$: Any such solution must coincide with an Aubin--Talenti profile. Moreover, the existence of an Aubin--Talenti profile as a solution implies that $h$ is constant on the entire interval of values attained by that profile. 

Subsequently, applying this classification result, we establish the following scale-invariant Schoen-type estimate
$$
 \left(\sup_{B_R} u\right)\left(\inf_{B_{2R}} u\right)^{p-1}
 \le C R^{p-n}
$$
for nonnegative weak solutions defined in $B_{3R}$. As a direct corollary, we apply the profound result of Ciraolo--Gatti \cite{CG2026} to obtain a fully unrestricted Liouville theorem: Every positive entire solution must coincide with an Aubin--Talenti profile. For the purely critical equation, we also show that the corresponding Liouville classification is equivalent to a Schoen-type Harnack inequality. 

The arguments in this work give quasilinear versions of the Kelvin transform and the method of moving spheres, which was previously available only in the semilinear setting, and also yield alternative proofs of the classical results. The technique developed here can likely be extended to a wider class of Liouville-type problems for critical equations.
\end{abstract}

\maketitle

\section{Introduction and main results}
\label{sec:introduction}

Liouville theorems and scale invariant Harnack estimates are two closely related components of the compactness theory for critical elliptic equations.
The model semilinear problem is
\begin{equation}\label{eq:intro-yamabe}
 -\Delta u=u^{\frac{n+2}{n-2}},\qquad u>0,\qquad \text{in }\R^n,
\end{equation}
for $n\ge3$.  Its geometric precursor is the rigidity theorem of Obata \cite{O1971}. The classification of \eqref{eq:intro-yamabe} in the absence of any \emph{a priori} decay assumptions was subsequently obtained by Caffarelli–Gidas–Spruck in their seminal manuscript \cite{CGS1989} via moving plane method, and a simplified argument  was later provided by Chen–Li \cite{CL1991}. In the same conformally invariant setting, Schoen's sup--inf estimate, 
recorded and developed in a series of remarkable works \cite{SZ1996,L1999,LZ2003}, provides the local quantitative counterpart of the entire classification. Li--Zhang \cite{LZ2003}, in particular, treated a broad class of critical semilinear nonlinearities and made 
explicit the interaction between Liouville and Harnack statements.

For the $p$-Laplacian with $p\ne2$, the corresponding theory is substantially less rigid at the level of available tools.  The Kelvin transform, in particular, no longer preserves the equation, and the moving plane/sphere argument applied in the case $p=2$ does not directly extend. Indeed, the Kelvin transform has no reasonable counterpart for general values of the exponent $1<p<n$ with $p\neq 2$; see e.g. \cite{L2016} and the reference therein.

For the purely critical equation
\begin{equation}\label{eq:intro-pure-critical}
 -\Delta_p u = u^{p^*-1}\quad\text{in }\R^n,
 \qquad p^*:=\frac{np}{n-p},
\end{equation}
a classification in the natural energy space was established through the works of Damascelli--Merchán--Montoro--Sciunzi \cite{DMMS2014}, V\'etois \cite{V2016}, and Sciunzi \cite{S2016}. An anisotropic finite-energy theory in convex cones was subsequently developed by Ciraolo--Figalli--Roncoroni \cite{CFR2020}. Beyond the energy framework, Catino--Monticelli--Roncoroni \cite{CMR2023} initiated an integral-identity approach that applies in several regimes of the parameter $p$. More recently, Ou \cite{O2025} obtained a comprehensive classification result without additional constraints for the range $p>(n+1)/3$ by employing an elegantly transformed stress-tensor approach. Subsequently, V\'etois \cite{V2024} refined this result by further lowering the admissible threshold for $p$ in dimensions $n\ge4$.In addition, Ciraolo–Gatti \cite{CG2026} classified bounded, or moderately growing, solutions satisfying the sharp  infimum decay on balls, and related extensions to convex cones were derived in \cite{CDGL2026}. These contributions rest on distinct and non-nested sets of assumptions: Some results classify arbitrary positive solutions only for a restricted range of exponents $p$, whereas others cover the full interval $1<p<n$ under supplementary conditions such as finite energy, boundedness, prescribed growth, or specific asymptotic behavior at infinity. Moreover, the author is also partially motivated by the recent progress on the stability of $p$-Sobolev inequalities; see \cite{CFMP2009, FN2019,FZ2022,CG20262} and \cite{LZ2025}, and the earlier work \cite{BE1991, CFM2018, FG2020, DSW2025} in the case $p=2$.

The purpose of this paper is to obtain a classification-to-Harnack theory for the full range $1<p<n$ under a monotonicity assumption on the nonlinearity.
We consider
\begin{equation}\label{eq:g-structure}
 -\Delta_pu=g(u),  \qquad  g(s)=s^q h(s),  \qquad  q:=p^*-1,
\end{equation}
where
\begin{equation}\label{eq:h-assumptions}
 h\in C((0,\infty)),  \qquad  0<L\le h(s)\le\Lambda<\infty,
 \qquad  h\text{ is nonincreasing}.
\end{equation}
We further set $g(0)=0$.  By \eqref{eq:h-assumptions}, the limits
$$
 h_0:=\lim_{s\to 0^+}h(s),  \qquad h_\infty:=\lim_{s\to\infty}h(s)
$$
exist and belong to $[L,\Lambda]$.

A fundamental aspect of the blow-up problem is that the height of the sequence of functions is not restricted to converge solely to zero or to diverge to infinity. If its amplitude tends to a finite
number $\tau>0$, the limiting equation is
$$
 -\Delta_pV=V^q h(\tau V),
$$
which is not a pure-power equation a priori.  Thus even a complete classification of \eqref{eq:intro-pure-critical} would not, by itself, close the blow-up argument for a general $h$.  This motivates us to consider the following compact family
\begin{equation}\label{eq:defn-ftau}
 g_\tau(t):=
 \begin{cases}
 t^qh(\tau t),&0<\tau<\infty,\\
 h_0t^q,&\tau=0,\\
 h_\infty t^q,&\tau=\infty.
 \end{cases}
\end{equation}
For a more detailed discussion, we refer the reader to Proposition~\ref{prop:point-selection} below, and in particular to the normalized formulation given by \eqref{eq:Utau-normalized}.

Our first main result classifies precisely the normalized entire profiles that can arise from this family. 
Set
$$
 p':=\frac{p}{p-1},  \qquad  a:=\frac{n-p}{p}.
$$
As for the fully unrestricted version, see Corollary~\ref{cor:full-liouville} below. 

\begin{thm}
\label{thm:frozen-classification}
Let $\tau\in[0,\infty]$, and suppose that
\begin{equation}\label{eq:frozen-equation}
 V\in W^{1,p}_{\loc}(\R^n),
 \qquad  -\Delta_pV=g_\tau(V)\quad\text{in }\R^n,
 \qquad  0<V\le1,
 \qquad  V(0)=1.
\end{equation}
Then
$$
 V(x)=\left(1+\gamma_\tau|x|^{p'}\right)^{-a},
$$
where
\begin{equation}\label{eq:gamma-tau}
 \gamma_\tau  =\frac1{p'}  \left(\frac{a^{1-p}\widehat h_\tau}{n}\right)^{1/(p-1)},
 \qquad 
 \widehat h_\tau:=
 \begin{cases}
 h(\tau),&0<\tau<\infty,\\
 h_0,&\tau=0,\\
 h_\infty,&\tau=\infty.
 \end{cases}
\end{equation}
If $0<\tau<\infty$, the existence of such a profile also forces
\begin{equation}\label{eq:h-flat-below-tau}
 h(s)=h(\tau)
 \qquad\text{for every }0<s\le\tau.
\end{equation}
\end{thm}

The last conclusion is an essential part of the theorem.  At a finite  amplitude, the result is simultaneously a rigidity theorem for the solution and for the coefficient: The apparently non-pure equation becomes pure power on the full amplitude interval attained by the profile.

Our second main result is the local estimate needed to control all scales between a concentrating core and the boundary.

\begin{thm}
\label{thm:main}
Assume the conditions \eqref{eq:g-structure} and \eqref{eq:h-assumptions}.  Then there exists 
$$C=C(n,p,g)>0$$
with the following property:  For every $R>0$ and
$x_0\in\R^n$, every nonnegative weak solution
$$
 u\in W^{1,p}_{\loc}(B_{3R}(x_0)),
 \qquad -\Delta_pu=g(u)\quad\text{in }B_{3R}(x_0),
$$
satisfies
\begin{equation}\label{eq:main-harnack}
 \left(\sup_{B_R(x_0)}u\right)
 \left(\inf_{B_{2R}(x_0)}u\right)^{p-1}
 \le C R^{p-n}.
\end{equation}
\end{thm}

\begin{rem} 
\label{rem:constant-dependence}
The contradiction argument proving Theorem~\ref{thm:main} is carried out for one fixed nonlinearity $g$.  Accordingly, the theorem states only the dependence 
$$C=C(n,p,g).$$
In particular, in the current manuscript we do not claim uniformity over all nonincreasing functions $h$ satisfying the same numerical bounds $L$ and $\Lambda$.
\end{rem}

To the best of our knowledge, even when $p=2$, no alternative proof of Theorems~\ref{thm:frozen-classification} and~\ref{thm:main} is currently available that completely avoids the moving-plane/moving-sphere method. 
For the reader’s convenience and in order to clarify the core underlying mechanism, we include in Appendix~\ref{appen:B} a concise exposition of the principal arguments in the case $p = 2$ with the pure-power nonlinearity.

Thanks to the profound result of Ciraolo--Gatti \cite[Theorem~1.2]{CG2026}, the combination of the two theorems stated above yields an immediate corollary, which provides a global classification without requiring any further auxiliary hypotheses.

\begin{cor} 
\label{cor:full-liouville}
Assume \eqref{eq:g-structure} and \eqref{eq:h-assumptions}, and let
$$
 u\in W^{1,p}_{\loc}(\R^n),
 \qquad u>0,  \qquad -\Delta_pu=u^qh(u)\quad\text{in }\R^n.
$$
Then there exist $M>0$ and $x_0\in\R^n$ such that
\begin{equation}\label{eq:full-liouville-profile}
 u(x)=M\left( 1+\gamma_M M^{p'/a}|x-x_0|^{p'}
 \right)^{-a},
\end{equation}
where $\gamma_M$ is given by \eqref{eq:gamma-tau} with $\tau=M$.
In particular, $M=\max_{\R^n}u$.  Moreover,
$$
 h(s)=h(M) \qquad\text{for every }0<s\le M.
$$
\end{cor}

\begin{rem} 
\label{rem:pure-power-equivalence}
For $h \equiv \lambda > 0$, Corollary~\ref{cor:full-liouville} yields that the Harnack inequality implies the unrestricted Liouville theorem for positive entire solutions. Conversely, an unrestricted pure-power Liouville theorem immediately provides the normalized classification of entire blow-up profiles required for the blow-up analysis, and the subsequent first-contact argument on suitable spheres then yields
the desired Harnack estimate. Consequently, within the structural framework developed in the current manuscript, these two properties are equivalent for the purely critical equation. In contrast, when $h$ is nonconstant, the finite-amplitude  equation is not necessarily of pure-power type, and this is precisely why Theorem~\ref{thm:frozen-classification} for the full compact family is genuinely indispensable.
\end{rem}

\subsection{Main ideas}

Let us describe the two parts of the proof and their relation to the existing literature.  The first part consists of a global classification result for the entire profile, whose underlying idea is  partially inspired by the seminal contribution of  \cite{O2025} by Ou.  For a solution $V$ to \eqref{eq:frozen-equation}, we introduce
\begin{equation}\label{eq:intro-transform}
 w:=V^{-p/(n-p)}=V^{-1/a}
\end{equation}
and the nonlinear stress
$$
 \Xbf_w:=|\nabla w|^{p-2}\nabla w.
$$
The transformation and trace-free derivative of $\Xbf_w$ are closely related to the quantities introduced by Ou in the study of the pure critical equation \cite[Lemma~2.1]{O2025}. In the pure-power case, Ou completed his argument by establishing a global weighted integral estimate, which yields the vanishing of the trace-free component of the stress tensor; see \cite[(3.34)]{O2025} and the subsequent discussion. The invariant-tensor framework and the dimensional cancellation phenomena underlying such vector fields are further developed by Ma and Wu \cite[Section~2.3, (2.8)]{MW2024}. In the semilinear Riemannian context, Ciraolo, Farina, and Polvara also derive  classification results and ambient rigidity by proving the constancy of an appropriately defined $P$-function in \cite{CFP2025}, and this is further generalized  for the quasilinear case  in the subsequent work \cite{SW2025} on the Riemannian manifold.

Our argument is based on a weak Bochner-type identity \eqref{eq:bochner} for a modified $P$-function denoted by $Q_w$, which appears to be new, although closely related ideas can already be found in certain computations in \cite{O2025}; see also \cite{MW2024}. More specifically, in the pure-power setting, the specialization of \eqref{eq:bochner} is (on the regular set) essentially equivalent to Ou’s identities, in particular \cite[Lemma 2.2, Formula (2.12)]{O2025}, upon choosing $q=1-n$ and $m=0$ in that reference.
Furthermore, while \cite[Proposition 2.3]{O2025} yields a global one-sided integral inequality, our identity \eqref{eq:bochner} leads to the fundamental PDE inequality \eqref{eq:completed-square} for the quantity $Q_w\phi_{\ell,\mathfrak m}$, where $\phi_{\ell,\mathfrak m}$ is a suitably constructed \emph{multiplier}. In this manner, via a nonlinear Kato-type inequality  in Lemma~\ref{lem:rayleigh}, the global integral coercivity estimate employed in \cite{O2025} is replaced with a local \emph{pointwise} coercive term; the author was later informed that a related idea was also applied in \cite{SW2025}. This inequality plays a central role in our setting: Indeed, it is employed together with the maximum principle to derive the subsequential upper bound \eqref{eq:Qw-upper}. 
This step is crucial because, in combination with \eqref{eq:completed-square}, it enables us to derive the desired rigidity result directly in the special case $p=2$ and $h\equiv 1$; see Appendix~\ref{appen:B} for additional discussion.  As for the full range $1<p<n$, however, the degeneracy of the underlying operator requires the introduction of further analytical tools, such as a tangent analysis and a punctured-ball barrier, to overcome the resulting difficulties. We emphasize that the use of the maximum principle in the normalized setting constitutes a principal distinction from the arguments in \cite{O2025} and \cite{SW2025}, where in both manuscripts global integral inequalities are eventually employed to establish the rigidity.

Let us be more specific. Indeed, there are two additional difficulties arising in our setting.  Firstly, $h$ is assumed only continuous and monotone, so its distributional derivative may contain a singular Cantor
part.  Thus we split the transformed coefficient into an absolutely continuous component and a nonnegative singular (Stieltjes) defect in \eqref{eq:defect-functions}.  Then after subtracting that defect from the transformed source, the pure-power first-order cancellation survives,
while the nonsmooth part enters the weak Bochner identity with a favorable sign; see Proposition~\ref{prop:bochner}. Secondly, a fixed power weight in the Bochner identity alone is not sufficient to deduce the desired rigidity, since it does not prevent a positive maximum of $Q_w$ from escaping from $\infty$. 
Consequently, we construct the aforementioned multiplier $\phi_{\ell,\mathfrak m}$ by solving the one-dimensional ordinary differential equation \eqref{eq:Y-ODE} backward, starting from the desired inequality \eqref{eq:completed-square}.
This gives a global bound from above for the modified $P$-function $Q_w$; see Proposition~\ref{prop:Q-bound}.  Now a tangent argument at a minimum, followed by employing a  barrier function on the punctured ball and an open-and-closed propagation, upgrades the scalar inequality to equality; see Proposition~\ref{prop:global-rigidity}. After that, the Bochner identity \eqref{eq:bochner} 
forces the tensor defect to vanish.  We refer to detailed comparisons with Ou's first-order and Bochner identities in Remark~\ref{rem:ou-bochner-comparison}, the multiplier design in Remark~\ref{rem:target-multiplier-design}, and the different closure mechanism in Remark~\ref{rem:ou-closure-comparison}; see also Remark~\ref{rem:stieltjes-splitting} for the Stieltjes decomposition.

The second part is the Harnack argument.  In an earlier joint work, Qin--Zhang \cite{QZ2026} proved a conditional
version under two inputs, a classification of the possible bounded entire blow-up limits and a preliminary neck estimate of order $|x|^{-a}$.  The preliminary estimate there supplies a uniform bound for the zeroth-order coefficient on a larger annulus at the first-crossing scale, which is precisely what makes the subsequent Harnack chain legitimate; see the proof of \cite[Lemma~3.7]{QZ2026}.
The current Theorem~\ref{thm:frozen-classification} completely removes the first conditional input.
Moreover, the second condition is removed by changing the order of the argument: Rather than use a neck bound to justify Harnack at first contact, we first derive the coefficient 
bound by  de-concentration via the celebrated work of Kilpel\"ainen and Mal\'y \cite{KM1994} on potential theory, and then apply Harnack inequality only afterwards. 
The exact distinction is recorded in Remark~\ref{rem:no-preliminary-neck}.



The principal contributions and the outline of the present manuscript can be summarized as follows.  First of all, we provide a complete classification of solutions within a compact family and establish rigidity of the coefficient at finite amplitude.  
Secondly, we generalize the transformed invariant-tensor method to the setting of merely continuous monotone coefficients by employing a weak Bochner identity \eqref{eq:bochner} in conjunction with a multiplier tailored to the target PDE inequality \eqref{eq:completed-square}.  
Thirdly, in the conditional Pohozaev-type argument, we replace the preliminary neck assumption with a first-contact de-concentration mechanism derived from potential-theoretic considerations.  
Fourthly, the ensuing Harnack estimate implies the full Liouville theorem stated in Corollary~\ref{cor:full-liouville}.

\subsection{Organization of the paper}

Section~\ref{sec:structure} records the relevant critical-exponent identities, establishes the sign of the Pohozaev invariant, and verifies the compactness properties of the nonlinear terms under consideration. It subsequently derives the weak-solution regularity framework that is employed throughout the article.  

Sections~\ref{sec:classification-start}--\ref{sec:classification-end} are devoted to the proof of Theorem~\ref{thm:frozen-classification}. In this part, we obtain a weak Bochner-type identity, construct an appropriate multiplier tailored to the target PDE inequality \eqref{eq:completed-square}, establish a global sharp bound for $Q_w$, and finally upgrade the resulting   equality to a corresponding tensorial rigidity statement.  

Section~\ref{sec:pole} develops a local Pohozaev identity and carries out the analysis of isolated $p$-harmonic poles. Section~\ref{sec:first-contact} proves the first-contact theorem and derives a de-concentration estimate at the moving contact point. Section~\ref{sec:final-proof} completes the contradiction argument on suitably chosen spheres and thereby establishes Theorem~\ref{thm:main}, together with Corollary~\ref{cor:full-liouville}.  

The concluding appendix presents the standard domain-variation derivation of the weak local Pohozaev identity, as well as a simplified argument in the special case $p=2$ and $h\equiv 1$.

\medskip
\noindent{\bf Acknowledgment}: The author expresses his sincere gratitude to Yanyan Li and Jingang Xiong for their valuable comments and suggestions on an earlier version of this manuscript.

The author used ChatGPT 5.6, together with Overleaf AI, to review the manuscript for grammar, clarity, calculations, logical consistency, and to correct certain technical details. The author then independently revised the text as needed and assumes full responsibility for all statements and conclusions in the final manuscript.

\section{Preliminaries}\label{sec:structure}

Recall that
$$
p^* := \frac{np}{n-p}, 
\qquad q := p^*-1,
$$
and set
\begin{equation*}
p':=\frac{p}{p-1}, \qquad a:=\frac{n-p}{p},
\qquad \alpha:=\frac{n-p}{p-1}=ap', 
\qquad b_p:=\frac{n(p-1)}p=\frac n{p'}.
\end{equation*}
We will repeatedly use the following identities.
\begin{equation}\label{eq:exponent-identities}
\begin{gathered}
a(q+1)=n,
\qquad aq=b_p+1=n-a,
\qquad q-p+1=\frac pa=\frac{p^2}{n-p},\\
\alpha(p-1)=n-p,
\qquad (\alpha+1)(p-1)=n-1,
\qquad n-\alpha q=-p'.
\end{gathered}
\end{equation}
In fact, applying $q+1=p^*=np/(n-p)$ into each expression gives these identities.

\subsection{Properties of the nonlinearity}
Define
$$
G(s):=\int_0^sg(t)\dd t.
$$
First of all, we show the following sign lemma. Note that the coefficient $a=(n-p)/p=n/p^*$ is forced by critical
scaling.  With this choice, the difference between the primitive term and the reaction term is exactly
$$
 nG(s)-asg(s)
 =n\int_0^s t^q\bigl(h(t)-h(s)\bigr)\,dt.
$$
Thus monotonicity of $h$ becomes a nonnegative Pohozaev defect, and equality detects the flat portions of $h$. See also \cite[Lemma 3.3]{QZ2026}. 

\begin{lem} 
For every $s\ge0$,
\begin{equation}\label{eq:pohozaev-sign}
nG(s)-a s g(s)\ge0.
\end{equation}
The same inequality is preserved by every amplitude scaling, i.e. for any $A>0$ 
$$g_A(t):=A^{-q} g(At), \qquad t>0$$
and spatial scaling, i.e. for any $\rho>0$
$$\widehat g(s)=\rho^n g(\rho^{-\alpha} s),\qquad s>0.$$
\end{lem}

\begin{proof}
At $s=0$ both terms vanish.  Let $s>0$.  Since $h$ is nonincreasing, $h(t)\ge h(s)$ for $0<t\le s$.  Therefore
$$
 G(s)=\int_0^st^qh(t)\dd t  \ge h(s)\frac{s^{q+1}}{q+1}.
$$
Since $n/(q+1)=a$ by the first identity in \eqref{eq:exponent-identities}, multiplication by $n$ gives
$nG(s)\ge a s^{q+1}h(s)=a s g(s)$.

Now we consider the scalings. 
For $A>0$, by a change of variable
$$
G_A(t):=\int_0^t g_A(s)\dd s=A^{-q-1}G(At).
$$
Then
$$
 n G_A(t)-a t g_A(t) =A^{-q-1}\bigl[nG(At)-a\,(At)g(At)\bigr]\ge0.
$$
Moreover, when a spatial blow-down is subsequently made with
$$
\widehat G(t):=\int_0^t \widehat g(s) \dd s=\rho^{n+\alpha}G(\rho^{-\alpha}t),
$$
then
$$
 n\widehat G(s)-a s\widehat g(s)
 =\rho^{n+\alpha} \bigl[nG(\rho^{-\alpha}s) -a\,(\rho^{-\alpha}s)g(\rho^{-\alpha}s)\bigr]\ge 0.
$$
\end{proof}

Recall $g_\tau$ defined in \eqref{eq:defn-ftau}.
We prove the following compactness result. 
\begin{lem}\label{lem:frozen-convergence}
If $A_j>0$ and $A_j\to\tau\in[0,\infty]$, then for $g_\tau$ defined in \eqref{eq:defn-ftau}, 
$g_{A_j}\to g_\tau$ locally uniformly on  $(0,\infty)$.
\end{lem}

\begin{proof}
Suppose first that $0<\tau<\infty$, and let
$K=[k_1,k_2]\Subset(0,\infty)$.
For all sufficiently large $j$,
$$
A_jt\in \left[\frac{\tau k_1}{2},\,2\tau k_2\right] \qquad \text{ for any }\ t\in K.
$$
Since $h$ is uniformly continuous on this fixed compact
interval,
\begin{align*}
 \sup_{t\in K}|g_{A_j}(t)-g_\tau(t)|
 &= \sup_{t\in K}  t^q|h(A_jt)-h(\tau t)|\\
 &\le k_2^q \sup_{t\in K}|h(A_jt)-h(\tau t)| \to 0.
\end{align*}

Now assume that $A_j\to 0$.  Then given any pair $0<m<M$, the nonincreasing property  of $h$ gives
$$
h(A_jM)\le h(A_jt)\le h(A_j m) \qquad \text{ for any} \ m\le t\le M.
$$
Note that both  the outer terms converge to $h_0$, uniformly squeezing the middle term and sending it to $h_0$ as well.
When $A_j\to\infty$, the same inequalities hold, and the two outer terms instead converge to $h_\infty$.
\end{proof}

\subsection{Weak-solution regularity}\label{sec:regularity}

Throughout this subsection, $\Omega\subset\R^n$ is open and
\begin{equation}\label{eq:weak-equation-general}
u\in W^{1,p}_{\loc}(\Omega),
\qquad  u\ge0,  \qquad
-\Delta_pu=g(u)\quad\text{in }\Omega.
\end{equation}
The weak formulation is meaningful since $u\in L^{p^*}_{\loc}$ and
$g(u)\le\Lambda u^{p^*-1}\in L^{(p^*)'}_{\loc}$.

Define
\begin{equation}\label{eq:defn-Ku}
 K_u(x):= 
 \begin{cases}
g(u(x))u(x)^{1-p} &u(x)>0,\\
0  &u(x)=0.
 \end{cases}
\end{equation}
Then
\begin{equation}\label{eq:estimate-Ku}
0\le K_u\le\Lambda u^{q-p+1},
\qquad K_u\in L^{n/p}_{\loc}(\Omega),
\end{equation}
since $(q-p+1)n/p=p^*$ by the third identity in \eqref{eq:exponent-identities}. The following lemma is sometimes called Brezis–Kato bootstrap.

\begin{lem} \label{lem:bk-step}
Let $B_r\subset\subset B_R\subset\subset\Omega$ and  $t\ge p$.  Suppose that  $u\in L^t(B_R)$. Then one has
\begin{equation}\label{eq:bk-step-estimate}
\|u\|_{L^{\chi t}(B_r)} \le C(n,p,t,r,R,u)\|u\|_{L^t(B_R)},
\qquad \chi:=\frac n{n-p}.
\end{equation}
Consequently, $u\in L^m_{\loc}(\Omega)$ for every finite $m$.
\end{lem}

\begin{proof}
Choose $\eta\in C_c^\infty(B_R)$ with $0\le\eta\le1$ and $\eta=1$ on $B_r$.  For $L>0$, we consider the truncated function as
$$
u_L:=\min\{u,L\}, \qquad \beta:=t-p\ge 0. 
$$
Then by Sobolev chain rule, the function
$$
\varphi:=\eta^p u u_L^\beta
$$
belongs to $W^{1,p}_0(B_R)$.  This can be obtained first with a bounded smooth approximation of $s\mapsto s\min\{s,L\}^\beta$ and then by passage to the limit; thus it is an admissible test function.  Then the source term is integrable since both \eqref{eq:g-structure} and \eqref{eq:h-assumptions} yield
$$
g(u)u u_L^\beta \le \Lambda L^\beta u^{p^*}\in L^1(B_R).
$$
Moreover, almost everywhere one has
\begin{align*}
 \nabla\varphi={}&p\eta^{p-1}u u_L^\beta\nabla\eta
 +\eta^p u_L^\beta\nabla u  +\beta\eta^p u u_L^{\beta-1}\nabla u_L.
\end{align*}
  
We note that
$$\int_{B_R}\beta\eta^p u u_L^{\beta-1}|\nabla u|^{p-2}\nabla u\cdot\nabla u_{L}\, \dd x=\int_{B_R\cap \{u< L\}}\beta\eta^p u^{\beta}|\nabla u|^{p} \,dx\ge 0$$
as $\beta\ge 0.$
Let $v_L:=u u_L^{\beta/p}$.
Then testing \eqref{eq:weak-equation-general} with $\varphi$ gives
\begin{multline*}
     \int_{B_R}\eta^p u_L^\beta|\nabla u|^p\,dx + \int_{B_R\cap \{u< L\}}\beta\eta^p u^{\beta}|\nabla u|^{p} \,dx \\=   \int_{B_R}g(u)\eta^p u u_L^\beta \,dx -\int_{B_R} p \eta^{p-1} u u_L^{\beta}|\nabla u|^{p-2} \nabla u\cdot\nabla \eta\,dx.
\end{multline*}
Now by removing the second term on the left-hand side, which is positive, and estimating the cutoff term by Young's inequality, we conclude 
\begin{equation}\label{eq:bk-energy-u}
 \int_{B_R}\eta^p u_L^\beta|\nabla u|^p\,dx
 \le C(p)\left[\int_{B_R}K_u(\eta v_L)^p\,dx + \int_{B_R}v_L^p|\nabla\eta|^p \,dx \right].
\end{equation}
where we applied \eqref{eq:defn-Ku} to get
$$
g(u)\eta^p u u_L^\beta =K_u\eta^p u^p u_L^\beta =K_u(\eta v_L)^p.
$$

Observe that, on $\{u<L\}$,
$\nabla v_L=(t/p)u_L^{\beta/p}\nabla u$ as $t/p=\beta/p+1\ge 1$, whereas on $\{u>L\}$,
$\nabla v_L=L^{\beta/p}\nabla u$.  Thus
$$
|\nabla v_L|^p \le\left(\frac tp\right)^p u_L^\beta|\nabla u|^p.
$$
Combining this with \eqref{eq:bk-energy-u}, the product estimate for
$\nabla(\eta v_L)$, and the Sobolev inequality yields
\begin{align}\label{eq:bk-sobolev}
\|\eta v_L\|_{L^{p^*}(B_R)}^p &\le C(n,p)  \|\nabla(\eta v_L)\|_{L^{p}(B_R)}^p\notag\\
&\le   C(n,p)t^p\left[\int_{B_R}K_u(\eta v_L)^p \,dx + \int_{B_R}v_L^p|\nabla\eta|^p\,dx  \right].
\end{align}

Now fix $t$ and split $K_u$ at a level $M>0$.  By H\"older's inequality,
$$
 \int_{\{K_u>M\}}K_u(\eta v_L)^p\,dx
 \le \|K_u\mathbf1_{\{K_u>M\}}\|_{L^{n/p}(B_R)} \|\eta v_L\|_{L^{p^*}(B_R)}^p,
$$
where we note that the first factor tends to zero as $M\to\infty$ by \eqref{eq:estimate-Ku}.  Thus we can choose $M=M(n,p,t, K_u|_{B_R})\gg 1$ so that this term can be absorbed into the left side of \eqref{eq:bk-sobolev}.

On the other hand, on the set $\{K_u\le M\}$, the first term on the right-hand side of \eqref{eq:bk-sobolev} is bounded by
$M\int(\eta v_L)^p$.  Hence
$$
\|v_L\|_{L^{p^*}(B_r)}^p \le C(n,p,t,r,R,u)\int_{B_R}v_L^p \,dx.
$$
Then monotone convergence  gives
$$
\|u\|_{L^{\chi t}(B_r)}^t \le C(n,p,t,r,R,u)\|u\|_{L^t(B_R)}^t,\quad \text{ as $L\to\infty$}
$$
which implies \eqref{eq:bk-step-estimate} after taking the $t$-th root.

The Sobolev embedding gives $u\in L^{p^*}_{\loc}$.  Then starting with $t_0=p^*$, choose a finite strictly nested sequence of balls, apply the estimate \eqref{eq:bk-step-estimate} with $t_{k+1}=\chi t_k$, and iterate.  This gives the estimates with arbitrarily large finite exponents, and H\"older's inequality then gives the desired estimates with every smaller finite exponent. 
\end{proof}

For a nonnegative Radon measure $\mu$, define the truncated Wolff potential
\begin{equation}\label{eq:wolff}
 \Wolff_{1,p}^{\mu}(x;r) :=\int_0^r  \left(\frac{\mu(B_t(x))}{t^{n-p}}\right)^{1/(p-1)} \frac{\dd t}{t}.
\end{equation}
The potential estimate of Kilpel\"ainen and Mal\'y
\cite[Theorem~1.6]{KM1994} gives
\begin{equation}\label{eq:KM-estimate}
 c_1\Wolff_{1,p}^{\mu}(x;\rho) \le v(x)  \le c_2\inf_{B_\rho(x)}v      +c_3\Wolff_{1,p}^{\mu}(x;2\rho),
\end{equation}
for the nonnegative $p$-superharmonic representative $v$ of a solution of
$$
-\Delta_pv=\mu\ge0,
$$
whenever the corresponding larger ball is compactly contained in the domain; see also \cite{KM2012}.
We also refer to the foundational framework established by Heinonen, Kilpel\"ainen, and Martio \cite{HKM06} in their seminal monograph, which is   known within the analysis community as the "Finnish Bible". In addition, whenever we invoke the Kilpel\"ainen--Mal\'y framework, we consistently choose the canonical superharmonic representative, even though this implicit selection is not stated explicitly. 

\begin{prop} \label{prop:local-boundedness}
Every weak solution satisfying \eqref{eq:weak-equation-general} belongs to
$L^\infty_{\loc}(\Omega)$.
\end{prop}

\begin{proof}
By Lemma~\ref{lem:bk-step}, $u\in L^m_{\loc}$ for every finite $m$.  Choose $s>n/p$.  Since $g(u)\le\Lambda u^q$, taking $m=qs$ gives $g(u)\in L^s_{\loc}$.

Fix $B_{3\rho}(x)\subset\subset\Omega$ and let $\dd\mu=g(u)\dd x$.  For $0<t\le2\rho$, we estimate the measure $\mu$ via H\"older's inequality as follows: 
$$
 \mu(B_t(x)) \le |B_t|^{1-1/s}\|g(u)\|_{L^s(B_{2\rho}(x))}
 \le C(n) t^{n-n/s}\|g(u)\|_{L^s(B_{2\rho}(x))}.
$$
Plugging into \eqref{eq:wolff} with
$$\frac{n-\frac n s -(n-p)}{p-1} -1=\frac{p-\frac n s}{p-1}-1, $$
we deduce that
\begin{equation}\label{eq:wolff-Ls-bound}
 \Wolff_{1,p}^{\mu}(x;2\rho) \le C(n,p,s)\rho^{(p-n/s)/(p-1)}
 \|g(u)\|_{L^s(B_{2\rho}(x))}^{1/(p-1)}.
\end{equation}
The exponent is positive as $p-n/s>0$. The infimum appearing in the upper bound \eqref{eq:KM-estimate} is dominated by  
$\fint_{B_\rho(x)} u\,dx$, the mean value of $u$ over the ball $B_\rho(x)$.  Moreover, these mean values remain uniformly bounded as $x$ varies within any fixed compact subset.

As $g(u)\in L^s_{\loc}$, \eqref{eq:KM-estimate} and \eqref{eq:wolff-Ls-bound} yield a uniform pointwise bound on every compactly contained ball.  Hence $u\in L^\infty_{\loc}$.
\end{proof}

\begin{cor} \label{cor:C1-regularity}
Every solution satisfying \eqref{eq:weak-equation-general} belongs to $C^{1,\gamma}_{\loc}(\Omega)$ for some $\gamma\in(0,1)$.  If $\Omega$ is connected, then either $u\equiv0$ or $u>0$ in $\Omega$.
\end{cor}

\begin{proof}
Proposition~\ref{prop:local-boundedness} gives $u\in L^\infty_{\loc}$, hence $g(u)\in L^\infty_{\loc}$.  Then Tolksdorf's interior gradient theorem \cite{T1984} applies to
$$
-\Delta_pu=g(u),\qquad g(u)\in L^\infty_{\loc},
$$
and gives $u\in C^{1,\gamma}_{\loc}$; see also the work of DiBenedetto \cite{D1983}. For positivity, V\'azquez's strong minimum principle \cite[Theorem~5]{V1984} applies with
$$
\Delta_pu=-g(u)\le0 \quad \text{ and }
\quad \Delta_pu\in L^\infty_{\loc} .
$$
Therefore, on each connected component, either $u\equiv0$ or $u>0$.
\end{proof}

The following lemma is standard. 
\begin{lem}
\label{lem:uniform-C1-compactness}
Let $B_{2R}(x_0)\subset\mathbb{R}^n$, and let
$$
u_j \in W^{1,p}\bigl(B_{2R}(x_0)\bigr)\cap L^\infty\bigl(B_{2R}(x_0)\bigr)
$$
be weak solutions of the boundary value problem
$$
-\Delta_p u_j = f_j \qquad \text{in } B_{2R}(x_0),
$$
where the source terms $f_j$ are not assumed to have a fixed sign (i.e., they may take both positive and negative values). Assume that
$$
\sup_j\|u_j\|_{L^\infty(B_{2R}(x_0))}\le M_0,
\qquad \sup_j\|f_j\|_{L^\infty(B_{2R}(x_0))}\le F_0.
$$
Then there exist $\beta=\beta(n,p)\in(0,1)$ and
$C=C(n,p,R,M_0,F_0)$ such that
$$
\|u_j\|_{C^{1,\beta}(B_R(x_0))}\le C \qquad\text{for every }j.
$$
Consequently, $\{u_j\}$ is precompact in $C^1(B_R(x_0))$.
Moreover, if $u_j\to u$ in $C^1_{\rm loc}$ and
$f_j\to f$ in $L^1_{\rm loc}$, then $u$ is a weak solution of
$-\Delta_pu=f$.
\end{lem}

\begin{proof}
The uniform $C^{1,\beta}$-estimate is the interior gradient estimate for bounded solutions with bounded right-hand side; see Tolksdorf \cite[Section~2]{T1984}; the theorem applies to a signed right-hand side as well\footnote{If one uses the version stated in terms of a local energy norm, the required
uniform energy bound follows first from the Caccioppoli inequality and the two assumed $L^\infty$-bounds.}.  Now the desired compactness conclusion follows from Arzel\`a--Ascoli.  Finally,
$$
|\nabla u_j|^{p-2}\nabla u_j \to |\nabla u|^{p-2}\nabla u
$$
locally uniformly under $C^1_{\rm loc}$-convergence, hence passage to the weak formulation proves the last claim of the lemma.
\end{proof}

Now we conclude the following relative Harnack inequality together with a gradient estimate.
\begin{lem} 
\label{lem:relative-gradient}
Let $R>0$ and let $Z>0$ solve
$$
-\Delta_pZ=b(x)Z^{p-1} \quad\text{in }B_{2R}(y),
\qquad 0\le R^p b(x)\le N_0.
$$
Then
$$
\sup_{B_R(y)}Z\le C\inf_{B_R(y)}Z,
\qquad |\nabla Z(y)|\le \frac{C}{R}Z(y),
$$
where $C=C(n,p,N_0)$.
\end{lem}

\begin{proof}
Set $\widetilde Z(\xi):=Z(y+R\xi)$. Then
$$
-\Delta_p\widetilde Z =R^pb(y+R\xi)\widetilde Z^{p-1} \quad\text{in }B_2.
$$
Now the unit-scale Harnack inequality gives the first estimate after scaling.

Moreover, Tolksdorf's interior gradient estimate \cite{T1984}, followed by scaling back, gives
$$
|\nabla Z(y)|\le C\left[ R^{-1}\sup_{B_R(y)}Z +R^{1/(p-1)}
\|bZ^{p-1}\|_{L^\infty(B_R(y))}^{1/(p-1)} \right].
$$
Since $R^p\|b\|_{L^\infty}\le N_0$, the second term is bounded by $CR^{-1}\sup_{B_R(y)}Z$. 
Thus, Harnack estimate together with
$\inf_{B_R(y)}Z\le Z(y)$ completes the proof.
\end{proof}

\section{Classification of profiles}

\subsection{Transformed equation}\label{sec:classification-start}

Corollary~\ref{cor:C1-regularity} in the earlier section gives
$V\in C^{1,\gamma}_{\loc}$ and $V>0$.  Since the map $s\mapsto s^{-1/a}$ is smooth on every compact
subset of $(0,\infty)$, we may define
\begin{equation}\label{eq:defn-w}
w:=V^{-1/a}\ge1, \qquad w\in C^{1,\gamma}_{\loc}(\R^n).
\end{equation}
For $t\ge1$, put
\begin{equation}\label{eq:c-tau}
 c_\tau(t):=
 \begin{cases}
a^{1-p}h(\tau t^{-a}),&0<\tau<\infty,\\
a^{1-p}h_0,&\tau=0,\\
a^{1-p}h_\infty,&\tau=\infty.
 \end{cases}
\end{equation}
The function $c_\tau$ is continuous, positive, bounded, and nondecreasing.
For abbreviation, set
\begin{equation}\label{eq:defn-Pw}
\Acal(\xi):=|\xi|^{p-2}\xi,
\qquad \Xbf_w:=\Acal(\nabla w),
\quad \text{ and } \quad  P_w:=\operatorname{div}\Xbf_w\in\mathcal D'(\R^n).
\end{equation}

\begin{lem}
The function $w$ satisfies
\begin{equation}\label{eq:transformed-equation}
wP_w=b_p|\nabla w|^p+c_\tau(w)
\end{equation}
in the weak sense in $\R^n$.
\end{lem}

\begin{proof}
Since $V=w^{-a}$ by \eqref{eq:defn-w}, the identities below follow
\begin{equation}\label{eq:defn-AV}
   \nabla V=-aw^{-a-1}\nabla w, \qquad \Acal(\nabla V)=-a^{p-1}w^{-b_p}\Xbf_w 
\end{equation}
where one used $(a+1)(p-1)=\frac{n(p-1)}{p}=b_p$.  Let $\eta\in C_c^\infty(\R^n)$.  

Since $w\in C^1_{\loc}(\R^n)$, $w\ge 1$, and
$\eta$ has compact support, one has
$$
\psi=a^{1-p}w^{b_p}\eta \in W^{1,p}_0(\operatorname{spt}\eta)\cap L^\infty.
$$
Moreover, $g_\tau(V)\in L^\infty_{\loc}$.
Thus the weak formulation, initially stated for smooth compactly supported test functions, extends to $\psi$ by density. Namely $\psi$ is an admissible test function in the equation for $V$.  

Note that, \eqref{eq:defn-AV} gives
$$
\int_{\R^n}\Acal(\nabla V)\cdot\nabla\psi\,dx
=-\int_{\R^n}\Xbf_w\cdot\nabla\eta\,dx -b_p\int_{\R^n}\frac{|\nabla w|^p}{w}\eta\,dx.
$$
On the other hand,  since  $aq=b_p+1$ by the second identity in \eqref{eq:exponent-identities},  the definition of
$c_\tau$ in \eqref{eq:c-tau} gives, for every $\tau\in[0,\infty]$,
$$
\int_{\R^n}g_\tau(V)\psi\,dx
= \int_{\R^n} a^{1-p}w^{b_p}g_\tau(w^{-a})\eta\,dx 
= \int_{\R^n} c_\tau(w)w^{b_p-aq}\eta\,dx
= \int_{\R^n}\frac{c_\tau(w)}w\eta\,dx.
$$
Thus, by testing the equation for $V$ via $\psi$, it follows from the two identities above that
$$
 -\int\Xbf_w\cdot\nabla\eta\,dx  =\int\frac{b_p|\nabla w|^p+c_\tau(w)}w\eta\,dx,
$$
which yields exactly \eqref{eq:transformed-equation}.
\end{proof}

Define the following functions
\begin{align}
 c_{\tau,0}&:=c_\tau(1), \notag\\
 \mathfrak k_\tau(t)&:=c_{\tau,0}
-\int_1^t\frac{c_\tau(s)}{s^2}\dd s,\label{eq:defect-functions}\\
 \mathfrak e_\tau(t)&:=\frac{c_\tau(t)}t-\mathfrak k_\tau(t),
 \qquad t\ge1. \notag
\end{align}
The function $\mathfrak k_\tau$ may change sign.
By contrast, the next lemma shows that
$\mathfrak e_\tau$ is nonnegative and nondecreasing.

\begin{lem}
The function $\mathfrak e_\tau$ is continuous, nonnegative, and nondecreasing.  In the sense of Stieltjes measures, i.e. the weak derivative of $BV$-functions,
\begin{equation}\label{eq:stieltjes-defect}
\dd\mathfrak e_\tau(t)=\frac1t\dd c_\tau(t).
\end{equation}
Equivalently, for every $1\le u<v$,
\begin{equation}\label{eq:int-ctau}
\mathfrak e_\tau(v)-\mathfrak e_\tau(u) =\int_{(u,v]}\frac 1 t \dd c_\tau(t)\ge 0.
\end{equation}
\end{lem}

\begin{proof}
Since $c_\tau$ is continuous and nondecreasing, it is locally of bounded variation. Then Riemann-Stieltjes integration by parts gives
$$
 \frac{c_\tau(v)}v-\frac{c_\tau(u)}u
 =\int_{(u,v]}\frac1t\dd c_\tau(t)  -\int_u^v\frac{c_\tau(t)}{t^2}\dd t.
$$
Adding $\mathfrak k_\tau(u)-\mathfrak k_\tau(v)$ to the identity above proves the identity  \eqref{eq:int-ctau}, and hence \eqref{eq:stieltjes-defect}.  Moreover, since $\mathfrak e_\tau(1)=0$, the nonnegativity in \eqref{eq:int-ctau} follows.
\end{proof}

\begin{rem}\label{rem:stieltjes-splitting}
The identity
$$
\dd\left(\frac{c_\tau(t)}{t}\right)
=-\frac{c_\tau(t)}{t^2}\,\dd t+\frac1t\,\dd c_\tau(t)
$$
splits the derivative into an absolutely continuous part and a possibly singular part, since a continuous monotone function may have a Cantor component. Thus the decomposition
$$
\frac{c_\tau(t)}t=\mathfrak k_\tau(t)+\mathfrak e_\tau(t)
$$
isolates these two contributions. Here
$Q_w:=P_w-\mathfrak e_\tau(w)$, defined below, has the same
differentiable first-order structure as the transformed source in the pure-power problem, whereas the nonsmooth part survives as a favorable nonnegative measure in the weak Bochner identity introduced later.

In addition, note that when $c_\tau\equiv c_{\tau, 0}$, one has
$$\mathfrak k_\tau(t)= \frac{c_{\tau, 0}} t, \qquad \mathfrak e_\tau\equiv 0, \qquad Q_w=P_w.$$
This is the reduction to the transformed pure-power framework in \cite{O2025}.
\end{rem}

Recall the definition of $P_w$ in \eqref{eq:defn-Pw} and define
\begin{equation}\label{eq:defn-Qw}
Q_w:=P_w-\mathfrak e_\tau(w).
\end{equation}
By \eqref{eq:transformed-equation} and \eqref{eq:defect-functions},
\begin{equation}\label{eq:Q-algebraic}
Q_w=\mathfrak k_\tau(w)+\frac{b_p}{w}|\nabla w|^p.
\end{equation}
Let
$$
\Omega_{\mathrm{reg}}:=\{x\in\R^n:\nabla w(x)\ne0\},
$$
which is open by the $C^{1,\gamma}$-regularity of $w$. 

\begin{lem} \label{lem:Q-derivative}
On $\Omega_{\mathrm{reg}}$, $Q_w\in W^{1,2}_{\loc}$ and
\begin{equation}\label{eq:Q-derivative}
w\nabla Q_w =b_p\nabla(|\nabla w|^p)-P_w\nabla w
\end{equation}
almost everywhere.
\end{lem}

\begin{proof}
The transformed equation may be written in the sign convention of Antonini--Ciraolo--Farina \cite{ACF2023} as
$$
-\operatorname{div}\Xbf_w = -P_w,
\qquad \Xbf_w = |\nabla w|^{p-2}\nabla w,
\qquad P_w = \frac{b_p|\nabla w|^p+c_\tau(w)}{w}.
$$
On each compact set $K$, $w$ is bounded above and below away from zero, $\nabla w$ is bounded, and $c_\tau(w)$ is bounded.  Hence
$$
P_w\in L^\infty_{\loc}(\R^n).
$$
Thus the hypotheses of Antonini--Ciraolo--Farina
\cite[Theorem~1.1]{ACF2023} are satisfied with the notation there
$$
H(\xi)=|\xi|, \qquad 
a(\xi)=\frac 1 p \nabla_\xi H(\xi)^p =|\xi|^{p-2}\xi,
\qquad f=-P_w\in L^\infty_{\loc}.
$$
Consequently,
$$
    \Xbf_w\in W^{1,2}_{\loc}(\R^n;\R^n).
$$
Let $K\subset\subset\Omega_{\mathrm{reg}}$.  Since
$\nabla w$ is continuous and nonvanishing on $K$, there are
constants $0<\delta_K\le M_K<\infty$ such that
$$
    \delta_K\le|\nabla w|\le M_K \qquad\text{on }K.
$$
Then the inverse of the stress map
$$
    \Acal^{-1}(\zeta)=|\zeta|^{p'-2}\zeta
$$
is Lipschitz on the compact annulus containing
$\Xbf_w(K)$.  Thus, the Sobolev chain rule  gives
$$
    \nabla w=\Acal^{-1}(\Xbf_w)\in W^{1,2}(K),
$$
and hence $w\in W^{2,2}(K)$.  Formula \eqref{eq:Q-algebraic} then implies $Q_w\in W^{1,2}(K)$.

Observe that \eqref{eq:defect-functions} shows that
$\mathfrak k_\tau\in C^1$ and
$\mathfrak k_\tau'(t)=-c_\tau(t)/t^2$.  Differentiating
\eqref{eq:Q-algebraic} almost everywhere on the regular set gives
\begin{align*}
 \nabla Q_w &= -\frac{c_\tau(w)}{w^2}\nabla w + \frac{b_p}{w}\nabla(|\nabla w|^p) -  \frac{b_p}{w^2} |\nabla w|^p \nabla w \\
 & = \frac{b_p}{w}\nabla(|\nabla w|^p)
  -\frac{c_\tau(w)+b_p|\nabla w|^p}{w^2}\nabla w.    
\end{align*}
Then the transformed equation \eqref{eq:transformed-equation} identifies the numerator in the last term with $wP_w$, proving \eqref{eq:Q-derivative}.
\end{proof}

\subsection{The weak Bochner identity}

On the regular set $\Omega_{\mathrm{reg}}$, define
\begin{equation}\label{eq:defn-Ew}
\mathsf A_w :=D_\xi\Acal(\nabla w),
\qquad \mathsf E_w:=D_x\Xbf_w-\frac{P_w}{n}\Id.
\end{equation}
Thus, if $\mathbf e=\nabla w/|\nabla w|$, then
$$
\mathsf A_w =|\nabla w|^{p-2}
\bigl(\Id+(p-2)\mathbf e\otimes\mathbf e\bigr).
$$
In particular, $\mathsf A_w$ is symmetric and positive definite on
$\Omega_{\mathrm{reg}}$.

We use the convention
$$
(D_x\Xbf_w)_{ij}=\partial_j(\Xbf_w)_i.
$$
Writing $H=D^2w$, we have
\begin{equation}\label{eq:DXw-AH}
D_x\Xbf_w=\mathsf A_w H.
\end{equation}
Moreover,
$$
\operatorname{tr}(D_x\Xbf_w) =\operatorname{div}\Xbf_w=P_w.
$$
Together with
$$
p b_p=n(p-1),\qquad \nabla(|\nabla w|^p)=pH\Xbf_w,\qquad
\mathsf A_w\nabla w=(p-1)\Xbf_w,
$$
these identities and \eqref{eq:Q-derivative} yield the following first-order tensor identity:
\begin{align}
n\mathsf E_w^{\,T}\nabla w = & {} n H \mathsf A_w \nabla w-P_w\nabla w =n(p-1)H\Xbf_w -P_w\nabla w \notag\\
= & {} b_p \nabla(|\nabla w|^p)-P_w\nabla w  ={w}\nabla Q_w \label{eq:Ou-form}
\end{align}
almost everywhere on $\Omega_{\mathrm{reg}}$.
This result is the analogue of \cite[Lemma 2.1 (i)]{O2025} for the transformed purely critical equation. It is noteworthy that Ou utilized the same transformation and the same stress field, but in combination with a trace-free defect tensor which is, in general, not symmetric.

\begin{lem} \label{lem:stress-relation}
Almost everywhere on $\Omega_{\mathrm{reg}}$,
\begin{equation}\label{eq:stress-relation}
\mathsf E_w\Xbf_w =\frac{w}{n(p-1)}\mathsf A_w\nabla Q_w.
\end{equation}
Moreover,
\begin{equation}\label{eq:defn-Tw}
\mathsf T_w:=\mathsf A_w^{-1/2} \mathsf E_w\mathsf A_w^{1/2}
\end{equation}
is symmetric and trace-free.  Hence
\begin{equation}\label{eq:E-square-positive}
\tr(\mathsf E_w^2)=|\mathsf T_w|^2\ge0.
\end{equation}
\end{lem}
\begin{proof}
We apply  \eqref{eq:DXw-AH} again to get
$$
\frac{w}{n(p-1)}\mathsf A_w\nabla Q_w
= \frac1{p-1} \mathsf A_w\mathsf E_w^{\,T}\nabla w
= \frac1{p-1} \mathsf E_w\mathsf A_w\nabla w
= \mathsf E_w\Xbf_w,
$$
where we applied \eqref{eq:DXw-AH} to get
$$
\mathsf E_w\mathsf A_w=\mathsf A_wH\mathsf A_w - \frac{P_w} n \mathsf A_w=\mathsf A_w\mathsf E_w^{\,T}.
$$
This is exactly \eqref{eq:stress-relation}. 

In addition, by \eqref{eq:defn-Ew} and \eqref{eq:DXw-AH}
$$
 \mathsf T_w  =\mathsf A_w^{1/2}H\mathsf A_w^{1/2}-\frac{P_w}{n}\Id,
$$
which is symmetric (while $\mathsf E_w$ is generally not).  As similarity preserves trace, we have $\tr\mathsf T_w=\tr\mathsf E_w=0$, and
$\tr(\mathsf E_w^2)=\tr(\mathsf T_w^2)=|\mathsf T_w|^2$.
\end{proof}

Observe that when $p=2$, $\mathsf A_w=\Id$ and $\mathsf T_w=\mathsf E_w$. 
Furthermore, note that the monotone defect can have a singular Cantor derivative; compare with the background chain-rule theory in \cite{AD1990}.  Therefore in the argument below we use the monotone smoothing and the coarea formula from \cite{EG2015}.

\begin{lem}\label{lem:BV-composition}
If $U'\subset\subset U\subset\subset\Omega_{\mathrm{reg}}$, then
\begin{equation}\label{eq:e-composition-BV}
\mathfrak e_\tau(w)\in BV(U')
\end{equation}
and, as Radon measures on $U'$,
\begin{equation}\label{eq:positive-pairing}
\Xbf_w\cdot D(\mathfrak e_\tau(w))\ge0.
\end{equation}
\end{lem}

\begin{proof}
As $U\subset\subset\Omega_{\mathrm{reg}}$, there are constants $0<c_*\le|\nabla w|\le C_*$ on $\overline U$.  The range of $w$ on $\overline U$ is a compact interval $I=[m,M]\subset[1,\infty)$.  Extend $\mathfrak e_\tau|_I$ monotonically to a slightly larger interval and convolve with one-dimensional nonnegative mollifiers.  This gives smooth nondecreasing $e_k$ such that
$$
e_k\to\mathfrak e_\tau\quad\text{uniformly on }I,
\qquad
\int_Ie_k'(t)\dd t
\le \mathfrak e_\tau(M)-\mathfrak e_\tau(m)+1.
$$

A finite collection of $C^1$ submersion charts covers
$\overline{U'}$. We employ $w$ as one of the coordinate functions in each chart, since $\nabla w$ does not vanish on $\overline{U}$, which implies that $w$ is a $C^{1}$ submersion. Consequently, the level sets of $w$ can be represented as graphs whose associated area elements are uniformly bounded. Thus 
$$
\sup_{t\in I} \mathcal H^{n-1}(U'\cap\{w=t\})<\infty,
$$
and the coarea formula gives
\begin{align*}
 |D(e_k(w))|(U')
 &=\int_{U'}e_k'(w)|\nabla w|\dd x\\
 &=\int_Ie_k'(t) \mathcal H^{n-1}(U'\cap\{w=t\})\dd t\le C.
\end{align*}
Since the uniform convergence of $e_k \to \mathfrak e_\tau$ gives $e_k(w)\to\mathfrak e_\tau(w)$ in $L^1(U')$, the  $BV$ compactness proves \eqref{eq:e-composition-BV} and, after passing to a subsequence,
$$
D(e_k(w))\weakto D(\mathfrak e_\tau(w)) \quad\text{weakly as vector measures}.
$$
For each $k$,
$$
 \Xbf_w\cdot D(e_k(w))  =e_k'(w)\Xbf_w\cdot\nabla w\dd x
 =e_k'(w)|\nabla w|^p\dd x\ge 0.
$$
Now for every nonnegative $\varphi\in C_c(U')$,
\begin{align*}
 \int_{U'}\varphi\,\Xbf_w\cdot
 D(\mathfrak e_\tau(w))
 &=\lim_{k\to\infty}
 \int_{U'}\varphi\,\Xbf_w\cdot D(e_k(w))\\
 &=\lim_{k\to\infty}
 \int_{U'}\varphi e_k'(w)|\nabla w|^p\,dx
 \ge 0.
\end{align*}
Here we used the fact that $\varphi \Xbf_w$ is continuous, that the measures $D(e_k(w))$ have uniformly bounded total variation, and that they converge weakly. This proves \eqref{eq:positive-pairing}.
\end{proof}

Define, on the regular set,
\begin{equation}\label{eq:defn-Lw}
\Lscr_w\zeta :=\operatorname{div}\bigl(w^{2-n}\mathsf A_w\nabla\zeta\bigr).
\end{equation}
The weight $w^{2-n}$ is chosen in the spirit of Ma and Wu
\cite{MW2024}; see the proof of Proposition~\ref{prop:bochner} below.

Recall the convention: For a vector field $X$ and a (not necessarily symmetric) matrix $E=DX-\frac {P} n \Id$, one has
$(DX)_{ij}=\partial_jX_i$ so that
$$E_{ij} \partial_iX_j = \tr(E \, DX).$$
Now we prove the following weak Bochner identity. 

\begin{prop}\label{prop:bochner}
On $\Omega_{\mathrm{reg}}$, in the sense of Radon measures,
\begin{equation}\label{eq:bochner}
 \Lscr_wQ_w  = n(p-1)w^{1-n}\tr(\mathsf E_w^2)\dd x +(n-1)(p-1)w^{1-n} \Xbf_w\cdot D(\mathfrak e_\tau(w)).
\end{equation}
In particular, $\Lscr_w Q_w$ is a nonnegative Radon measure.
\end{prop}

\begin{proof}
Fix
$$
U'\subset\subset U''\subset\subset U \subset\subset\Omega_{\mathrm{reg}}.
$$
Lemma~\ref{lem:Q-derivative} gives
$Q_w\in W^{1,2}(U)$, and
Lemma~\ref{lem:BV-composition}, applied to
$U''\subset \subset U$, gives
$$
    \mathfrak e_\tau(w)\in BV(U'').
$$
Since $U''$ is bounded, $W^{1,2}(U'')\subset W^{1,1}(U'')$, and therefore
$$
    P_w=Q_w+\mathfrak e_\tau(w)\in BV(U'').
$$
Moreover, on $U''$, 
\begin{equation}\label{eq:DP-decomposition}
DP_w = \nabla Q_w\,dx+D(\mathfrak e_\tau(w)).
\end{equation}

We first prove
\begin{equation}\label{eq:div-EX}
 \operatorname{div}(\mathsf E_w\Xbf_w)
 =\tr(\mathsf E_w^2)\dd x  +\frac{n-1}{n}\Xbf_w\cdot DP_w.
\end{equation}
For a smooth vector field $X$, let
$P=\operatorname{div}X$ and $E=DX-(P/n)\Id$.  With Einstein convention,
$$
\partial_iE_{ij} =\partial_i\partial_jX_i-\frac1n\partial_jP
=\frac{n-1}{n}\partial_j P.
$$
Furthermore,
$$
 E_{ij}\partial_iX_j =\tr(E\,DX) =\tr(E^2),
$$
as $DX=E+(P/n)\Id$ and $\tr E=0$.  Hence
$$\operatorname{div}(EX)=\tr(E^2)+\frac{n-1}n X\cdot\nabla P. $$

Recall $U''$ with $U'\subset\subset U''\subset\subset U$ and mollify $\Xbf_w$ inside $U$:
$$
\Xbf_w^{(k)}:=\rho_k*\Xbf_w,
\qquad P_w^{(k)}:=\operatorname{div}\Xbf_w^{(k)},
\qquad \mathsf E_w^{(k)}:=D\Xbf_w^{(k)}-\frac{P_w^{(k)}}n\Id.
$$
Then, for all sufficiently large $k$,
$$
P_w^{(k)}=\operatorname{div}\Xbf_w^{(k)} =\rho_k*P_w \qquad\text{on }U''.
$$
and
$$
 D\Xbf_w^{(k)}\to D\Xbf_w\quad\text{in }L^2(U''),
 \qquad  \Xbf_w^{(k)}\to\Xbf_w\quad\text{uniformly on }U''.
$$
Since $P_w\in BV(U'')$, $DP_w^{(k)}\weakto DP_w$ weakly  on $U'$ and $|DP_w^{(k)}|(U')$ is uniformly bounded.  Then for every test function $\eta\in C_c^1(U')$,
\begin{multline*}
 \left| \int\eta\Xbf_w^{(k)}\cdot DP_w^{(k)} -\int\eta\Xbf_w\cdot DP_w \right|\\
\le  \|\eta\|_\infty \|\Xbf_w^{(k)}-\Xbf_w\|_{L^\infty(U')} |DP_w^{(k)}|(U') +  \left| \int \eta\Xbf_w\cdot DP_w^{(k)}  -\int \eta\Xbf_w\cdot DP_w \right| \to 0.
\end{multline*}
The left-hand divergence term in \eqref{eq:div-EX} also converges.  Indeed, for every $\eta\in C_c^1(U')$,
\begin{multline*}
\|\eta(\mathsf E_w^{(k)}\Xbf_w^{(k)}-\mathsf E_w\Xbf_w)\|_{L^1(U')}\\
\le  \|\eta\|_\infty
\|\mathsf E_w^{(k)}-\mathsf E_w\|_{L^2(U')}
\|\Xbf_w^{(k)}\|_{L^2(U')}+
\|\eta\|_\infty \|\mathsf E_w\|_{L^2(U')}\|\Xbf_w^{(k)}-\Xbf_w\|_{L^2(U')} \to 0.
\end{multline*}
Moreover, the first term on the right-hand side of  \eqref{eq:div-EX} converges as well, since
\begin{multline*}
\left\| \tr\bigl((\mathsf E_w^{(k)})^2\bigr)-\tr(\mathsf E_w^2) \right\|_{L^1(U')} 
\le C\|\mathsf E_w^{(k)}-\mathsf E_w\|_{L^2(U')}
\left( \|\mathsf E_w^{(k)}\|_{L^2(U')}+\|\mathsf E_w\|_{L^2(U')}\right) \to 0.
\end{multline*}
Passing to the limit in the smooth identity proves \eqref{eq:div-EX}.

By Lemma~\ref{lem:stress-relation},
\begin{equation}\label{eq:LQ-as-divEX}
w^{2-n}\mathsf A_w\nabla Q_w =n(p-1)w^{1-n}\mathsf E_w\Xbf_w.
\end{equation}
Also,
$$
 \nabla w\cdot\mathsf E_w\Xbf_w
 =\frac{w}{n(p-1)}\nabla w\cdot\mathsf A_w\nabla Q_w
 =\frac wn\Xbf_w\cdot\nabla Q_w,
$$
since $\mathsf A_w\nabla w=(p-1)\Xbf_w$.  Now taking divergence in \eqref{eq:LQ-as-divEX}, using \eqref{eq:div-EX}, and differentiating the weight gives
\begin{align*}
 &\frac1{n(p-1)}\Lscr_wQ_w = \frac{1}{n(p-1)}{\rm div}(w^{2-n}\mathsf A_w\nabla Q_w ) \\
 ={}&w^{1-n}\tr(\mathsf E_w^2)\dd x
 +\frac{n-1}{n}w^{1-n}\Xbf_w\cdot DP_w -\frac{n-1}{n}w^{1-n} \Xbf_w\cdot\nabla Q_w\dd x.
\end{align*}
Plugging  \eqref{eq:DP-decomposition} into the formula above cancels the absolutely continuous $\nabla Q_w$-terms and proves \eqref{eq:bochner}; recall \cite{MW2024}.  In addition, the positivity of  $\Lscr_w Q_w$ follows from \eqref{eq:E-square-positive} and \eqref{eq:positive-pairing}.
\end{proof}

\begin{rem} \label{rem:ou-bochner-comparison}
In the pure-power case $c_\tau \equiv c_{\tau,0}$, we have
$$
\mathfrak e_\tau \equiv 0, \qquad Q_w = P_w.
$$
Then the first-order identity \eqref{eq:Ou-form} reduces to
$$
\mathsf E_w^{\,T}\nabla w = \frac{w}{n}\nabla P_w,
$$
and the Bochner identity above becomes 
$$
\Lscr_w P_w = n(p-1) w^{1-n} \tr(\mathsf E_w^2)\,dx; 
$$
see also \cite[Lemma 2.2, Formula (2.12)]{O2025}.
Moreover, the identity \eqref{eq:div-EX} is an analogue of the fundamental transformed-tensor identity \cite[Lemma 2.1 (iii)]{O2025} employed by Ou. 
This also shows that the choice of the weight $w^{1-n}$ is dictated by an exact cancellation mechanism introduced by Ma--Wu \cite{MW2024}: Differentiation of $w^{1-n}$ produces a term that cancels the corresponding contribution from $\Xbf_w \cdot \nabla P_w$.

For a general monotone coefficient, we have
$$
P_w = Q_w + \mathfrak e_\tau(w).
$$
The same cancellation mechanism eliminates the absolutely continuous component of $\nabla Q_w$, while the residual singular (Stieltjes) contribution
$$
\Xbf_w \cdot D(\mathfrak e_\tau(w))
$$
exhibits a favorable sign. This is precisely the mechanism by which the invariant-tensor identity, initially derived in the pure-power framework, is generalized to the broader class of functions analyzed in the present manuscript.
\end{rem}

We next establish the following nonlinear Kato inequality; the author was subsequently informed that this result had in fact been obtained earlier by Sun and Wang in \cite[Lemma 2.1 (3)]{SW2025}.

\begin{lem} \label{lem:rayleigh}
Almost everywhere on the regular set $\Omega_{\mathrm{reg}}$,
\begin{equation}\label{eq:rayleigh}
 \tr(\mathsf E_w^2) \ge  \frac{w^2(\Xbf_w\cdot\nabla Q_w)^2}  {n(n-1)|\nabla w|^{2p}}.
\end{equation}
\end{lem}

\begin{proof}
Let $y:=\mathsf A_w^{-1/2}\Xbf_w$. Since we work on $\Omega_{\mathrm{reg}}$, one has
$\mathbf X_w\ne 0$, and hence $y\ne0$.
Let $T$ be a symmetric and trace-free matrix. We choose an orthonormal basis with first vector ${\mathbf e}_1=y/|y|$ and write $t_i={\mathbf e}_i\cdot T{\mathbf e}_i$.  Then the off-diagonal entries contribute nonnegatively to $|T|^2$, while $\sum_{i=1}^nt_i=0$.  Therefore,  
$$
 |T|^2\ge \sum_{i=1}^nt_i^2  \ge t_1^2+\frac1{n-1}\left(\sum_{i=2}^nt_i\right)^2  =\frac n{n-1}t_1^2.
$$
Since $t_1=y\cdot Ty/|y|^2$, this yields
\begin{equation}\label{eq:tracefree-linear-algebra}
|T|^2\ge\frac n{n-1} \left(\frac{y\cdot Ty}{|y|^2}\right)^2.
\end{equation}

Now we apply this to $T=\mathsf T_w$, which has been defined in \eqref{eq:defn-Tw}.  
Namely,  \eqref{eq:tracefree-linear-algebra} and
\eqref{eq:E-square-positive} give
\begin{equation}\label{eq:traceEw}
    \tr (\mathsf E_w^2) = |\mathsf T_w|^2\ge \frac n{n-1} \left(\frac{y\cdot \mathsf T_wy}{|y|^2}\right)^2. 
\end{equation}
Moreover, by \eqref{eq:stress-relation},
$$
 y\cdot\mathsf T_wy =\Xbf_w\cdot\mathsf A_w^{-1}\mathsf E_w\Xbf_w  =\frac{w}{n(p-1)}\Xbf_w\cdot\nabla Q_w.
$$
Also,
$$
 |y|^2  =\Xbf_w\cdot\mathsf A_w^{-1}\Xbf_w
 =\frac{|\nabla w|^p}{p-1}.
$$
Consequently, applying these substitutions in \eqref{eq:traceEw} we obtain \eqref{eq:rayleigh}.
\end{proof}

\begin{lem}
There is $C=C(n,p,L,\Lambda)$ such that
\begin{equation}\label{eq:gradient-w}
|\nabla w|\le Cw  \quad\text{in }\R^n,
\end{equation}
and
\begin{equation}\label{eq:Q-growth}
|Q_w|\le C(1+w^{p-1}).
\end{equation}
\end{lem}

\begin{proof}
By \eqref{eq:h-assumptions}, the equation for $V$ can be written
$$
-\Delta_pV=b_\tau(x)V^{p-1},\qquad    b_\tau(x):=\frac{g_\tau(V(x))}{V(x)^{p-1}}.
$$
Since $0<V\le 1$ and $q-p+1>0$, and
$g_\tau(s)\le\Lambda s^q$, one has
$$
0<b_\tau(x) \le\Lambda V(x)^{q-p+1} \le\Lambda.
$$
Then Lemma~\ref{lem:relative-gradient}, applied in every unit ball, gives $|\nabla V|\le C(n,p,\Lambda)V$.  Since $w=V^{-1/a}$,
$$
|\nabla w| =\frac1aV^{-1/a-1}|\nabla V|\le CV^{- 1/a}\le Cw.
$$
This proves \eqref{eq:gradient-w}.

Moreover, the function $\mathfrak k_\tau$ is bounded on $[1,\infty)$ since $c_\tau$ is bounded and $\int_1^\infty s^{-2}\dd s<\infty$.  Formula \eqref{eq:Q-algebraic} now gives \eqref{eq:Q-growth}.
\end{proof}

\subsection{A global bound for \texorpdfstring{$Q_w$}{Qw}}
The weak Bochner identity shows that $Q_w$ is subharmonic with
respect to the linearized weighted operator on the regular set. This alone does not prevent a positive maximum from escaping to infinity.  We therefore fix a target level
$\ell>c_{\tau,0}$ and construct a decreasing weight
$\phi_{\ell,\mathfrak m}$ below such that
$$
\mathscr M_{\ell,\mathfrak m}:=Q_w\phi_{\ell,\mathfrak m}(w)
$$
satisfies, on its superlevel set
$\{\mathscr M_{\ell,\mathfrak m}>\ell\}$, a uniformly elliptic
differential inequality with a strictly positive zeroth-order
term.  The parameter $\mathfrak m$ will later be sent to infinity so that $\phi_{\ell,\mathfrak m}\to 1$ on every bounded $w$-range.

To be more precise, define
\begin{equation}\label{eq:defn-Jl}
\ell>c_{\tau,0}, \qquad
\mathfrak m>0, \qquad J_\ell(t):=\ell-\mathfrak k_\tau(t).
\end{equation}
Since $\mathfrak k_\tau$ is nonincreasing and
$\mathfrak k_\tau(1)=c_{\tau,0}$,
$$
J_\ell(t)\ge\ell-c_{\tau,0}>0
\quad \text{ for every } \ t\ge 1.
$$
Let $Y_{\ell,\mathfrak m}$ solve
\begin{equation}\label{eq:Y-ODE}
 Y' +\left(\frac a t+\frac{b_p \mathfrak k_\tau'}{J_\ell}\right)Y  =\frac{b_p\mathfrak e_\tau}{tJ_\ell},
 \qquad  Y(1)=\mathfrak m.
\end{equation}

We claim that, for every $t\ge1$,
\begin{equation}\label{eq:Y-representation}
 Y_{\ell,\mathfrak m}(t) = t^{-a}J_\ell(t)^{b_p}
 \Bigg[ \mathfrak m(\ell-c_{\tau,0})^{-b_p} +b_p\int_1^t
 s^{a-1}J_\ell(s)^{-b_p-1}\mathfrak e_\tau(s)\dd s \Bigg].
\end{equation}
Indeed,  since $J_\ell'=-\mathfrak k_\tau'$ by the definition \eqref{eq:defn-Jl}, then
$$
 \frac{\dd}{\dd t}\log(t^aJ_\ell^{-b_p})
 =\frac at+\frac{b_p\mathfrak k_\tau'}{J_\ell}.
$$
Applying this identity, multiplying \eqref{eq:Y-ODE} by $t^aJ_\ell^{-b_p}$ and integrating from $1$ to $t$ , we obtain  \eqref{eq:Y-representation}.  Note that the initial term is strictly positive and the integral term is nonnegative. In particular, $Y_{\ell,\mathfrak m}>0$.

Define
\begin{equation}\label{eq:defn-thetalm}
\begin{aligned}
 \theta_{\ell,\mathfrak m}(t)  &:=\frac{n-1}{1+Y_{\ell,\mathfrak m}(t)},\\
 \mathfrak s_{\ell,\mathfrak m}(t) &:=\theta_{\ell,\mathfrak m}(t)  \left(1-\frac{\theta_{\ell,\mathfrak m}(t)}{n-1}\right),\\
 \phi_{\ell,\mathfrak m}(t) &:=\exp\left(-\int_1^t \frac{\theta_{\ell,\mathfrak m}(s)}s\dd s\right).
\end{aligned}
\end{equation}
Thus, we have $0<\theta<n-1$, $\mathfrak s>0$, $0<\phi\le1$, and
\begin{equation}\label{eq:phi-derivative}
\frac{t\phi'_{\ell,\mathfrak m}}{\phi_{\ell,\mathfrak m}}=-\theta_{\ell,\mathfrak m}.
\end{equation}
Let
\begin{equation}\label{eq:defn-alm}
\mathfrak a_{\ell,\mathfrak m}(t) :=\frac{a\mathfrak s_{\ell,\mathfrak m}(t)  -t\theta_{\ell,\mathfrak m}'(t)}{b_p}.
\end{equation}

\begin{lem} 
For every $t\ge1$,
\begin{equation}\label{eq:a-cancellation-form}
 \mathfrak a_{\ell,\mathfrak m}(t)
 =\frac{ \mathfrak s_{\ell,\mathfrak m}(t)\mathfrak k_\tau(t)
 +\theta_{\ell,\mathfrak m}(t)\mathfrak e_\tau(t)}
 {J_\ell(t)}.
\end{equation}
Namely,
\begin{equation}\label{eq:key-cancellation}
 \mathfrak a_{\ell,\mathfrak m}(t)
 \bigl(\ell-\mathfrak k_\tau(t)\bigr)
 =\mathfrak s_{\ell,\mathfrak m}(t)\mathfrak k_\tau(t)
  +\theta_{\ell,\mathfrak m}(t)\mathfrak e_\tau(t).
\end{equation}
Moreover,
\begin{equation}\label{eq:a-positive}
 \mathfrak a_{\ell,\mathfrak m}(t)
 =\frac{\theta_{\ell,\mathfrak m}(t)}
 {(1+Y_{\ell,\mathfrak m}(t))J_\ell(t)}
 \left( \mathfrak e_\tau(t) +Y_{\ell,\mathfrak m}(t)\frac{c_\tau(t)}t \right)>0.
\end{equation}
\end{lem}

\begin{proof}
Write $Y$, $\theta$, and $\mathfrak s$ for brevity.  From
$\theta=(n-1)/(1+Y)$ in \eqref{eq:defn-thetalm}, it follows that
$$
\mathfrak s=\frac{\theta Y}{1+Y},
\qquad \text{ and } \quad -t\theta'=\frac{\theta tY'}{1+Y}.
$$
Consequently, \eqref{eq:defn-alm} yields
$$
\mathfrak a_{\ell,\mathfrak m} =\frac{\theta}{b_p(1+Y)}(aY+tY').
$$
Since $t\mathfrak k_\tau'(t)=-c_\tau(t)/t$ by \eqref{eq:defect-functions}, the equation
\eqref{eq:Y-ODE} is equivalent to
$$
aY+tY' =\frac{b_p}{J_\ell(t)}
  \left(Y\frac{c_\tau(t)}t+\mathfrak e_\tau\right).
$$
This proves \eqref{eq:a-positive}, including the claimed strict positivity.  

To obtain \eqref{eq:a-cancellation-form}, one expands \eqref{eq:a-positive} and uses \eqref{eq:defn-thetalm} together with
$$
\mathfrak k_\tau(t)+\mathfrak e_\tau(t) =\frac{c_\tau(t)}t.
$$
We therefore conclude the lemma. 
\end{proof}

\begin{rem}\label{rem:target-multiplier-design}
The ODE \eqref{eq:Y-ODE} is designed backwards from the
 inequality \eqref{eq:completed-square} below. For simplicity, we omit the subscripts $\ell$ and $\mathfrak m$ here. Let
$$\mathscr M=Q_w\phi(w), \qquad -\frac{t\phi'(t)}{\phi(t)}=\theta(t).
$$
Then the expansion of $\mathscr L_w\mathscr M$, after a suitable drift is subtracted, leaves the zeroth-order expression
$$
 Q_w\Bigl[ \mathfrak a(Q_w-\mathfrak k_\tau) -\mathfrak s\mathfrak k_\tau  -\theta\mathfrak e_\tau \Bigr].
$$
We therefore impose
$$
 \mathfrak a(\ell-\mathfrak k_\tau)
 = \mathfrak s\mathfrak k_\tau+\theta\mathfrak e_\tau,
$$
which turns the bracket into
\(\mathfrak a(Q_w-\ell)\).
Writing
\(\theta=(n-1)/(1+Y)\) reduces this cancellation condition to the
linear ODE \eqref{eq:Y-ODE}.

This indeed follows the general strategy of inserting a scalar weight to cancel tensor-gradient cross terms.  In the pure-power setting, fixed weights $v^qg^m$ appear in
\cite[Lemma~2.2,  (2.12), and Proposition~2.3,
 (2.15)]{O2025}; in the semilinear invariant-tensor
setting, a power weight is chosen for the same purpose in e.g.
\cite[Section~2.3, (2.8)]{MW2024}.
The present target-dependent ODE is needed since
$\mathfrak k_\tau$ and $\mathfrak e_\tau$ are nonconstant.
\end{rem}

Define
\begin{equation}\label{eq:defn-Mlm}
\Mscr_{\ell,\mathfrak m} :=Q_w\phi_{\ell,\mathfrak m}(w).
\end{equation}
The  formula \eqref{eq:Q-algebraic} shows that $Q_w$
is continuous.  Hence $\mathscr M_{\ell,\mathfrak m}$ is
continuous and its super-level set  is open.

Observe that, on the super-level set
$$
\mathcal O_{\ell,\mathfrak m}:=\{\Mscr_{\ell,\mathfrak m}>\ell\},
$$
one has $Q_w>\ell$ since $0<\phi\le1$. Since $w\ge1$ by \eqref{eq:defn-w}, \eqref{eq:Q-algebraic}, together with
$\mathfrak k_\tau\le c_{\tau,0}$, gives
\begin{equation}\label{eq:gradient-lower-superlevel}
|\nabla w|^p =\frac{w}{b_p}(Q_w-\mathfrak k_\tau(w))
\ge\frac{\ell-c_{\tau,0}}{b_p}>0.
\end{equation}
Thus $\Lscr_w$ defined in \eqref{eq:defn-Lw} is locally uniformly elliptic on this set.

Set
\begin{equation}\label{eq:defn-zeta} 
\begin{aligned}
 \zeta_{\ell,\mathfrak m}
 &:=\frac{w} {\phi_{\ell,\mathfrak m}(w)|\nabla w|^p}
 \Xbf_w\cdot\nabla\Mscr_{\ell,\mathfrak m},\\
 \mathbf b_{\ell,\mathfrak m}
 &:=2(p-1)\theta_{\ell,\mathfrak m}(w)w^{1-n}
 \left(\frac{wQ_w}{(n-1)|\nabla w|^p}-1  \right)\Xbf_w.
\end{aligned}
\end{equation}

\begin{prop}\label{prop:bochner-lower}
On $\mathcal O_{\ell,\mathfrak m}$, in the sense of Radon measures,
\begin{multline}
\label{eq:completed-square}
\Lscr_w\Mscr_{\ell,\mathfrak m}
 -\mathbf b_{\ell,\mathfrak m}\cdot
  \nabla\Mscr_{\ell,\mathfrak m}  
\ge (p-1)\phi_{\ell,\mathfrak m}(w)w^{1-n}
 \left[  \frac{\zeta_{\ell,\mathfrak m}^2}{n-1} +\mathfrak a_{\ell,\mathfrak m}(w) Q_w(Q_w-\ell) \right]\dd x\\
 +(n-1)(p-1)\phi_{\ell,\mathfrak m}(w)w^{1-n}
 \Xbf_w\cdot D(\mathfrak e_\tau(w)).
\end{multline}
\end{prop}

\begin{proof}
We suppress the subscripts and write $\theta$, $\mathfrak s$, $\mathfrak a$, $\phi$, $\Mscr$, and $\zeta$ in the argument.   

Fix $U'\subset\subset U\subset\subset\mathcal O_{\ell,\mathfrak m}$.  On $U$, the gradient of $w$ is bounded away from zero by
\eqref{eq:gradient-lower-superlevel}.  Since $U\subset\subset
\mathcal O_{\ell,\mathfrak m}$, the functions $w$,
$|\nabla w|$, and $|\nabla w|^{-1}$ are bounded on $U$.
Consequently, $w^{2-n}\mathsf A_w$ is bounded and uniformly
elliptic there, and particularly all coefficients below are bounded. Moreover,  $Q_w\in W^{1,2}(U)$ by Lemma~\ref{lem:Q-derivative}, and the vector field
$w^{2-n}\mathsf A_w\nabla Q_w$ has Radon-measure divergence by
Proposition~\ref{prop:bochner}.

Note that, the distributional product rule
$\operatorname{div}(\chi F) =\chi\operatorname{div}F+\nabla\chi\cdot F$
is valid for $\chi\in C^1(U)$ and such a divergence-measure field $F$.  Indeed, for $\eta\in C_c^\infty(U)$, the definition of distributional divergence gives
\begin{equation}\label{eq:measure-product}
\langle\operatorname{div}(\chi F),\eta\rangle
 =-\int_U\chi F\cdot\nabla\eta\,\dd x  =\langle\operatorname{div}F,\chi\eta\rangle    +\int_U\eta\nabla\chi\cdot F\,\dd x,
\end{equation}
which is the desired product rule. 
The remaining product rules involve only $W^{1,2}$ functions and the
$W^{1,2}$ stress.  Thus the calculation below is valid on $U'$ as an identity of measures, and since $U'$ is arbitrary, it is valid on the whole super-level set.

From \eqref{eq:phi-derivative} and \eqref{eq:defn-Mlm},
\begin{equation}\label{eq:nablaM}
\nabla\Mscr =\nabla Q_w \phi(w)-\frac{\theta(w) \phi(w)}{w}  Q_w \nabla w 
=\phi(w)\left(\nabla Q_w-\frac{\theta(w)Q_w}{w}\nabla w \right).
\end{equation}
Taking the inner product with $\Xbf_w$ and applying \eqref{eq:defn-zeta}  gives
\begin{equation}\label{eq:zeta-relation}
\frac{w}{|\nabla w|^p}\Xbf_w\cdot\nabla Q_w =\zeta+\theta Q_w,
\end{equation}
where we applied $\Xbf_w\cdot Dw=|Dw|^p$.

Since $\mathsf A_w\nabla w=(p-1)\Xbf_w$, the product rule gives
\begin{align}\label{eq:LM-first-expansion}
{}&\Lscr_w\Mscr
=\operatorname{div}\!\left(
 \phi(w)w^{2-n}\mathsf A_w\nabla Q_w
\right) +(p-1)\operatorname{div}\!\left(
 Q_w\phi'(w)w^{2-n}\Xbf_w\right)\notag\\
={}&\phi(w)\Lscr_wQ_w+2(p-1)w^{2-n}\phi'(w) \Xbf_w\cdot\nabla Q_w  +(p-1)Q_w\operatorname{div}\!\left( w^{2-n}\phi'(w)\Xbf_w\right). 
\end{align}
Observe that the product rule \eqref{eq:measure-product} is applicable to the first divergence term, since $\phi(w)$ is locally of class $C^1$. For the second divergence term, note that on $U$ we have
$$
 Q_w\phi'(w)w^{2-n}\Xbf_w\in W^{1,1}(U;\R^n).
$$
Indeed, $Q_w,\Xbf_w\in W^{1,2}(U)\cap L^\infty(U)$, whereas the scalar factors depending on $w$ are bounded and have bounded first derivatives as well. Consequently, the distributional divergence of this vector field belongs to $L^1(U)$. Therefore, we only need to apply the measure-valued product rule \eqref{eq:measure-product} for the first term, namely the one involving $w^{2-n}\mathsf A_w\nabla Q_w$.

Since $\phi'=-\theta\phi/w$,
\begin{equation}\label{eq:phi'}
 w^{2-n}\phi'=-\theta\phi w^{1-n}.
\end{equation}
Using \eqref{eq:phi'} twice gives
$$
\frac{\dd}{\dd w}(w^{2-n}\phi')
=\frac{\dd}{\dd w}(-\theta\phi w^{1-n})
=\phi w^{-n}\bigl((n-1)\theta+\theta^2-w\theta'\bigr).
$$
Hence, employing $\Xbf_w\cdot\nabla w=|\nabla w|^p$
and \eqref{eq:phi'} again,
\begin{equation}\label{eq:div-weighted-phiX}
 \operatorname{div}(w^{2-n}\phi'\Xbf_w)
 =\phi w^{-n}\left[((n-1)\theta+\theta^2-w\theta')|Dw|^p
   -\theta wP_w \right].
\end{equation}

Now Proposition~\ref{prop:bochner} and Lemma~\ref{lem:rayleigh} give
\begin{equation}\label{eq:LQ-rayleigh}
 \Lscr_wQ_w \ge \frac{p-1}{n-1}w^{3-n}
\frac{(\Xbf_w\cdot\nabla Q_w)^2}{|\nabla w|^{2p}}\dd x +(n-1)(p-1)w^{1-n} \Xbf_w\cdot D(\mathfrak e_\tau(w)).
\end{equation} 
Substituting \eqref{eq:div-weighted-phiX} and
\eqref{eq:LQ-rayleigh} into \eqref{eq:LM-first-expansion}, applying \eqref{eq:defn-Qw} together with \eqref{eq:Q-algebraic} to get
\begin{equation}\label{eq:PwQw}
\frac{|\nabla w|^p}{w}=\frac{Q_w-\mathfrak k_\tau(w)}{b_p},
\qquad P_w=Q_w+\mathfrak e_\tau(w),
\end{equation}
and further employing \eqref{eq:zeta-relation},  we conclude the absolutely continuous part of \eqref{eq:LM-first-expansion} has the following lower bound
\begin{align}
 (p-1)\phi w^{1-n}&{}\left( \frac{(\zeta+\theta Q_w)^2}{n-1}
 -\frac{2\theta(Q_w-\mathfrak k_\tau)}{b_p}
  (\zeta+\theta Q_w)\right.\notag\\
 &\left. \qquad\quad +Q_w\left[
 \frac{(n-1)\theta+\theta^2-w\theta'}{b_p}
 (Q_w-\mathfrak k_\tau) -\theta(Q_w+\mathfrak e_\tau)
 \right]\right).\label{eq:absolute-LQ}
\end{align}
Moreover, by the definition of $\zeta$ in \eqref{eq:defn-zeta},
$$
\Xbf_w\cdot\nabla\Mscr =\frac{\phi(w)|\nabla w|^p}{w}\,\zeta.
$$
Consequently, using \eqref{eq:PwQw} and the definition of $\mathbf b$ in \eqref{eq:defn-zeta}, we arrive at
\begin{equation}\label{eq:drift-expansion}
\mathbf b\cdot\nabla\Mscr= 2(p-1)\phi w^{1-n}\theta
\left( \frac{Q_w}{n-1} -\frac{Q_w-\mathfrak k_\tau}{b_p}\right)\zeta.
\end{equation}

Subtracting \eqref{eq:drift-expansion} from \eqref{eq:absolute-LQ} gives an exact  cancellation of the terms which are linear in $\zeta$, as their coefficient are exactly
$$
 \frac{2\theta Q_w}{n-1}  -\frac{2\theta(Q_w-\mathfrak k_\tau)}{b_p}
 -2\theta\left( \frac{Q_w}{n-1} -\frac{Q_w-\mathfrak k_\tau}{b_p}\right)=0.
$$
Furthermore, note that the $\zeta^2$-term remains, and the remaining non-$\zeta$ part is
\begin{align*}
&(p-1) \phi w^{1-n}\mathcal R\\
:={}&
(p-1) \phi w^{1-n}\left( \frac{\theta^2Q_w^2}{n-1}
 -\frac{2\theta^2Q_w(Q_w-\mathfrak k_\tau)}{b_p}\right.\\
&\left. \qquad \qquad \qquad \qquad +Q_w\left[
 \frac{(n-1)\theta+\theta^2-w\theta'}{b_p}(Q_w-\mathfrak k_\tau) -\theta(Q_w+\mathfrak e_\tau)
 \right]\right).
\end{align*}

Recall from \eqref{eq:defn-thetalm} and \eqref{eq:defn-alm} that 
$$\mathfrak s=\theta\left(1-\frac\theta {n-1}\right) \quad \text{ and } \quad \mathfrak a(w)= \frac{ a \mathfrak s- w\theta'}{b_p}. $$
Then since $Q_w>\ell>0$ on
$\mathcal O_{\ell,\mathfrak m}$, we may factor out $Q_w$
from the remaining zeroth-order term. Then the  coefficient of $Q_w$ is
$$
 \left(\frac{\theta^2}{n-1}-\theta\right)  +\frac{(n-1)\theta-\theta^2-w\theta'}{b_p} 
=-\mathfrak s+\frac{(n-1)\mathfrak s-w\theta'}{b_p}
=\frac{a\mathfrak s-w\theta'}{b_p}   =\mathfrak a,
$$
and the coefficient of $\mathfrak k_\tau$ is
\begin{align*}
\frac{2\theta^2}{b_p}  -\frac{(n-1)\theta+\theta^2-w\theta'}{b_p} 
=\frac{\theta^2-(n-1)\theta+w\theta'}{b_p}  =\frac{w\theta'-(n-1)\mathfrak s}{b_p} 
=-\mathfrak a-\mathfrak s.
\end{align*}
Therefore we conclude that
$$
 \frac{\mathcal R}{Q_w} =\mathfrak a(Q_w-\mathfrak k_\tau) -\mathfrak s\mathfrak k_\tau -\theta\mathfrak e_\tau.
$$
Thus the absolutely continuous part after the drift is
\begin{equation}\label{eq:after-drift}
\frac{\zeta^2}{n-1} + \mathcal R= \frac{\zeta^2}{n-1}
 +Q_w\left[ \mathfrak a(Q_w-\mathfrak k_\tau) 
-\mathfrak s\mathfrak k_\tau -\theta\mathfrak e_\tau
 \right].
\end{equation}
Finally, \eqref{eq:key-cancellation} changes the bracket in
\eqref{eq:after-drift} into $\mathfrak a(Q_w-\ell)$.  The singular measure in \eqref{eq:LQ-rayleigh} is multiplied by the positive continuous factor $\phi(w)$.  This proves \eqref{eq:completed-square}.
\end{proof}

Next we study the asymptotic behavior of the  multipliers.
\begin{lem}
The limit
$$
d_\tau:=\lim_{t\to\infty}\mathfrak e_\tau(t)
$$
exists in $[0,\infty)$.  Moreover, for any $\mathfrak m>0$
\begin{equation}\label{eq:k-e-limits}
\mathfrak k_\tau(t)\to-d_\tau \quad
\text{ and } \quad
Y_{\ell,\mathfrak m}(t)\to \frac{b_pd_\tau}{a(\ell+d_\tau)}
\quad \text{ as } \ t\to \infty. 
\end{equation}
Consequently,
\begin{equation}\label{eq:theta-infinity}
 \theta_{\ell,\mathfrak m}(t)\to
 \theta_{\ell,\infty}  :=\frac{n-1}{1+b_pd_\tau/[a(\ell+d_\tau)]}>0 \quad \text{ as } \ t\to \infty,
\end{equation}
and $\theta_{\ell,\infty}\to n-1$ as $\ell\to\infty$.
For every finite $T>1$,
\begin{equation}\label{eq:m-to-infinity}
\phi_{\ell,\mathfrak m}\to 1
\quad\text{uniformly on }[1,T] \quad \text{ as } \ \mathfrak m\to\infty.
\end{equation}
\end{lem}

\begin{proof}
Recall that, the function $\mathfrak e_\tau$ is nonnegative and nondecreasing.  It is bounded since $c_\tau(t)/t$ is bounded and $\mathfrak k_\tau(t)=c_{\tau,0}-\int_1^tc_\tau(s)s^{-2}\,\dd s$ is bounded uniformly when $t\ge1$.  Hence $d_\tau$ exists.  Moreover, since $c_\tau(t)/t\to 0$ as $t\to \infty$ and
$\mathfrak e_\tau=c_\tau/t-\mathfrak k_\tau$, the first limit in \eqref{eq:k-e-limits} follows as well.

Using \eqref{eq:Y-representation}, write
$$
Y_{\ell,\mathfrak m}(t)
=\mathfrak m(\ell-c_{\tau,0})^{-b_p}
t^{-a}J_\ell(t)^{b_p} + b_pJ_\ell(t)^{b_p}
\left[t^{-a}\int_1^t s^{a-1}J_\ell(s)^{-b_p-1}
\mathfrak e_\tau(s)\,\dd s \right].
$$
Since $J_\ell(t)\to\ell+d_\tau$ is bounded, the first term tends to zero due to the $t^{-a}$-term.
Further set
$$
r(s):=J_\ell(s)^{-b_p-1}\mathfrak e_\tau(s),
\qquad r_\infty:=(\ell+d_\tau)^{-b_p-1}d_\tau .
$$
Then $r(s)\to r_\infty$ by the definition of $J_\ell$ in \eqref{eq:defn-Jl}.  Thus for $1<T<t$,
\begin{align*}
&\ \left|t^{-a}\int_1^t s^{a-1}r(s)\,\dd s
-\frac{r_\infty}{a}\right|\\
=& \ \left| t^{-a}\int_1^t s^{a-1}\bigl(r(s)-r_\infty\bigr)\,\dd s -\frac{r_\infty}{a}t^{-a} \right|\\
\le& \ t^{-a}\int_1^T
s^{a-1}|r(s)-r_\infty|\,\dd s+
t^{-a}\int_T^t s^{a-1}|r(s)-r_\infty|\,\dd s +\frac{|r_\infty|}{a}t^{-a}.
\end{align*}
Given $\epsilon>0$, choose $T$ so that $|r(s)-r_\infty|\le\epsilon$ for $s\ge T$.  Thus by letting $t\to\infty$ and then $\epsilon\to 0$, the definition of $r(s)$ gives
$$
 t^{-a}\int_1^t  s^{a-1}J_\ell(s)^{-b_p-1}\mathfrak e_\tau(s)\dd s \to \frac{d_\tau}{a}(\ell+d_\tau)^{-b_p-1}.
$$
Consequently, multiplying the preceding limit by $b_p J_\ell(t)^{b_p}$ yields the second limit in \eqref{eq:k-e-limits}. Note that the formula remains valid when $d_\tau=0$ as well.  Now equation \eqref{eq:theta-infinity} follows from the definition of $\theta$ in \eqref{eq:defn-thetalm}.

For fixed $T$, since $(\ell-c_{\tau, 0})^{-b_p}>0$ by \eqref{eq:defn-Jl},  by \eqref{eq:Y-representation}  one has $Y_{\ell,\mathfrak m}\to\infty$ uniformly on $[1, T]$ as $\mathfrak m\to \infty$, hence
$\theta_{\ell,\mathfrak m}\to0$ uniformly.  Thus, integrating
$t\phi'_{\ell,\mathfrak m}/\phi_{\ell,\mathfrak m}=-\theta_{\ell,\mathfrak m}$ in \eqref{eq:phi-derivative} proves \eqref{eq:m-to-infinity}.
\end{proof}

Let us recall the following interior maximum principle. 

\begin{lem}
\label{lem:linear-weak-harnack-minimum}
Let $D\subset\mathbb R^n$ be connected. Let
$A\in L^\infty(D;\mathbb R^{n\times n})$ be symmetric and uniformly elliptic, and let $b\in L^\infty(D;\mathbb R^n)$. Suppose
$$
z\in W^{1,2}_{\rm loc}(D)\cap C(D), \qquad z\ge0,
$$
and
$$
-\operatorname{div}(A\nabla z)+b\cdot\nabla z\ge0
$$
weakly in $D$. If $z(x_0)=0$ at some $x_0\in D$, then
$z\equiv0$ in $D$.
\end{lem}

\begin{proof}
For every $B_{2r}\subset \subset D$, the weak Harnack inequality for
nonnegative supersolutions gives, for some $\theta>0$,
$$
\left(\fint_{B_r}z^\theta\right)^{1/\theta}
\le C\inf_{B_r}z;
$$
see, for example, \cite[Theorems~8.18 and~8.19]{GT2001}.
Thus, an interior zero  forces $z$ to vanish on a smaller ball, and then a chain of overlapping balls propagates this vanishing throughout the connected domain.
\end{proof}

\begin{lem} \label{lem:maximum-exclusion}
Let $D$ be a ball, let $A(x)$ be bounded and uniformly elliptic on $D$, and let $b\in L^\infty(D;\R^n)$.  Suppose
$z\in W^{1,2}_{\loc}(D)\cap C(D)$ satisfies
\begin{equation}\label{eq:strict-subsolution-maximum}
\operatorname{div}(A\nabla z)-b\cdot\nabla z
\ge c_0
\end{equation}
in the weak sense for some $c_0>0$.  Then $z$ cannot attain a maximum at an interior point of $D$.
\end{lem}

\begin{proof}
Suppose that $z$ has a local maximum at $x_0$. Choose
$B_{2\rho}(x_0)\subset \subset D$ such that $z\le z(x_0)$ there, and set
$u:=z(x_0)-z\ge0.$
Then $u(x_0)=0$ and
$$
-\operatorname{div}(A\nabla u)+b\cdot\nabla u\ge c_0>0
\quad\text{weakly in }B_{2\rho}(x_0).
$$
Now Lemma~\ref{lem:linear-weak-harnack-minimum} gives
$u\equiv0$ in $B_{2\rho}(x_0)$. This contradicts the strict positive right-hand side by testing with a nonnegative nonzero function in $C_c^\infty(B_{2\rho}(x_0))$.
\end{proof}

Now we are ready to prove the following sharp global bound. 
\begin{prop}\label{prop:Q-bound}
The transformed function satisfies
\begin{equation}\label{eq:Qw-upper} 
Q_w\le c_{\tau,0} \quad\text{in }\R^n.
\end{equation}
\end{prop}
\begin{rem}\label{rem:two-stage-Q-bound}
Let us explain the idea of the proof. 
For a sufficiently large target $\bar\ell$, the limiting decay
exponent  $\theta_{\bar\ell,\infty}$ exceeds $p-1$.  Hence, for every fixed $\mathfrak m>0$, the multiplier
$\phi_{\bar\ell,\mathfrak m}$ has asymptotic decay exponent
$\theta_{\bar\ell,\infty}>p-1$, and therefore decays faster than the a priori growth $Q_w=O(w^{p-1})$, which yields a coarse global bound $Q_w\le\bar\ell$ in   Step 2 below.

Once the coarse bound is available, any positive limiting exponent $\theta_{\ell,\infty}>0$ is enough to prevent the positive part of $Q_w\phi_{\ell,\mathfrak m}(w)$ from escaping to infinity.
For the sharp $\ell$, the parameter $\mathfrak m$ is chosen after identifying a hypothetical point where $Q_w > \ell$, ensuring that the  convergence $\phi_{\ell,\mathfrak m} \to 1$ pushes the weighted quantity above $\ell$ at that spot.

\end{rem}
\begin{proof}[Proof of Proposition~\ref{prop:Q-bound}]
We first choose a large target.  Observe that when $\ell\to \infty$, one has $\theta_{\ell,\infty}\to n-1$ according to \eqref{eq:theta-infinity} and  that  $p-1<n-1$. Then we take
$\bar\ell>c_{\tau,0}$ so that
$$
\theta_{\bar\ell,\infty}>p-1.
$$

\medskip
\noindent{\bf Step 1: A decay at $\infty$.}
Fix an arbitrary $\mathfrak m>0$.  From \eqref{eq:theta-infinity} and \eqref{eq:phi-derivative}, for every sufficiently small $\epsilon>0$ there is $T$ such that
\begin{equation}\label{eq:phi-decay}
\phi_{\bar\ell,\mathfrak m}(t)
\le C_\epsilon t^{-\theta_{\bar\ell,\infty}+\epsilon}
\quad\text{ for any } \ t\ge T.
\end{equation}
We choose $\epsilon$ so that the decay exponent remains greater than $p-1$.  Then the definition of $\Mscr_{\bar\ell,\mathfrak m}$ in \eqref{eq:defn-Mlm} together with the growth estimate \eqref{eq:Q-growth} gives
\begin{equation}\label{eq:large-ell-escape}
\Mscr_{\bar \ell, \mathfrak m}(w)=Q_w\phi_{\bar\ell,\mathfrak m}(w)\to 0
\qquad\text{ along every sequence } x_j  \text{ for which }  w(x_j)\to\infty.
\end{equation}

\medskip
\noindent{\bf Step 2: A coarse global bound.}
We claim that
\begin{equation}\label{eq:M-large-target}
\Mscr_{\bar\ell,\mathfrak m}\le\bar\ell
\quad\text{in }\R^n \quad \text{ for any } \mathfrak m>0.
\end{equation}
Indeed, we firstly observe that, the supremum of $\Mscr_{\bar\ell,\mathfrak m}$ is finite.  Indeed, on $\{w\le T\}$ the relative-gradient estimate \eqref{eq:gradient-w} and the  formula \eqref{eq:Q-algebraic} bound $Q_w$ and hence $\Mscr_{\bar\ell,\mathfrak m}$. Note that on $\{w>T\}$, the bound follows from \eqref{eq:large-ell-escape} up to increasing $T$ if necessary.  

Assume now that
$$S:=\sup_{\R^n}\Mscr_{\bar\ell,\mathfrak m}>\bar\ell,$$
and choose
$x_j$ with $\Mscr_{\bar\ell,\mathfrak m}(x_j)\to S$.  Then 
\eqref{eq:large-ell-escape} yields that $w(x_j)$ must be bounded from above, hence $V(x_j)=w(x_j)^{-a}$ is bounded below by a positive constant $\delta_0>0$.  Define
$$
V_j(x):=V(x+x_j) \quad \text{ and } \quad w_j(x):=w(x+x_j).
$$
For each fixed $R>0$, applying the  Harnack inequality in Lemma~\ref{lem:relative-gradient} together with $V_j(0)\ge\delta_0$ gives
$$
\inf_{B_{2R}}V_j\ge c_R>0.
$$
Moreover,
$$
0<V_j\le1, \qquad
\|g_\tau(V_j)\|_{L^\infty(B_{2R})}\le\Lambda.
$$
Thus, Lemma~\ref{lem:uniform-C1-compactness}  gives
$$
\sup_j\|V_j\|_{C^{1,\beta}(B_R)}\le C_R.
$$
After passing to a  subsequence,
$$
V_j\to V_\infty\qquad\text{in }C^1_{\rm loc}(\R^n).
$$
Passing to the weak formulation shows that
$-\Delta_pV_\infty=g_\tau(V_\infty)$ in $\R^n$.
Since $V_j\ge c_R$ on each fixed ball,
$$
w_j=V_j^{-1/a}\to w_\infty=V_\infty^{-1/a} \qquad\text{in }C^1_{\rm loc}(\R^n).
$$
Define
$$
\mathscr M_j:=Q_{w_j}\phi_{\bar\ell,\mathfrak m}(w_j).
$$
Then, locally uniformly,
$$
Q_{w_j}\to Q_\infty:=Q_{w_\infty},
\qquad \mathscr M_j\to \mathscr M_\infty
:= Q_\infty\phi_{\bar\ell,\mathfrak m}(w_\infty).
$$
Since $\mathscr M_j\le S$ in $\R^n$ and
$\mathscr M_j(0)\to S$, we have
\begin{equation}\label{eq:translated-attained-maximum}
\mathscr M_\infty(0)=S, \qquad \mathscr M_\infty\le S
\quad\text{in }\ \R^n.
\end{equation}

Since $S>\bar\ell$, choose a ball $D$ centered at the origin such that
$$
\overline D\subset \subset \{\mathscr M_\infty>\bar\ell\}.
$$
Since $0<\phi_{\bar\ell,\mathfrak m}\le1$,
$$
Q_\infty>\bar\ell \quad\text{on }\ \overline D.
$$
Consequently, $|\nabla w_\infty|$ is bounded away from zero on
$\overline D$, the coefficient matrix is uniformly elliptic, and the
drift is bounded. Furthermore,
$$
c_1:= \min_{\overline D} \mathfrak a_{\bar\ell,\mathfrak m}(w_\infty) Q_\infty(Q_\infty-\bar\ell)>0.
$$
Hence, after discarding the nonnegative square and singular-measure
terms,
$$
\Lscr_{w_\infty}\mathscr M_\infty -\mathbf b_\infty\cdot\nabla\mathscr M_\infty
\ge c_2>0  \quad\text{in }D
$$
in the weak sense. Lemma~\ref{lem:maximum-exclusion} then contradicts
the interior maximum in
\eqref{eq:translated-attained-maximum}. This proves our claim
\eqref{eq:M-large-target}.

Observe that the argument above holds for every $\mathfrak m>0$.  Fix $x$ and let
$\mathfrak m\to\infty$.  By \eqref{eq:m-to-infinity},
$$\phi_{\bar\ell,\mathfrak m}(w(x))\to 1,$$ and
\eqref{eq:M-large-target} gives
\begin{equation}\label{eq:coarse-Q-bound}
Q_w\le\bar\ell\quad\text{globally}.
\end{equation}

\medskip
\noindent{\bf Step 3: An upper bound with arbitrary $\ell>c_{\tau,0}$.}
Now fix an arbitrary $\ell>c_{\tau,0}$.
Suppose, for contradiction, that $Q_w(x_0)>\ell$ at some point $x_0\in\R^n$.
Since $w(x_0)<\infty$, the convergence \eqref{eq:m-to-infinity} allows us to choose $\mathfrak m>0$ sufficiently large that
\begin{equation}\label{eq:chosen-m-positive-level}
Q_w(x_0) \phi_{\ell,\mathfrak m}(w(x_0)) >\ell.
\end{equation}

Recall that, for this fixed $\mathfrak m$,
$\theta_{\ell,\mathfrak m}(t)\to\theta_{\ell,\infty}>0$
as $t\to\infty$ by \eqref{eq:theta-infinity}.
Consequently,
$$
\phi_{\ell,\mathfrak m}(t)\to 0  \qquad\text{as }t\to\infty.
$$
Together with the coarse bound \eqref{eq:coarse-Q-bound}, this gives
$$
0\le \bigl( Q_w\phi_{\ell,\mathfrak m}(w) \bigr)_+ \le  \bar\ell\, \phi_{\ell,\mathfrak m}(w) \to 0
$$
along every sequence for which $w\to\infty$, where for a real-valued function $F$, its positive part is denoted by
$$
    F_+(x):=\max\{F(x),0\}.
$$

Define
$$
S_\ell:=\sup_{\R^n} \mathscr M_{\ell,\mathfrak m}.
$$
By \eqref{eq:chosen-m-positive-level}, one has $S_\ell >\ell$.
Then a maximizing sequence for
$\mathscr M_{\ell,\mathfrak m}$ cannot satisfy $w\to\infty$;
hence $w$ remains bounded along a subsequence.
Repeating the  compactness argument under the translation used in  Step 2 above, we obtain an entire limiting solution for which the limiting multiplier attains an interior maximum $S_\ell>\ell$.

Choose a ball compactly contained in $\{\mathscr M_{\ell,\mathfrak m}>\ell\}$.
On this ball the operator is uniformly elliptic, the drift is bounded, and the  inequality \eqref{eq:completed-square} has a strictly positive zeroth-order density since
$$
\mathfrak a_{\ell,\mathfrak m}>0, \qquad Q_w(Q_w-\ell)>0.
$$
Thus Lemma~\ref{lem:maximum-exclusion} gives a contradiction, and  no point $x_0$ with $Q_w(x_0)>\ell$ exists. Hence
$$   Q_w\le\ell   \qquad\text{in }\R^n.$$
Finally, by letting $\ell\to (c_{\tau,0})^+$, we obtain \eqref{eq:Qw-upper}.
\end{proof}

\subsection{Tangent rigidity and completion of the  classification}\label{sec:classification-end}

Subtracting the  identity \eqref{eq:Q-algebraic} from the bound \eqref{eq:Qw-upper} in Proposition~\ref{prop:Q-bound}, and using the definition of $\mathfrak k_\tau$, we obtain a sharp first-order inequality
\begin{equation}\label{eq:sharp-gradient-w}
b_p|\nabla w|^p \le w\int_1^w\frac{c_\tau(t)}{t^2}\dd t
\quad\text{in }\R^n.
\end{equation}
The inequality is sharp at the minimum level $w=1$.  Indeed,
continuity of $c_\tau$ at $1$ gives
\begin{equation}\label{eq:sharp-gradient-asymptotic}
    w\int_1^w\frac{c_\tau(t)}{t^2}\,\dd t
    = c_{\tau,0}(w-1)+o(w-1)  \qquad\text{as } \ w\to 1^+.
\end{equation}

Recall that Ou obtains an explicit transformed profile after proving global vanishing of the trace-free stress tensor; see
\cite[(3.34)]{O2025} as well as the paragraph following it.
Here we use a different closure mechanism. First of all, we prove that this asymptotic inequality determines a unique $p'$-homogeneous tangent at every minimum point.  The subsequent lemmas then propagate the resulting equality away from the minimum.   

\begin{lem}\label{lem:unique-tangent}
Let $x_*\in\R^n$ satisfy $w(x_*)=1$, and define
\begin{equation}\label{eq:K-tau}
    K_\tau :=\frac1{p'}\left(\frac{c_{\tau,0}}n\right)^{1/(p-1)}.
\end{equation}
Then
\begin{equation}\label{eq:tangent-C1}
    \frac{w(x_*+rx)-1}{r^{p'}} \to K_\tau|x|^{p'} \qquad\text{in }C^1_{\loc}(\R^n)
\end{equation}
as $r\to 0$. In particular, there exists $r_0>0$ such that
$$
\nabla w(x)\ne0 \qquad \text{whenever }\ 0<|x-x_*|<r_0.
$$
Thus $x_*$ is the only critical point of $w$ in $B_{r_0}(x_*)$.
\end{lem}

\begin{proof}
Since $w\ge1$ and $w(x_*)=1$, the point $x_*$ is a
global minimum and $\nabla w(x_*)=0$.
Set $y(x):=w(x)-1.$
By continuity of $w$ and $c_\tau$, there exist $\rho,\delta,C>0$ such that
$$
0\le y\le\delta \quad\text{and}\quad
w\int_1^w\frac{c_\tau(t)}{t^2}\,\dd t \le Cy
\qquad\text{in }B_\rho(x_*).
$$
Hence \eqref{eq:sharp-gradient-w} gives
$$
|\nabla y|^p\le Cy \qquad\text{in }B_\rho(x_*).
$$

For $\epsilon>0$, define
$$
\Psi_\epsilon := (y+\epsilon)^{1/p'} -\epsilon^{1/p'}.
$$
Then
$$
|\nabla\Psi_\epsilon|^p
=\left(\frac1{p'}\right)^p (y+\epsilon)^{-1}|\nabla y|^p \le C\frac{y}{y+\epsilon} \le C.
$$
Since $\Psi_\epsilon(x_*)=0$, integration along line
segments in $B_\rho(x_*)$ gives
$$
\Psi_\epsilon(x)\le C|x-x_*|.
$$
Letting $\epsilon\to 0$, we obtain
\begin{equation}\label{eq:w-minimum-growth}
    0\le w(x)-1\le C|x-x_*|^{p'}
\end{equation}
for $x$ sufficiently close to $x_*$.

For $r>0$, define
$$
W_r(x):=\frac{w(x_*+rx)-1}{r^{p'}}.
$$
Then the transformed equation \eqref{eq:transformed-equation} becomes
\begin{equation}\label{eq:Wr-equation}
 \Delta_pW_r  =\frac{b_pr^{p'}|\nabla W_r|^p
+c_\tau(1+r^{p'}W_r)}{1+r^{p'}W_r}.
\end{equation}
The growth bound \eqref{eq:w-minimum-growth} bounds $W_r$ on each fixed ball.  Dividing  the scaled version of \eqref{eq:sharp-gradient-w} by $r^{p'}$, we obtain
\begin{equation}\label{eq:W-upper}
b_p  |\nabla W_r|^p \le \frac{1+r^{p'}W_r}{r^{p'}}\int_1^{1+r^{p'}W_r}\frac{c_\tau(t)}{t^2}\, \dd t \le C_R W_r     \qquad\text{on }\  B_R \ \text{ with }\  R>0
\end{equation}
for all sufficiently small $r$.
Hence, on every fixed ball $B_{2R}$, both $W_r$ and the right-hand side of \eqref{eq:Wr-equation} are uniformly bounded for all sufficiently small $r$.  Writing
$$
-\Delta_pW_r=f_r,
\qquad f_r:=-\frac{b_pr^{p'}|\nabla W_r|^p
+c_\tau(1+r^{p'}W_r)}{1+r^{p'}W_r},
$$
we have
$$
\sup_{0<r<r_R}\|f_r\|_{L^\infty(B_{2R})}<\infty,
$$
and Lemma~\ref{lem:uniform-C1-compactness} therefore makes
$\{W_r\}_{0<r<r_R}$ precompact in $C^1(B_R)$.

Take a sequence $r_j\to 0$ and let $W$ be a subsequential limit.  Then
\begin{equation}\label{eq:tangent-limit-system}
 \Delta_p W=c_{\tau,0},
 \qquad  W\ge0,  \qquad  W(0)=0,
 \qquad  b_p|\nabla W|^p\le c_{\tau,0}W  \quad \text{in } \ \R^n.
\end{equation}
Indeed, the equation follows from \eqref{eq:Wr-equation}. To obtain the last inequality, we divide \eqref{eq:sharp-gradient-w} by $r_j^{p'}$ and use the identity $c_\tau(1)=c_{\tau,0}$ to get that
$$
 \frac{1}{r_j^{p'}}\bigl(1 + r_j^{p'} W_{r_j}\bigr)
 \int_1^{1 + r_j^{p'} W_{r_j}} \frac{c_\tau(t)}{t^2}\,\dd t  \;\to\;  c_{\tau,0}\,W  \quad \text{locally uniformly.}
$$

The differential inequality in \eqref{eq:tangent-limit-system} needs to be used through an explicit zero-set regularization.  Towards this, we fix $\epsilon>0$ and define
$$
\Phi_\epsilon := (W+\epsilon)^{1/p'} -\epsilon^{1/p'}.
$$
Using the last inequality in
\eqref{eq:tangent-limit-system}, we obtain
\begin{align*}
|\nabla\Phi_\epsilon|^p &=\left(\frac 1 {p'}\right)^p (W+\epsilon)^{-1}|\nabla W|^p\\
&\le\left(\frac 1 {p'} \right)^p \frac{c_{\tau,0}}{b_p} \frac{W}{W+\epsilon} \le \left(\frac1{p'}\right)^p \frac{c_{\tau,0}}{b_p}.
\end{align*}
Since $\Phi_\epsilon(0)=0$, the Lipschitz estimate stated above implies that
$$
\Phi_\epsilon(x) \le \frac{1}{p'}\left(\frac{c_{\tau,0}}{b_p}\right)^{1/p}\lvert x\rvert.
$$
Equivalently, by raising both sides of this inequality to the power $p'$, we obtain
$$
\lvert \Phi_\epsilon(x)\rvert^{p'} \le \left(\frac{1}{p'}\right)^{p'}\left(\frac{c_{\tau,0}}{b_p}\right)^{1/(p-1)}\lvert x\rvert^{p'}.
$$
Letting $\epsilon\to 0$ gives
\begin{equation}\label{eq:tangent-upper-cone}
W(x)\le K_\tau|x|^{p'}.
\end{equation}
The constant in \eqref{eq:tangent-upper-cone} is exactly
\eqref{eq:K-tau} as $b_p=n/p'$.

We next show that the upper bound is indeed attained. To begin with, note that the  gradient estimate in \eqref{eq:tangent-limit-system} and \eqref{eq:tangent-upper-cone} imply
\begin{equation}\label{eq:tangent-stress-bound}
|\Acal(\nabla W(x))|=|\nabla W|^{p-1}\le \left(\frac{p'}{n} c_{\tau,0}\right)^{\frac {p-1}{p}} W^{\frac {p-1}{p}} \le\frac{c_{\tau,0}}n|x| \qquad\text{for almost every } x\in \R^n.
\end{equation}

For every $R>0$, let
$$
\nu(x):=\frac{x}{R} \qquad (x\in\partial B_R)
$$
denote the outward unit normal.
Testing the weak equation \eqref{eq:tangent-limit-system} with radial cutoffs converging to $\mathbf1_{B_R}$ and using the coarea formula yields that, for almost every $R>0$,
\begin{equation}\label{eq:tangent-flux}
 \int_{\partial B_R}\Acal(\nabla W)\cdot\nu\dd \mathcal H^{n-1}
 =c_{\tau,0}|B_R|  =\frac{c_{\tau,0}}nR|\partial B_R|.
\end{equation}

Observe that, on $\partial B_R$, \eqref{eq:tangent-stress-bound} gives
$$
\Acal(\nabla W)\cdot\nu \le |\Acal(\nabla W)|\le \frac{c_{\tau,0}}nR.
$$
Since the integral in \eqref{eq:tangent-flux} equals the
integral of the last upper bound, the nonnegative function
$$
\frac{c_{\tau,0}}nR -\Acal(\nabla W)\cdot\nu
$$
has zero integral over $\partial B_R$.  Hence
$$
\Acal(\nabla W)\cdot\nu =\frac{c_{\tau,0}}nR
$$
for $\mathcal H^{n-1}$-almost every point of
$\partial B_R$.
Therefore, the equality must hold in both inequalities above. The equality case of the Cauchy--Schwarz inequality gives
$$
\Acal(\nabla W) = \frac{c_{\tau,0}}nR\,\nu
= \frac{c_{\tau,0}}n x
$$
almost everywhere on $\partial B_R$.
Since this holds for almost every $R>0$, polar-coordinate
Fubini gives
\begin{equation}\label{eq:tangent-stress-identity}
\Acal(\nabla W) =\frac{c_{\tau,0}}n x \qquad \text{ almost everywhere in }\R^n.
\end{equation}

Moreover, thanks to 
$$
\Acal(\xi)=|\xi|^{p-2}\xi, \qquad \Acal^{-1}(\zeta)=|\zeta|^{p'-2}\zeta,
$$
the identity \eqref{eq:tangent-stress-identity} yields
$$
\nabla W(x) =\left(\frac{c_{\tau,0}}n\right)^{1/(p-1)}
    |x|^{p'-2}x
$$
almost everywhere.  As $W\in C^1(\R^n)$, this identity actually holds everywhere.  Integrating along the segment from $0$ to $x$, and using $W(0)=0$, we finally arrive at
$$
W(x) = \frac1 {p'} \left(\frac{c_{\tau,0}}n\right)^{1/(p-1)} |x|^{p'} = K_\tau|x|^{p'}.
$$
Since every convergent subsequence has the same limit and the family $\{W_r\}$ is locally precompact in $C^1$, the convergence in \eqref{eq:tangent-C1} holds for the full family as $r\to 0$.

Finally, on the annulus
$$
    A:=\{x\in\R^n:1/2\le|x|\le2\},
$$
the gradient of $K_\tau|x|^{p'}$ is bounded away from zero.
Hence the $C^1(A)$-convergence gives
$\nabla W_r\ne 0$ on $A$
for every sufficiently small $r$.
Then rescaling back to $w$ shows that
$$
 \nabla w(x)\ne 0 \qquad \text{whenever }0<|x-x_*|<r_0
$$
for some $r_0>0$.
\end{proof}

Define the nonnegative rigidity defect
\begin{equation}\label{eq:defn-Zw}
Z_w:=c_{\tau,0}-Q_w\ge0.
\end{equation}

\begin{lem}
\label{lem:local-Z-vanishing}
If $w(x_*)=1$, then there exists $r_*>0$ such that
$$
    Z_w\equiv0 \qquad\text{in }B_{r_*}(x_*).
$$
\end{lem}

\begin{proof}
Translate $x_*$ to the origin. We claim that Lemma~\ref{lem:unique-tangent} and the definition of $Z_w$ give
\begin{equation}\label{eq:Z-little-o}
Z_w(x)=o(|x|^{p'})
\quad \text{ as } \ x\to 0.
\end{equation}
To verify this explicitly, apply \eqref{eq:defect-functions} and \eqref{eq:Q-algebraic} to \eqref{eq:defn-Zw} and get
$$
 Z_w=\int_1^w\frac{c_\tau(t)}{t^2}\dd t -\frac{b_p}{w}|\nabla w|^p.
$$
In addition, the $C^1$ convergence in \eqref{eq:tangent-C1} yields
\begin{equation}\label{eq:w-expansion}
     w(x)-1=K_\tau|x|^{p'}+o(|x|^{p'}),
 \qquad  |\nabla w(x)|^p=(K_\tau p')^p|x|^{p'}+o(|x|^{p'}).
\end{equation}
Since
$$  K_\tau p' =  \left(\frac{c_{\tau,0}}n\right)^{1/(p-1)}
    \quad\text{and}\quad  b_p=\frac n{p'},
$$
one has
$$
b_p(K_\tau p')^p=c_{\tau,0}K_\tau.
$$
Thus by recalling \eqref{eq:sharp-gradient-asymptotic}, the terms of order $|x|^{p'}$ cancel, and \eqref{eq:Z-little-o} follows.

Furthermore, choose
\begin{equation}\label{eq:choose-sigma}
\max\{0,1-b_p\}<\sigma<1 \quad \text{ and }\quad \psi:=(w-1)^\sigma.
\end{equation}
By Lemma~\ref{lem:unique-tangent}, a punctured ball
$B_{r_0}\setminus\{0\}$ lies in the regular set where $\nabla w\neq 0$.  Then a direct chain-rule computation   gives
\begin{equation}\label{eq:Lpsi-formula}
 \Lscr_w\psi = (p-1)\sigma w^{2-n}(w-1)^{\sigma-2}\Bigg[
 (w-1)P_w +\left(\sigma-1+(2-n)\frac{w-1}{w}\right)|\nabla w|^p \Bigg].
\end{equation}
The right-hand side of \eqref{eq:Lpsi-formula} involves only the continuous functions $w$, $\nabla w$, and $P_w$.  Hence it is a continuous function on the punctured ball.  Moreover, by \eqref{eq:transformed-equation} and \eqref{eq:w-expansion}, when $|x|\to 0$,  
$$
 \frac 1 {w-1}\left[(w-1)P_w+ \left( \sigma-1+(2-n)\frac{w-1}{w} \right)|\nabla w|^p \right] \to c_{\tau,0} \left(1+\frac{\sigma-1}{b_p}\right)>0.
$$
Then up to decreasing $r_0$ suitably, we therefore have
\begin{equation}\label{eq:Lpsi-positive}
\Lscr_w\psi>0 \quad\text{in }B_{r_0}\setminus\{0\}.
\end{equation}
On the other hand, Proposition~\ref{prop:bochner} together with \eqref{eq:defn-Zw} gives
\begin{equation}\label{eq:LZ-negative}
\Lscr_wZ_w\le0 \quad\text{as a Radon measure on }B_{r_0}\setminus\{0\}.
\end{equation}

Suppose that $Z_w$ is not identically zero in
$B_{r_0}\setminus\{0\}$.  On every ball compactly contained in
this punctured domain, the matrix
$w^{2-n}\mathsf A_w$ is bounded and uniformly elliptic.
Then the strong minimum principle for distributional supersolutions Lemma~\ref{lem:linear-weak-harnack-minimum}
shows that the zero set of $Z_w$ is relatively open, which is also relatively closed by continuity.
Since $B_{r_0}\setminus\{0\}$ is connected, the assumption
$Z_w\not\equiv 0$ implies
$$
Z_w>0 \qquad\text{in }B_{r_0}\setminus\{0\}.
$$
Fix $0<r_1<r_0$ and put
$m_1:=\min_{\partial B_{r_1}}Z_w>0$.  Choose $\lambda>0$, independent of the inner radius, so small that
$$
\lambda\max_{\partial B_{r_1}}\psi\le\frac{m_1}{2}.
$$
For $0<\epsilon<r_1/2$, set
$M_\epsilon:=\max_{\partial B_\epsilon}\psi$ and
$$
H_\epsilon:= Z_w-\lambda(\psi-M_\epsilon)
\quad\text{in }B_{r_1}\setminus\overline{B}_\epsilon.
$$
On the inner boundary, $\psi-M_\epsilon\le0$, so
$H_\epsilon\ge0$.  On the outer boundary,
$H_\epsilon\ge m_1/2$.  Moreover,
\eqref{eq:Lpsi-positive} and \eqref{eq:LZ-negative} give, on the closed annulus
$A_\epsilon:=B_{r_1}\setminus\overline B_\epsilon$,
$$
\Lscr_wH_\epsilon\le -\lambda c_\epsilon\,\dd x
$$
for some $c_\epsilon>0$.  To avoid any sign convention ambiguity, fix
$\delta>0$ and test this inequality with
$$
\eta_\delta:=(-H_\epsilon-\delta)_+\in W^{1,2}_0(A_\epsilon).
$$
Note that, on $\{H_\epsilon<-\delta\}$, $\nabla\eta_\delta=-\nabla H_\epsilon$, hence
\begin{align*}
0&\le  \int_{A_\epsilon\cap\{H_\epsilon<-\delta\}}
 w^{2-n}\mathsf A_w\nabla H_\epsilon\cdot\nabla H_\epsilon\,\dd x\\
&= -\int_{A_\epsilon} w^{2-n}\mathsf A_w\nabla H_\epsilon\cdot\nabla\eta_\delta\,\dd x
 = \langle \Lscr_wH_\epsilon,\eta_\delta\rangle\\
&\le -\lambda c_\epsilon\int_{A_\epsilon}\eta_\delta\,\dd x\le 0.
\end{align*}
Thus $\eta_\delta=0$ for every $\delta>0$, and $H_\epsilon\ge0$ in $A_\epsilon$. 

Now letting
$\epsilon\to 0$ and using $M_\epsilon\to0$ gives
$$
Z_w\ge\lambda(w-1)^\sigma
\quad\text{in }B_{r_1}\setminus\{0\}.
$$
By the tangent expansion \eqref{eq:w-expansion},
$$
    (w(x)-1)^\sigma  = K_\tau^\sigma|x|^{\sigma p'}(1+o(1)).
$$
Consequently,
$$
  \lambda K_\tau^\sigma \le \frac{Z_w(x)}{|x|^{\sigma p'}} 
  =  o\!\left(|x|^{(1-\sigma)p'}\right),
$$
which is impossible as $x\to 0$ due to $\sigma<1$.
Thus $Z_w\equiv 0$ in a neighborhood of the origin.
\end{proof}

We now prove global rigidity.
\begin{prop}\label{prop:global-rigidity}
One has
\begin{equation}\label{eq:Q-equality-global}
Q_w\equiv c_{\tau,0}
\quad\text{in }\R^n.
\end{equation}
Moreover,
\begin{equation}\label{eq:w-explicit}
w(x)=1+K_\tau|x|^{p'}.
\end{equation}
\end{prop}

\begin{proof}
Let
$$
\mathcal N:=\{x\in\R^n:Z_w(x)=0\}.
$$
The algebraic formula for $Q_w$ and the $C^1$ regularity of $w$ show that $Z_w$ is continuous; hence $\mathcal N$ is closed.  It is nonempty since $w(0)=1$ and $\nabla w(0)=0$, so
$Q_w(0)=\mathfrak k_\tau(1)=c_{\tau,0}$.

We next prove that $\mathcal N$ is open.  Let $x\in\mathcal N$.  When $\nabla w(x)\ne0$, choose a ball compactly contained in the regular set $\Omega_{\rm reg}$.  Then one has $Z_w\ge 0$ there and $\Lscr_wZ_w\le0$ with a uniformly elliptic coefficient matrix.  The strong minimum principle Lemma~\ref{lem:linear-weak-harnack-minimum}, applied at the zero $x$, gives $Z_w\equiv0$ on a smaller ball.  

When $\nabla w(x)=0$, then the   identity \eqref{eq:defect-functions} and $Z_w(x)=0$ give
$$
 0=c_{\tau,0}-\mathfrak k_\tau(w(x))    =\int_1^{w(x)}\frac{c_\tau(t)}{t^2}\dd t.
$$
Since $c_\tau>0$, this forces $w(x)=1$. Then  Lemma~\ref{lem:local-Z-vanishing} with the centered  $x$, again gives a neighborhood contained in
$\mathcal N$.  Thus $\mathcal N$ is both open and closed.  Thus the connectedness of $\R^n$ proves \eqref{eq:Q-equality-global}.

We next exploit equality in the Bochner identity \eqref{eq:bochner}. 
Towards this, we first show that the critical set
$$
    \mathcal C:=\{x\in\R^n:\nabla w(x)=0\}
$$
has Lebesgue measure zero.  If $x\in\mathcal C$, then
\eqref{eq:Q-equality-global} and \eqref{eq:Q-algebraic} give
$$
0=c_{\tau,0}-\mathfrak k_\tau(w(x))
= \int_1^{w(x)}\frac{c_\tau(t)}{t^2}\,\dd t.
$$
Since $c_\tau>0$ and $w\ge1$, it follows that $w(x)=1$.
Lemma~\ref{lem:unique-tangent} therefore shows that every point of $\mathcal C$ is isolated.
Furthermore, the set $\mathcal C$ is closed as $\nabla w$ is continuous. Hence $\mathcal C$ is locally finite and, in particular, $|\mathcal C|=0.$

Fix $U\subset\subset\Omega_{\mathrm{reg}}$.
Thanks to $Q_w\equiv c_{\tau,0}$, one has
$$\Lscr_w Q_w=0 \quad \text{ in }\ U.$$
Then Proposition~\ref{prop:bochner}  gives
$$
0= n(p-1)w^{1-n}\tr(\mathsf E_w^2)\,\dd x +
(n-1)(p-1)w^{1-n} \Xbf_w\cdot D(\mathfrak e_\tau\circ w).
$$
Recall that both of terms on the right are nonnegative Radon measures. Hence
$$
 \tr(\mathsf E_w^2)=0 \qquad\text{almost everywhere in }U.
$$
Then by \eqref{eq:E-square-positive},
$\tr(\mathsf E_w^2)=|\mathsf T_w|^2$. Consequently
$$\mathsf E_w=0 \quad \text{ almost everywhere in} \ U.$$
Exhausting the regular set and using
$|\{\nabla w=0\}|=0$, we obtain
\begin{equation}\label{eq:DX-isotropic}
D\Xbf_w=\frac{P_w}{n}\Id \qquad\text{almost everywhere in }\R^n.
\end{equation}
As $\Xbf_w\in W^{1,2}_{\loc}$, this equality also holds
distributionally.

We next prove that $P_w$ is constant.
Fix $j\in\{1,\dots,n\}$ and choose $i\ne j$, which is possible as $n\ge2$. From \eqref{eq:DX-isotropic}, one gets
$$
\partial_i(\Xbf_w)_i=\frac{P_w}{n} \quad \text{ while } \quad \partial_j(\Xbf_w)_i=0.
$$
Therefore, in the sense of distributions,
$$
\partial_j P_w = n\partial_j\partial_i(\Xbf_w)_i =
n\partial_i\partial_j(\Xbf_w)_i =0.
$$
Since $j$ is arbitrary, $P_w$ is constant, its continuous representative satisfies $P_w(0)=c_{\tau,0}$, and hence
\begin{equation}\label{eq:P-constant}
P_w\equiv c_{\tau,0}.
\end{equation}
Equation \eqref{eq:DX-isotropic} now implies
$$
\Xbf_w(x)=\frac{c_{\tau,0}}n x+\xi_0.
$$
At the minimum point $0$, $\nabla w(0)=0$ and hence
$\Xbf_w(0)=0$; therefore $\xi_0=0$.  Now inverting the stress map gives
$$
\nabla w(x) =\left(\frac{c_{\tau,0}}n\right)^{1/(p-1)}
  |x|^{p'-2}x.
$$
Integration from the origin and $w(0)=1$ give
\eqref{eq:w-explicit}.
\end{proof}

\begin{rem} \label{rem:ou-closure-comparison}
Ou establishes global vanishing of the transformed trace-free stress via a weighted integral estimate; see \cite[equation~(3.34)]{O2025}, together with the subsequent explanatory paragraph. In contrast, our approach is qualitatively different. We first employ the global scalar inequality
$$
Q_w \le c_{\tau,0}
$$
to identify the unique tangent at a minimum point. Next, we use a punctured-ball barrier argument to extend this scalar estimate to a full neighborhood of the minimum; see the proof of Lemma~\ref{lem:local-Z-vanishing}. An open–closed argument then shows that $Q_w \equiv c_{\tau,0}$ globally. Only after this scalar rigidity has been established,  we return to the Bochner identity and deduce that the tensorial defect vanishes identically.
\end{rem}

\begin{proof}[Proof of Theorem~\ref{thm:frozen-classification}]
Proposition~\ref{prop:global-rigidity} gives
$$
 V=w^{-a}
   =\left(1+K_\tau|x|^{p'}\right)^{-a}.
$$
Since $c_{\tau,0}=a^{1-p}\widehat h_\tau$, the constant
$K_\tau$ is precisely $\gamma_\tau$ in \eqref{eq:gamma-tau}.

It remains to prove the flatness assertion for finite $\tau$.  From
\eqref{eq:Q-equality-global} and \eqref{eq:P-constant},
$$
\mathfrak e_\tau(w)=P_w-Q_w=0
\quad\text{throughout }\R^n.
$$
The explicit function $w$ has range $[1,\infty)$, so
$\mathfrak e_\tau(t)=0$ for every $t\ge1$.  Since $\mathfrak e_\tau\equiv0 $ on $[1,\infty)$, its distributional derivative vanishes on $(1,\infty)$.
Thus, the identity \eqref{eq:stieltjes-defect}
implies that the (Stieltjes) measure $\dd c_\tau$ vanishes on $(1,\infty)$.
Hence $c_\tau$ is constant there, and the continuity makes it
constant on $[1,\infty)$. If
$0<\tau<\infty$, formula \eqref{eq:c-tau} says
$h(\tau t^{-a})=h(\tau)$ for every $t\ge1$.  The numbers
$\tau t^{-a}$ fill $(0,\tau]$, proving \eqref{eq:h-flat-below-tau}.
\end{proof}

\section{Proof of the Harnack inequality}

We start with a preparation for the contradiction. The mechanism to select the points below is essentially the same as the one
used in \cite[(4.1)--(4.6)]{QZ2026}. The current proposition strengthens that mechanism by obtaining a sharp local upper bound converging to $1$ and by identifying every amplitude limit through the compact family rather than assuming pure-power classification.

\begin{prop} \label{prop:point-selection}
Suppose that $R_j>0$ and that $u_j>0$ solves
$$
-\Delta_pu_j=g(u_j)\quad\text{in }B_{3R_j}.
$$
Set
$$
M_j:=\max_{\overline{B_{R_j}}}u_j, 
\qquad m_j:=\min_{\overline{B_{2R_j}}}u_j, 
$$
and assume
\begin{equation}\label{eq:selection-assump}
R_j^{n-p}M_jm_j^{p-1}\to \infty.
\end{equation}
Then, up to passing to a subsequence, there are points
$x_j\in B_{2R_j}$ and numbers
$$
A_j:=u_j(x_j), \quad D_j:=\operatorname{dist}(x_j,\partial B_{2R_j}), \quad \sigma_j:=\frac{D_j}{8},\quad 
\delta_j:=A_j^{-p/(n-p)},\quad
\Gamma_j:=\frac{\sigma_j}{\delta_j},
$$
with $\Gamma_j\to\infty$, such that the rescaled functions
\begin{equation}\label{eq:defn-vj}
    v_j(y):=A_j^{-1}u_j(x_j+\delta_jy)
\end{equation}
are defined in $B_{4\Gamma_j}$ and satisfy
\begin{equation}\label{eq:vj-normalization}
v_j(0)=1, \qquad 0<v_j\le2^{n-p}\quad\text{in }B_{4\Gamma_j},
\end{equation}
\begin{equation}\label{eq:vj-equation}
-\Delta_pv_j=g_{A_j}(v_j)=v_j^qh(A_jv_j),
\end{equation}
and
\begin{equation}\label{eq:vj-limit}
v_j\to  U_\tau \quad\text{in }C^1_{\loc}(\R^n)
\end{equation}
for some $\tau\in[0,\infty]$, where
\begin{equation}\label{eq:Utau-normalized}
-\Delta_pU_\tau=g_\tau(U_\tau), \qquad 0<U_\tau\le 1, \qquad U_\tau(0)=1.
\end{equation}

Moreover, for 
\begin{equation}\label{eq:defn-xi}
\sigma_j=\frac{\dist(x_j,\,\partial B_{2R_j})}{8} \quad \text{ and }\quad \Xi_j:=A_j\sigma_j^{n-p}m_j^{p-1},
\end{equation}
one has
\begin{equation}\label{eq:Xi-Gamma}
\Xi_j\to\infty, \qquad 
\Gamma_j^{n-p} =\Xi_j\left(\frac{A_j}{m_j}\right)^{p-1} \ge\Xi_j.
\end{equation}
\end{prop}

\begin{proof}
The continuous function 
$$
x\mapsto\dist(x,\partial B_{2R_j})^{n-p}u_j(x)
$$
vanishes on
$\partial B_{2R_j}$ and is positive in $B_{2R_j}$. It therefore attains its maximum at some interior point
$x_j\in B_{2R_j}$.
Set
$$
A_j:=u_j(x_j), \qquad
D_j:=\dist(x_j,\partial B_{2R_j}),
\qquad \sigma_j:=\frac{D_j}{8}.
$$
If $y_j\in\overline{B_{R_j}}$ is a maximum point of $u_j$, then $\dist(y_j,\partial B_{2R_j})\ge R_j$.  Hence
\begin{equation}\label{eq:point-selection-lower}
A_jD_j^{n-p}\ge M_jR_j^{n-p}.
\end{equation}
Note that, for each $x\in B_{D_j/2}(x_j)=B_{4\sigma_j}(x_j)$, one has
$\dist(x,\partial B_{2R_j})\ge D_j/2$.  Then the maximality of $x_j$ gives
\begin{equation}\label{eq:point-selection-upper}
u_j(x)\le2^{n-p}A_j \quad\text{in }B_{4\sigma_j}(x_j).
\end{equation}

Define
\begin{equation}\label{eq:defn-Gammaj}
\delta_j:= A_j^{-p/(n-p)},
\qquad \Gamma_j:= \frac{\sigma_j}{\delta_j},
\qquad v_j(y):= A_j^{-1} u_j(x_j+\delta_j y).
\end{equation}
Then \eqref{eq:vj-normalization} follows from the definition of $A_j$ together with \eqref{eq:point-selection-upper}.  Moreover, a change of variable in the weak formulation \eqref{eq:weak-equation-general} yields
$$
 -\Delta_p v_j =\delta_j^p A_j^{1-p} g(A_jv_j)
$$
as $g(s)=s^qh(s)$. 
Since $q-p+1=p^2/(n-p)$ by the third identity of \eqref{eq:exponent-identities}, one gets $\delta_j^p A_j^{q-p+1}=1$ which proves \eqref{eq:vj-equation}.

By \eqref{eq:selection-assump}, \eqref{eq:defn-xi} and \eqref{eq:point-selection-lower},
$$
 \Xi_j =8^{-(n-p)}A_jD_j^{n-p}m_j^{p-1}
 \ge  8^{-(n-p)}M_jR_j^{n-p}m_j^{p-1}\to\infty.
$$
Also, the definition of $\Gamma_j$ in \eqref{eq:defn-Gammaj} gives
$$
 \Gamma_j^{n-p}  =\sigma_j^{n-p}A_j^p
 = 8^{-(n-p)}D_j^{n-p} A_j^p =\Xi_j\left(\frac{A_j}{m_j}\right)^{p-1}.
$$
Since $x_j\in B_{2R_j}$, one has $A_j\ge m_j$, which gives \eqref{eq:Xi-Gamma} and concludes $\Gamma_j\to\infty$.

Fix $K>0$.  For $|y|\le K$ and $K<8\Gamma_j$, 
by triangle inequality
$$
 \dist(x_j+\delta_j y,\,\partial B_{2R_j})\ge D_j-\frac{\sigma_j K}{\Gamma_j} =  D_j\left(1-\frac{K}{8\Gamma_j}\right).
$$
Thus the maximality of $x_j$ yields the bound
$$u_j(x_j+\delta_j y) \operatorname{dist}
\bigl(x_j+\delta_j y,\partial B_{2R_j}\bigr)^{n-p} \le A_j D_j^{n-p},$$
which after scaling gives
\begin{equation}\label{eq:sharp-local-upper}
v_j(y)\le\left(1-\frac{K}{8\Gamma_j}\right)^{-(n-p)}.
\end{equation}

Now up to taking a subsequence, $A_j\to\tau\in[0,\infty]$.
Fix $R>0$.  For all sufficiently large $j$,
$B_{2R}\subset B_{4\Gamma_j}$, and
$$
0<v_j\le2^{n-p},\qquad 
\|g_{A_j}(v_j)\|_{L^\infty(B_{2R})} \le\Lambda\,2^{(n-p)q}.
$$
Lemma~\ref{lem:uniform-C1-compactness} gives a uniform
$C^{1,\beta}(B_R)$ bound.  Moreover, rewriting the equation as
$$
-\Delta_pv_j=b_j(y)v_j^{p-1},
\qquad 0\le b_j(y):=v_j(y)^{q-p+1}h(A_jv_j(y))\le C(n,p,\Lambda),
$$
and applying a finite Harnack chain from the point $v_j(0)=1$ gives
$$
v_j\ge c_R>0\qquad\text{in }B_R.
$$
Up to choosing a suitable subsequence,
$$
v_j\to U_\tau\qquad\text{in }C^1_{\rm loc}(\R^n).
$$
Then Lemma~\ref{lem:frozen-convergence} and
$v_j(B_R)\subset[c_R,2^{n-p}]$ imply
$$
g_{A_j}(v_j)\to g_\tau(U_\tau) \qquad\text{uniformly on }\ B_R.
$$
Passing to the weak formulation yields
$-\Delta_pU_\tau=g_\tau(U_\tau)$.
Letting $j\to\infty$ in \eqref{eq:sharp-local-upper} gives
$U_\tau\le 1$, while $U_\tau(0)=1$ and the strong minimum principle give $U_\tau>0$.
Thus  we conclude \eqref{eq:Utau-normalized}.
\end{proof}

\subsection{The Pohozaev identity and isolated  poles}\label{sec:pole}

Let $D\subset\R^n$ be open, let $f:[0,\infty)\to[0,\infty)$ be continuous, and put
$$
F(s):=\int_0^sf(t)\dd t.
$$
For a $C^1$ weak solution of $-\Delta_pu=f(u)$ in a ball centered at the origin, define
\begin{equation}\label{eq:Pohozaev-flux}
\begin{aligned}
 \Pscr_u(r):=\int_{\partial B_r}\Bigg[
 a u|\nabla u|^{p-2}u_\nu  -\frac r p|\nabla u|^p +r|\nabla u|^{p-2}u_\nu^2 +rF(u) \Bigg]\dd \mathcal H^{n-1},
\end{aligned}
\end{equation}
where $u_\nu=\nabla u\cdot\nu$.

The following proposition is classical, and its proof is standard. A portion of the argument already appears in e.g. \cite[Proposition 3.4]{QZ2026}. For the sake of completeness, we present a full proof in the Appendix.  
\begin{prop} \label{prop:local-Pohozaev}
Let $I\subset\mathbb R$ be an open interval, 
$f\in C(I)$, and $F\in C^1(I)$ satisfy $F'=f$.
Suppose
$$
u\in C^1(B_R)\cap W^{1,p}(B_R), \qquad u(B_R)\subset I,
$$
and that $u$ is a weak solution of
$$
-\Delta_pu=f(u) \qquad\text{in }B_R.
$$  
Then, for every $r\in(0,R)$,
\begin{equation}\label{eq:local-Pohozaev}
\Pscr_u(r) =\int_{B_r}\bigl[nF(u)-a u f(u)\bigr]\dd x.
\end{equation}
If the equation is assumed only in the annulus
$B_s\setminus\overline B_r$, the same identity holds in the form
\begin{equation}\label{eq:annular-Pohozaev}
\Pscr_u(s)-\Pscr_u(r) =\int_{B_s\setminus\overline B_r}
  \bigl[nF(u)-a u f(u)\bigr]\dd x.
\end{equation}
\end{prop}

We next isolate the precise consequence of the singularity theorem of Kichenassamy--V\'eron that is needed below.  Recall
$$
\Phi_\kappa(x):=\kappa|x|^{-\alpha},
\qquad \mathcal M_\kappa
:=|\mathbb S^{n-1}|(\alpha\kappa)^{p-1}.
$$
Then $-\Delta_p\Phi_\kappa=\mathcal M_\kappa\delta_0$ in distributions. The following proposition was partially established in \cite[Section 3]{QZ2026}. 

\begin{prop} \label{prop:pole-expansion}
Let $W>0$ be $p$-harmonic in $B_R\setminus\{0\}$ and assume
$$
W(x)\le C|x|^{-\alpha} \quad \text{ for any } \ 0<|x|<R/2.
$$
Then there are $\kappa\ge0$, $\beta\in\R$, and a function $z$ such that
$$
W(x)=\kappa|x|^{-\alpha}+\beta+z(x) \quad \text{ and } \quad z(x)\to0
\quad \text{ as } \ x\to 0.
$$
If $\kappa>0$, then
\begin{equation}\label{eq:improved-pole-derivative}
r\sup_{|x|=r}|\nabla z(x)|\to 0
\quad \text{ as }\ r\to 0.
\end{equation}
Moreover,
\begin{equation}\label{eq:pole-distribution}
-\Delta_pW=\mathcal M_\kappa\delta_0
\quad\text{in }B_R,
\end{equation}
and the $p$-harmonic Pohozaev flux has the finite value
\begin{equation}\label{eq:pole-Pohozaev-value}
\Pscr_W(r)
=-a\mathcal M_\kappa\beta
\quad \text{ for any } \ 0<r<R.
\end{equation}
\end{prop}

\begin{proof}
Kichenassamy--V\'eron in \cite[Theorem~1.1 \& Remark~1.4]{KV1986} give us the coefficient $\kappa$, the boundedness of the remainder, the existence of its finite limit $\beta$, and the weighted derivative asymptotic
\begin{equation}\label{eq:KV-weighted-gradient}
|x|^{\alpha+1} |\nabla(W-\Phi_\kappa)(x)|\to 0.
\end{equation}
We next derive the  Dirac mass normalization below in the sign convention of the present paper.  It remains to prove the sharper derivative estimate and then compute the finite part of the Pohozaev flux.

Moreover, if $\kappa=0$, the Kichenassamy--V\'eron theorem gives a removable singularity.  The extension of $W$ is therefore $p$-harmonic in $B_R$, i.e. $\mathcal M_\kappa=0$, and Proposition~\ref{prop:local-Pohozaev} with $f\equiv0$ gives
$$
\Pscr_W(r)=0=-a\mathcal M_\kappa\beta.
$$
Then the claim of the proposition follows immediately. We may therefore assume $\kappa>0$. 

\medskip
\noindent
{\bf Step 1: The vanishing  $|x|\nabla z$ at the origin.}

Recall that $\Acal(\xi):=|\xi|^{p-2}\xi.$
Write
$\Psi:=\Phi_\kappa+\beta$ so that $z=W-\Psi$.  Since $W$ and $\Psi$ are $p$-harmonic away from the origin,
\begin{equation}\label{eq:z-linear-equation-physical}
\operatorname{div}(\mathcal B(x)\nabla z)=0,
\qquad \text{ for }\ 
\mathcal B(x):=\int_0^1D\Acal(\nabla\Psi+t\nabla z)\dd t.
\end{equation}
Set the fixed annuli
$$
A:=\{1/2<|y|<2\},\qquad
A_1:=\{2/3<|y|<3/2\},\qquad
A_2:=\{3/4<|y|<4/3\}.
$$
By \eqref{eq:KV-weighted-gradient}, for all sufficiently small $r$ and all $y\in A$,
$$
|\nabla z(ry)|\le\frac12|\nabla\Phi_\kappa(ry)|.
$$
Consequently, no vector in the defining segment in \eqref{eq:z-linear-equation-physical} for $\mathcal B(ry)$ is equal to zero. 

Set $z_r(y):=z(ry)$.  Then equation
\eqref{eq:z-linear-equation-physical} becomes
\begin{equation}\label{eq:equ-qr}
\operatorname{div}_y(\widehat{\mathcal B}_r\nabla z_r)=0
\quad\text{in } A,
\end{equation}
where, by the $(p-2)$-homogeneity property of $D\Acal$, we obtain
$$
\widehat{\mathcal B}_r(y):= r^{(\alpha+1)(p-2)} \mathcal B(ry).
$$
The ellipticity constants of $\widehat{\mathcal B}_r$ on $A$ are independent of $r$.  Set $W_r(y):=r^\alpha W(ry)$. 
The estimate of Tolksdorf  \cite{T1984} applied on fixed annuli gives
\begin{equation}\label{eq:Wr-C1gamma}
\sup_{0<r<r_0} \|W_r\|_{C^{1,\gamma}(A_1)}<\infty.
\end{equation}
Then the Kichenassamy--V\'eron derivative asymptotic gives
$$
\nabla W_r\to\nabla_y(\kappa|y|^{-\alpha}) \quad\text{uniformly on }A.
$$
The limiting gradient never vanishes on $A$. Moreover,
$$
\widehat{\mathcal B}_r
=\int_0^1 D\Acal\!\left( \nabla_y\Phi_\kappa(y)
 +t[\nabla_yW_r(y)-\nabla_y\Phi_\kappa(y)]
\right)\,\dd t.
$$
Therefore, the defining segments for $\widehat{\mathcal B}_r$ remain in a compact subset of $\R^n\setminus\{0\}$ for all sufficiently small$r$, and then the smoothness of $D\Acal$ away from the origin and the preceding $C^{1,\gamma}$ bound \eqref{eq:Wr-C1gamma} give
\begin{equation}
\sup_{0<r<r_0} \|\widehat{\mathcal B}_r\|_{C^{0,\gamma}(A_1)}<\infty,
\end{equation}
with uniform ellipticity constants.

Since $\widehat{\mathcal B}_r$ is uniformly elliptic and uniformly bounded in $C^{0,\gamma}(A_1)$, the interior Schauder estimate for divergence-form elliptic equations (see, for instance, \cite[Chapter~8]{GT2001}) applied to \eqref{eq:equ-qr} implies that $\sup_{A_2}|\nabla z_r|$ is bounded above by $\sup_{A_1}|z_r|$, up to a multiplicative constant that is independent of $r$.
 Since $z(x)\to 0$ as $x\to 0$, we get
$$
\|z_r\|_{L^\infty(A_1)}\to 0.
$$
Therefore $\sup_{A_2}|\nabla z_r|\to 0.$
Note that for $|y|=1$, $\nabla z_r(y)=r\nabla z(ry)$, which proves
\eqref{eq:improved-pole-derivative}.

\medskip
\noindent
{\bf Step 2: The mass at the origin.}
We now compute the distributional mass.  The asymptotic
\eqref{eq:KV-weighted-gradient} gives, on $\partial B_r$,
$$
|\nabla W-\nabla\Phi_\kappa|=o(r^{-\alpha-1}).
$$
Since $|\nabla\Phi_\kappa|=\alpha\kappa r^{-\alpha-1}$, the fundamental theorem of calculus applied to $\Acal$ along the segment between $\nabla\Phi_\kappa$ and $\nabla W$ gives
$$
|\Acal(\nabla W)-\Acal(\nabla\Phi_\kappa)|
=o\bigl(r^{-(\alpha+1)(p-1)}\bigr)
=o(r^{-(n-1)}).
$$
Let $\varphi\in C_c^\infty(B_R)$ and choose $\rho<R$ with
$\operatorname{spt}\varphi\subset\subset B_\rho$.  Since $W$ is $p$-harmonic in
$B_\rho\setminus\overline B_r$,
$$
\begin{aligned}
 \int_{B_\rho\setminus B_r}\Acal(\nabla W)\cdot\nabla\varphi\,\dd x
 &=-\int_{\partial B_r}\varphi\,\Acal(\nabla W)\cdot\nu\,\dd \mathcal H^{n-1}  \\
 &=-\int_{\partial B_r}\varphi\,\Acal(\nabla\Phi_\kappa)\cdot\nu\,\dd \mathcal H^{n-1}+o(1).
\end{aligned}
$$
Here $\nu=x/|x|$ is the outward unit normal of $B_r$.  Moreover,
$$
\Acal(\nabla\Phi_\kappa)\cdot\nu =-(\alpha\kappa)^{p-1}r^{-(n-1)}.
$$
Letting $r\to 0$ yields
$$
\begin{aligned}
 \int_{B_R}\Acal(\nabla W)\cdot\nabla\varphi\,\dd x
 &=\lim_{r\to 0}(\alpha\kappa)^{p-1}r^{-(n-1)}
   \int_{\partial B_r}\varphi\,\dd \mathcal H^{n-1}  \\
 &=|\mathbb S^{n-1}|(\alpha\kappa)^{p-1}\varphi(0).
\end{aligned}
$$
This proves \eqref{eq:pole-distribution} with the normalization
$\mathcal M_\kappa=|\mathbb S^{n-1}|(\alpha\kappa)^{p-1}$.

\medskip
\noindent
{\bf Step 3: The Pohozaev value.}
This part is similar to \cite[Proposition 3.5]{QZ2026}. 
To begin with,  note that the primitive term in \eqref{eq:Pohozaev-flux} is absent, since $W$ is $p$-harmonic.  For a function $v$ defined near
$\partial B_r$, set
\begin{equation}\label{eq:defn-Ir}
\mathcal I_r(v):=a v\Acal(\nabla v)\cdot\nu
-\frac r p|\nabla v|^p +r|\nabla v|^{p-2}v_\nu^2.
\end{equation}

Write
$$
W=\Phi_\kappa+\beta+z, \qquad
\xi_r:=\nabla\Phi_\kappa, \qquad
\eta_r:=\nabla z \quad\text{on }\partial B_r,
$$
and define
$$
\delta_r:= \sup_{\partial B_r}\frac{|\eta_r|}{|\xi_r|}.
$$
By \eqref{eq:improved-pole-derivative},
$|\eta_r|=o(r^{-1})$, whereas
$|\xi_r|=\alpha\kappa r^{-\alpha-1}$.  Hence
\begin{equation}\label{eq:deltar}
\delta_r=o(r^\alpha).
\end{equation}

For all sufficiently small $r$, one has
$|\eta_r|\le |\xi_r|/2$.  Every vector on the segment joining
$\xi_r$ to $\xi_r+\eta_r$ then has norm between
$|\xi_r|/2$ and $3|\xi_r|/2$.  The fundamental theorem of calculus, applied to the maps
$$
\xi\mapsto\Acal(\xi), \qquad \xi\mapsto|\xi|^p,
\qquad  \xi\mapsto|\xi|^{p-2}(\xi\cdot\nu)^2,
$$
therefore gives, uniformly on $\partial B_r$ and for every $p>1$,
\begin{equation}\label{eq:vector-inequ}
\begin{aligned}
 |\Acal(\xi_r+\eta_r)-\Acal(\xi_r)|
 &\le C\delta_r|\xi_r|^{p-1},\\
 \bigl||\xi_r+\eta_r|^p-|\xi_r|^p\bigr|
 &\le C\delta_r|\xi_r|^p,\\
 \left|
 |\xi_r+\eta_r|^{p-2}
 ((\xi_r+\eta_r)\cdot\nu)^2
 -|\xi_r|^p
 \right|
 &\le C\delta_r|\xi_r|^p.
\end{aligned}
\end{equation}
Note the argument above is also valid when $1<p<2$, since the relevant segment stays a fixed positive distance from the origin relative to $|\xi_r|$.

We now compare the full Pohozaev integrand of $W$ with that of the model one $\Phi_\kappa+\beta$. 
First of all, by the last two estimates in \eqref{eq:vector-inequ},
\begin{align}
&\int_{\partial B_r}
\left| -\frac r p\bigl(|\nabla W|^p-|\xi_r|^p\bigr)
 +r\left( |\nabla W|^{p-2}W_\nu^2-|\xi_r|^p
 \right)\right|\dd \mathcal H^{n-1}\notag\\
 \le& \ Cr|\partial B_r|\delta_r|\xi_r|^p
\le C(n,p,\kappa)r^{-\alpha}\delta_r =o(1), \label{eq:term1}
\end{align}
where we used
$n-p(\alpha+1)=-\alpha$ and \eqref{eq:deltar}.
Next, we compute
\begin{align*}
&W\Acal(\nabla W)\cdot\nu
-(\Phi_\kappa+\beta) \Acal(\nabla\Phi_\kappa)\cdot\nu\\
 =&\ z\Acal(\xi_r)\cdot\nu
 +W\bigl[ \Acal(\xi_r+\eta_r)-\Acal(\xi_r) \bigr]\cdot\nu.
\end{align*}
The first term satisfies
\begin{equation}\label{eq:term2}
  \int_{\partial B_r}|z|\,|\Acal(\xi_r)|\dd \mathcal H^{n-1}
\le \mathcal M_\kappa\sup_{\partial B_r}|z| \to 0.  
\end{equation}
On the other hand, the pole expansion gives $W\le Cr^{-\alpha}$ on $\partial B_r$ for small $r$, and hence the first inequality in \eqref{eq:vector-inequ} yields
\begin{align}
&\int_{\partial B_r}
W\left|\Acal(\xi_r+\eta_r)-\Acal(\xi_r) \right|\dd \mathcal H^{n-1}\notag\\
\le& \ Cr^{-\alpha}|\partial B_r| \delta_r|\xi_r|^{p-1}
\le C(n,p,\kappa)r^{-\alpha}\delta_r =o(1). \label{eq:term3}
\end{align}
Combining \eqref{eq:term1}, \eqref{eq:term2} and \eqref{eq:term3}, we arrive at
\begin{equation}\label{eq:full-pohozaev-comparison}
\int_{\partial B_r}
\left|\mathcal I_r(W)-\mathcal I_r(\Phi_\kappa+\beta) \right|\dd \mathcal H^{n-1} \to 0
\qquad\text{as }r\to 0;
\end{equation}
recall $\mathcal I_r$ in \eqref{eq:defn-Ir}.

It remains to compute the model value.  Since
$\partial_\nu\Phi_\kappa=-|\nabla\Phi_\kappa|$ and
$\Phi_\kappa=(r/\alpha)|\nabla\Phi_\kappa|$ on $\partial B_r$,
$$
\Pscr_{\Phi_\kappa}(r) =\left(-\frac a\alpha-\frac1p+1\right) r\int_{\partial B_r}|\nabla\Phi_\kappa|^p\dd \mathcal H^{n-1} =0,
$$
as $a/\alpha=(p-1)/p$.  Adding the constant $\beta$ does not
change the gradient, and therefore
$$
\Pscr_{\Phi_\kappa+\beta}(r) = \Pscr_{\Phi_\kappa}(r)
+a\beta\int_{\partial B_r}\Acal(\nabla\Phi_\kappa)\cdot\nu\dd \mathcal H^{n-1} =-a\mathcal M_\kappa\beta.
$$
Together with \eqref{eq:full-pohozaev-comparison}, this proves
$$
\lim_{r\to 0}\Pscr_W(r) =-a\mathcal M_\kappa\beta.
$$
Now we conclude \eqref{eq:pole-Pohozaev-value} from \eqref{eq:annular-Pohozaev}.
\end{proof}

\begin{rem}\label{rem:KV-erratum}
The strict comparison lemma cited in the original paper \cite{KV1986} imposes the non-vanishing of the gradient of the lower comparison function, as clarified and corrected in the subsequent erratum \cite[Page 352]{KV1987}. In all applications of strict comparison in the present work, the comparison function is chosen to be a fundamental pole, whose gradient is nonzero on a punctured annulus around the pole. Consequently, the strengthened hypothesis of the corrected lemma is fulfilled in each instance.
\end{rem}

\subsection{The first-contact}\label{sec:first-contact}

For the bubble in Theorem~\ref{thm:frozen-classification}, set
$\kappa_\tau:=\gamma_\tau^{-a}.$
Then it follows that
\begin{equation}\label{eq:bubble-tail}
r^\alpha U_\tau(r)\to \kappa_\tau
\quad \text{ as } \ r\to\infty,
\end{equation}
and
\begin{equation}\label{eq:bubble-source-mass}
 \mathcal M_\tau :=\int_{\R^n}g_\tau(U_\tau)\dd x
 =|\mathbb S^{n-1}|(\alpha\kappa_\tau)^{p-1}.
\end{equation}
To verify the last equality, integrate the equation of $U_\tau$ over $B_R$:
$$
 \int_{B_R}g_\tau(U_\tau)\dd x  =-\int_{\partial B_R}\Acal(\nabla U_\tau)\cdot\nu\dd \mathcal H^{n-1}.
$$
The explicit bubble gives
\begin{equation}\label{eq:estimate-U}
U_\tau(R)=\kappa_\tau R^{-\alpha}+o(R^{-\alpha})
\end{equation}
and
$$U_\tau'(R)=-\alpha\kappa_\tau R^{-\alpha-1}+o(R^{-\alpha-1}).$$
Since $(\alpha+1)(p-1)=n-1$, the boundary integral above converges to $|\mathbb S^{n-1}|(\alpha\kappa_\tau)^{p-1}$ as $R\to \infty$.

The proof is inspired by the first-crossing construction in the proof of \cite[Lemma~3.7]{QZ2026}.
That lemma assumed a preliminary neck estimate and concludes an outer-sphere infimum bound.  Under the present two-sided critical source bounds, the theorem below removes that preliminary neck hypothesis and proves pointwise dominance.

\begin{prop} \label{prop:first-contact}
Let $\Gamma_j\to\infty$, let $g_j:[0,\infty)\to[0,\infty)$ be continuous, and let $v_j>0$ solve
\begin{equation}\label{eq:equation-vj}
-\Delta_pv_j=g_j(v_j) \quad\text{in }B_{4\Gamma_j},
\end{equation}
where
\begin{equation}\label{eq:source-fj}
\begin{gathered}
Ls^q\le g_j(s)\le\Lambda s^q\quad(s\ge0),\\
G_j(s):=\int_0^sg_j(t)\dd t,
\qquad
nG_j(s)-a s g_j(s)\ge0,
\end{gathered}
\end{equation}
and $g_j\to g_\tau$ locally uniformly on $(0,\infty)$.  Assume
$$
v_j(0)=1, \qquad 0<v_j\le C_0, \qquad
v_j\to U_\tau \quad\text{in }C^1_{\loc}(\R^n).
$$
Then, for every $\vartheta>1$ and every sequence
$S_j=o(\Gamma_j)$,
\begin{equation}\label{eq:first-contact-conclusion}
v_j<\vartheta U_\tau
\quad\text{in }B_{S_j}
\end{equation}
for all sufficiently large $j$.
\end{prop}

\begin{proof}
We divide the proof into six steps.

\medskip
\noindent{\bf Step 1: The first contact and its blow-down.}
Suppose that \eqref{eq:first-contact-conclusion} fails along a
subsequence.  Define
$$\rho_j:=\inf\left\{ r\in[0,S_j]: \max_{\overline B_r}\frac{v_j}{U_\tau} \ge\vartheta \right\}.
$$
Since $v_j/U_\tau\to 1$ locally uniformly and $\vartheta>1$,
one has $\rho_j\to\infty$.  Continuity gives a point
${\bf e}_j\in\mathbb S^{n-1}$ such that
\begin{equation}\label{eq:first-contact-radius}
    v_j<\vartheta U_\tau\quad\text{in }B_{\rho_j},
\qquad v_j(\rho_j{\bf e}_j) =\vartheta U_\tau(\rho_j).
\end{equation}
Moreover,
$$
\rho_j\le S_j=o(\Gamma_j).
$$

Set
\begin{equation}\label{eq:defn-Wj}
W_j(x):=\rho_j^\alpha v_j(\rho_jx), \qquad \epsilon_j:=\rho_j^{-p'},
\end{equation}
and
\begin{equation}\label{eq:hat-fj}
\widehat g_j(s):=\rho_j^ng_j(\rho_j^{-\alpha}s),
\qquad \widehat G_j(s):=\rho_j^{n+\alpha}G_j(\rho_j^{-\alpha}s).
\end{equation}
Then the exponent identities in \eqref{eq:exponent-identities} with $ap'=\alpha$, together with \eqref{eq:equation-vj} and \eqref{eq:source-fj}, give
\begin{equation}\label{eq:blowdown-equation-bounds}
\begin{gathered}
-\Delta_pW_j=\widehat g_j(W_j)
\quad\text{in } \ B_{4\Gamma_j/\rho_j},\qquad 
L\epsilon_js^q\le\widehat g_j(s) \le \Lambda\epsilon_js^q,\\
  n\widehat G_j(s)-a s\widehat g_j(s)\ge0,
\qquad \epsilon_j^aW_j = \rho_j^{-ap'} \rho_j^\alpha v_j(\rho_jx)\le C_0.
\end{gathered}
\end{equation}
Moreover, at the first contact point, by \eqref{eq:bubble-tail} one has
\begin{equation}\label{eq:contact-value-limit}
W_j({\bf e}_j) =\vartheta\rho_j^\alpha U_\tau(\rho_j)
\to \vartheta\kappa_\tau.
\end{equation}

\medskip
\noindent {\bf Step 2: Punctured compactness.}
Let $K\subset\subset B_1\setminus\{0\}$, and choose
$K\subset\subset K'\subset \subset B_1\setminus\{0\}$.  Then \eqref{eq:first-contact-radius}  and \eqref{eq:bubble-tail} give
$$
\sup_{K'}W_j\le C_{K'}.
$$
Furthermore,
$$
\|\widehat g_j(W_j)\|_{L^\infty(K')}
\le\Lambda\epsilon_jC_{K'}^q\to 0.
$$
Therefore, Lemma~\ref{lem:uniform-C1-compactness}  gives precompactness in $C^1(K)$, and up to relabeling the sequence, 
\begin{equation}\label{eq:Wj-punctured-limit}
W_j\to W \quad\text{in }C^1_{\loc}(B_1\setminus\{0\}),
\qquad -\Delta_pW=0.
\end{equation}
Moreover,
\begin{equation}\label{eq:W-upper-pole}
0\le W(x)\le\vartheta\kappa_\tau|x|^{-\alpha}
\quad  \text{for any } \ 0<|x|<1.
\end{equation}

Let
$$
\dd\mu_j:=\widehat g_j(W_j)\dd x.
$$
For $\varphi\in C_c(B_1)$, the change of variables $y=\rho_jx$ together with \eqref{eq:hat-fj} gives
\begin{equation}\label{eq:source-change-variables}
 \int\varphi\dd\mu_j
 =\int_{B_{\rho_j}}
   \varphi(y/\rho_j)g_j(v_j(y))\dd y.
\end{equation}
Inside $B_{\rho_j}$, the first-contact dominance gives
$g_j(v_j)\le\Lambda\vartheta^qU_\tau^q$.  Note that the function $U_\tau^q$ is integrable since
$\alpha q=n+p'>n$.  For every fixed $y$,
$g_j(v_j(y))\to g_\tau(U_\tau(y))$, while
$\varphi(y/\rho_j)\to\varphi(0)$ and
$\mathbf 1_{B_{\rho_j}}(y)\to1$.  Applying dominated convergence to the left-hand side of \eqref{eq:source-change-variables} therefore gives
$$
 \int\varphi\,\dd\mu_j
 \to   \varphi(0)\int_{\R^n}g_\tau(U_\tau)\,\dd y.
$$
Equivalently,
\begin{equation}\label{eq:source-delta-convergence}
\mu_j\weakto\mathcal M_\tau\delta_0
\quad\text{weakly as measures in }B_1.
\end{equation}

For each $j$, testing the weak equation with radial cutoffs
converging to $\mathbf 1_{B_r}$ gives
$$
-\int_{\partial B_r} \Acal(\nabla W_j)\cdot\nu\,\dd S=\mu_j(B_r)
$$
for almost every $r\in(1/4,1/2)$.
Choose $r_0\in(1/4,1/2)$ in the intersection of these
full-measure sets for all $j$.  Since 
$\mathcal M_\tau\delta_0(\partial B_{r_0})=0$,
\eqref{eq:source-delta-convergence} gives
$$
\mu_j(B_{r_0})\to \mathcal M_\tau.
$$
The $C^1$-convergence in a neighborhood of $\partial B_{r_0}$ therefore yields
\begin{equation}\label{eq:limit-pole-flux}
 -\int_{\partial B_{r_0}}\Acal(\nabla W)\cdot\nu\dd \mathcal H^{n-1}  =\mathcal M_\tau.
\end{equation}
The right side of \eqref{eq:limit-pole-flux} is positive, so $W$ is not identically zero.  The strong minimum principle for the nonnegative $p$-harmonic function $W$ gives $W>0$ in $B_1\setminus\{0\}$.  Proposition~\ref{prop:pole-expansion}, applied via \eqref{eq:W-upper-pole}, now gives
\begin{equation}\label{eq:W-pole-expansion}
W(x)=\kappa|x|^{-\alpha}+\beta+o(1).
\end{equation}
Its distributional mass is
$|\mathbb S^{n-1}|(\alpha\kappa)^{p-1}$.  Then comparing with
\eqref{eq:limit-pole-flux} and \eqref{eq:bubble-source-mass} yields
\begin{equation}\label{eq:kappa-identification}
\kappa=\kappa_\tau.
\end{equation}

\medskip
\noindent{\bf Step 3: The Pohozaev sign implies $\beta\le 0$.}
At the same radius $r_0$, Proposition~\ref{prop:local-Pohozaev} and
\eqref{eq:blowdown-equation-bounds} give
$$
\Pscr_{W_j}(r_0)=\int_{B_{r_0}}  \bigl[n\widehat G_j(W_j)
-aW_j\widehat g_j(W_j)\bigr]\dd x \ge0.
$$
On $\partial B_{r_0}$, $W_j$ is uniformly bounded and
$$
0\le\widehat G_j(W_j)
\le\frac{\Lambda}{q+1}\epsilon_jW_j^{q+1}
\to 0
$$
uniformly.  The remaining boundary terms in $\Pscr_{W_j}(r_0)$ converge to the corresponding boundary terms for $W$ by the $C^1$ convergence in a neighborhood of $\partial B_{r_0}$.  Therefore
$$
\Pscr_{W_j}(r_0)\to\Pscr_W(r_0).
$$
Proposition~\ref{prop:pole-expansion} and \eqref{eq:kappa-identification} therefore imply
$$
0\le\Pscr_W(r_0)
=-a\mathcal M_\tau\beta,
$$
so
\begin{equation}\label{eq:beta-nonpositive}
\beta\le0.
\end{equation}

\medskip
\noindent{\bf Step 4: A comparison gives $\beta\ge 0$ together with the exact pole.}
Let
$$
\mathcal R_j:=\frac{2\Gamma_j}{\rho_j}\to \infty.
$$
Fix $\delta\in(0,1/2)$.  By \eqref{eq:W-pole-expansion} and \eqref{eq:kappa-identification}, choose
$r_\delta\in(0,r_0)$ small so that
$$
W\ge(1-\delta/2)\kappa_\tau r_\delta^{-\alpha}
\quad\text{on }\partial B_{r_\delta}.
$$
Then the $C^1$ convergence on this sphere gives that, for all large $j$,
$$
\sup_{\partial B_{r_\delta}}|W_j-W|
\le \frac{\delta}{4}\kappa_\tau r_\delta^{-\alpha},
$$
and hence
$$
W_j\ge\left(1-\frac{3\delta}{4}\right)
\kappa_\tau r_\delta^{-\alpha}
>(1-\delta)\kappa_\tau r_\delta^{-\alpha}
\quad\text{on }\partial B_{r_\delta}.
$$
Define, on $D_{j,\delta}:=B_{\mathcal R_j}\setminus\overline B_{r_\delta}$,
$$
H_{j,\delta}(x)
:=(1-\delta)\kappa_\tau
  \bigl(|x|^{-\alpha}-\mathcal R_j^{-\alpha}\bigr).
$$
This function is $p$-harmonic, is zero on the outer sphere, and is below $W_j$ on both boundary components of $D_{j,\delta}$.  Set
$$
\zeta_j := \max\{H_{j,\delta}-W_j,0\} \in W^{1,p}_0(D_{j,\delta}).
$$
Subtracting the weak equations of $W_j$ and $H_{j,\delta}$, and testing with $\zeta_j$, gives
\begin{align*}
0 &\le\int_{D_{j,\delta}}\widehat g_j(W_j)\zeta_j\,\dd x =
\int_{D_{j,\delta}}\bigl(\Acal(\nabla W_j)-\Acal(\nabla H_{j,\delta})\bigr)\cdot\nabla\zeta_j\,\dd x\\
&=-\int_{\{H_{j,\delta}>W_j\}}
\bigl(\Acal(\nabla W_j)-\Acal(\nabla H_{j,\delta})\bigr)\cdot\bigl(\nabla W_j-\nabla H_{j,\delta}\bigr)\,\dd x \le 0.
\end{align*}
Strict monotonicity of $\Acal$, together with the zero trace of $\zeta_j$, implies $\zeta_j=0$.  Therefore
$$
W_j\ge H_{j,\delta} \qquad\text{in }D_{j,\delta}.
$$
Fix $x\in B_1\setminus\{0\}$ and choose $r_\delta<|x|$.  Passing first to $j\to\infty$ in the comparison at this point removes the term $\mathcal R_j^{-\alpha}$; then letting $\delta\to 0$ yields
\begin{equation}\label{eq:W-above-pole}
W(x)\ge\kappa_\tau|x|^{-\alpha} \quad \text{ for any } \ 0<|x|<1.
\end{equation}
Taking $x\to0$ in the expansion \eqref{eq:W-pole-expansion} shows $\beta\ge 0$.  Together with \eqref{eq:beta-nonpositive}, we get
\begin{equation}\label{eq:beta-zero}
\beta=0.
\end{equation}

We claim that equality holds in \eqref{eq:W-above-pole}.  Indeed, if this is not the case, the  strict comparison principle from Remark~\ref{rem:KV-erratum} gives
$$
W>\Phi_{\kappa_\tau}=\kappa_\tau|x|^{-\alpha}
$$
throughout the punctured ball as $\Phi_{\kappa_\tau}$ has nonvanishing gradient.  Fix
$R_0\in(0,1)$ and set
\begin{equation}\label{eq:m0}
m_0:=\min_{\partial B_{R_0}}  (W-\Phi_{\kappa_\tau})>0,  \qquad  \epsilon_0:=m_0/2.
\end{equation}
For each $\delta>0$, choose 
$$0<s_\delta<\min\{\delta,R_0/2\}$$ 
sufficiently small that 
\begin{equation}\label{eq:inner-sphere}
     W\ge(1-\delta)\Phi_{\kappa_\tau}+\epsilon_0\qquad \text{on } \ \partial B_{s_\delta};
\end{equation}
this is possible since
$W-\Phi_{\kappa_\tau}\to0$ while
$\delta\Phi_{\kappa_\tau}\to\infty$ as $|x|\to 0$.  The same inequality also holds on $\partial B_{R_0}$ by the definition of $\epsilon_0$ in \eqref{eq:m0}.  Now a comparison with the
$p$-harmonic function $(1-\delta)\Phi_{\kappa_\tau}+\epsilon_0$ gives  
$$
W\ge(1-\delta)\Phi_{\kappa_\tau}+\epsilon_0
\qquad\text{in } \ B_{R_0}\setminus\overline B_{s_\delta}.
$$
Fix $x$ with $0<|x|<R_0$. For all sufficiently small $\delta$, we have
$s_\delta<|x|$, and hence
$$
W(x)\ge(1-\delta)\Phi_{\kappa_\tau}(x)+\epsilon_0.
$$
Letting $\delta\to 0$ yields
$$
W(x)\ge\Phi_{\kappa_\tau}(x)+\epsilon_0
\qquad \text{ for any} \ 0<|x|<R_0,
$$
and further letting $x\to0$ contradicts \eqref{eq:beta-zero}.  Hence
\begin{equation}\label{eq:exact-pole-limit}
W(x)=\kappa_\tau|x|^{-\alpha}
\quad \text{ for any } \ 0<|x|<1.
\end{equation}

\medskip
\noindent {\bf Step 5: De-concentration at the moving contact point.}
We first obtain a uniform Morrey bound for the source measures near
${\bf e}_j$.  Fix $r_1\in(0,1/16)$.  Since $\Gamma_j/\rho_j\to\infty$ and ${\bf e}_j\in\partial B_1$, for every $0<r<r_1$ and all large $j$,
$$
B_{6r}({\bf e}_j)\subset\subset B_{4\Gamma_j/\rho_j}(0).
$$
In this ball, $W_j$ is the nonnegative $p$-superharmonic representative solving
$-\Delta_pW_j=\mu_j$, where $\mu_j=\widehat g_j(W_j)\dd x\ge0$.  Then the lower estimate in the distinguished nonlinear potential theorem of Kilpel\"ainen--Mal\'y \cite[Theorem~1.6]{KM1994} together with \eqref{eq:contact-value-limit} gives
$$
\Wolff_{1,p}^{\mu_j}({\bf e}_j;2r) \le C(n, p) W_j({\bf e}_j) \le C(n, p) ,
$$
for constants independent of $j$ and $r$.  By the monotonicity of $t\mapsto\mu_j(B_t({\bf e}_j))$, the part of the Wolff integral over $[r,2r]$ satisfies
$$
\begin{aligned}
 \Wolff_{1,p}^{\mu_j}({\bf e}_j;2r)
 &\ge\int_r^{2r} \left(\frac{\mu_j(B_r({\bf e}_j))}{t^{n-p}}\right)^{1/(p-1)} \frac{\dd t}{t}\\
 &\ge c(n,p) \left(\frac{\mu_j(B_r({\bf e}_j))}{r^{n-p}}\right)^{1/(p-1)}.
\end{aligned}
$$
Combining these two inequalities and raising to the power $p-1$ yields
\begin{equation}\label{eq:contact-Morrey}
\mu_j(B_r({\bf e}_j))\le C_{KM} r^{n-p} \quad \text{ for any } \ 0<r<r_1,
\end{equation}
with constant $C_{KM}$ particularly independent of $j$ and $r$.

Now we define the  intrinsic height
\begin{equation}\label{eq:defn-Thetaj}
\Theta_j(x):=\epsilon_j^{1/p}W_j(x)^{1/a}.
\end{equation}
We prove that $\Theta_j$ is uniformly bounded in a fixed neighborhood of ${\bf e}_j$.   

We first record an auxiliary  estimate.
Suppose that points $z_j$ are selected, and define
\begin{equation}\label{eq:defn-rjint}
\mathsf H_j:=W_j(z_j),
\qquad r_j^{\mathrm{int}} :=\epsilon_j^{-1/p}\mathsf H_j^{-1/a} =\Theta_j(z_j)^{-1}.
\end{equation}
Let us also set
\begin{equation}\label{eq:intrinsic-rescaled-contact}
    V_j(y):=\mathsf H_j^{-1} W_j(z_j+r_j^{\mathrm{int}}y).
\end{equation}
Assume that
$$
V_j(0)=1,
\qquad 0<V_j\le2^a \quad\text{in }B_2.
$$
 Then scaling the equation gives
$$
 -\Delta_pV_j(y) =(r_j^{\mathrm{int}})^p\mathsf H_j^{1-p} \widehat g_j(\mathsf H_jV_j(y)).
$$
As $q-p+1=p/a$ by the third inequality in \eqref{eq:exponent-identities} and $r_j^{\mathrm{int}}=\epsilon_j^{-1/p}\mathsf H_j^{-1/a}$, we conclude
$$
 \epsilon_j(r_j^{\mathrm{int}})^p \mathsf H_j^{q-p+1}=1.
$$
The bounds on the source term $\widehat g_j$ in \eqref{eq:blowdown-equation-bounds} therefore become
\begin{equation}\label{eq:equation-Vj}
LV_j^q\le-\Delta_pV_j\le\Lambda V_j^q.
\end{equation}
Since $V_j\le 2^a$, the equation \eqref{eq:equation-Vj} can be written with a uniformly bounded zeroth-order coefficient.  Then Harnack inequality and $V_j(0)=1$ give $V_j\ge c>0$ on a fixed smaller  ball $B_{\varrho_0}$.  Consequently,
\begin{align}
 \mu_j(B_{\varrho_0r_j^{\mathrm{int}}}(z_j))
 &\ge c\epsilon_j\mathsf H_j^q
(r_j^{\mathrm{int}})^n\notag\\
 &=c\epsilon_j^{1-n/p} \mathsf H_j^{q-n/a}\notag\\
 &=\frac{c}{\epsilon_j^a\mathsf H_j} \ge\frac{c}{C_0}=:m_*>0.
 \label{eq:source-quantum}
\end{align}
Here, $1-n/p=-a$ and $n/a=q+1$, while
$\epsilon_j^a\mathsf H_j=v_j(\rho_jz_j)\le C_0$.  The number $m_*$ has now been fixed before any contact radius is chosen.

Choose $0<r_*<\frac 1 4$ so small that $4r_*<r_1$ and
\begin{equation}\label{eq:rstar-after-quantum}
C_{KM}(2r_*)^{n-p}<m_*.
\end{equation}
Suppose that $\Theta_j$ defined in \eqref{eq:defn-Thetaj} is unbounded in $B_{r_*/2}({\bf e}_j)$.  Then on $\overline B_{r_*}({\bf e}_j)$ we can choose $z_j$ maximizing
\begin{equation}\label{eq:distance-weighted-selection}
d_j(x)\Theta_j(x) 
\qquad\text{ where } \ d_j(x):=r_*-|x-{\bf e}_j|.
\end{equation}
Then one gets $d_j(z_j)\Theta_j(z_j)\to\infty$. Indeed, a point in
$B_{r_*/2}({\bf e}_j)$ with diverging value $\Theta_j$ has distance at least $r_*/2$ from the boundary of $B_{r_*}({\bf e}_j)$.  Since $r_j^{\mathrm{int}}=\Theta_j(z_j)^{-1}$ by \eqref{eq:defn-rjint}, it follows that, for each fixed $R$ and all sufficiently large $j$, we have
$$
R\,r_j^{\mathrm{int}}<\frac{d_j(z_j)}{2}.
$$
Consequently, if $x\in B_{Rr_j^{\mathrm{int}}}(z_j)$, then
$$
d_j(x)\;\ge\;\frac{d_j(z_j)}{2}.
$$
By the maximality property in \eqref{eq:distance-weighted-selection}, this implies
$$
 d_j(x)\,\Theta_j(x)\le d_j(z_j)\,\Theta_j(z_j),
 \qquad\text{and hence}\qquad \Theta_j(x) \le 2\,\Theta_j(z_j).
$$
Raising the second inequality above to the power $a$ and canceling $\epsilon_j^{a/p}$ in \eqref{eq:defn-Thetaj} yields
$$
W_j(x)\le2^a\mathsf H_j
\quad\text{in }B_{Rr_j^{\mathrm{int}}}(z_j).
$$
Thus the rescaling \eqref{eq:intrinsic-rescaled-contact} has exactly the properties used to derive the source quantum
\eqref{eq:source-quantum}.  Since, for fixed $\varrho_0$, the ball $B_{\varrho_0 r_j^{\mathrm{int}}}(z_j)$ is contained in $B_{2r_*}({\bf e}_j)$ for all sufficiently large $j$, as a consequence of
$$
\frac{r_j^{\mathrm{int}}}{d_j(z_j)} \;=\; \frac{1}{d_j(z_j)\,\Theta_j(z_j)} \;\to \; 0,
$$
it follows from \eqref{eq:contact-Morrey}, \eqref{eq:source-quantum}, and \eqref{eq:rstar-after-quantum} that
$$
m_* \;\le\; \mu_j\bigl(B_{2r_*}({\bf e}_j)\bigr)
\;\le\; C_{KM}\,(2r_*)^{n-p}
\;<\; m_*,
$$
which leads to a contradiction.  Hence
\begin{equation}\label{eq:Theta-bounded-contact}
\Theta_j\le C_{\Theta} 
\quad\text{in }B_{r_*/2}({\bf e}_j).
\end{equation}

\medskip
\noindent{\bf Step 6: Compactness through the moving point.}
On $B_{r_*/2}({\bf e}_j)$ write
\begin{equation}\label{eq:equ-Wj}
-\Delta_pW_j=b_j(x)W_j^{p-1}, \qquad b_j(x):=\frac{\widehat g_j(W_j)}{W_j^{p-1}}.
\end{equation}
Since $q-p+1=p/a$, \eqref{eq:Theta-bounded-contact} gives
\begin{equation}\label{eq:uniform-bj}
    0\le b_j\le\Lambda\epsilon_jW_j^{p/a}
=\Lambda\Theta_j^p\le C(\Lambda,C_{\Theta}, p).
\end{equation}
Then the Harnack inequality in
Lemma~\ref{lem:relative-gradient} applies to
\eqref{eq:equ-Wj}. This together with the bounded nonzero value \eqref{eq:contact-value-limit}, bounds on $W_j$ from above and below on a smaller ball.  More explicitly, for a fixed ball $B_{2s}({\bf e}_j)\subset\subset B_{r_*/2}({\bf e}_j)$,
$$
\sup_{B_s({\bf e}_j)}W_j\le C_H\inf_{B_s({\bf e}_j)}W_j,
\qquad
W_j({\bf e}_j)\in[c_0,C_0],
$$
and hence
$$
C_H^{-1}c_0\le W_j\le C_HC_0
\quad\text{in }B_s({\bf e}_j).
$$
Together with \eqref{eq:uniform-bj}, this yields
$$
\|b_jW_j^{p-1}\|_{L^\infty(B_s({\bf e}_j))}\le C(\Lambda,C_{\Theta}, p,c_0,C_H).
$$
Applying Lemma~\ref{lem:uniform-C1-compactness} to the translated functions
$$
\widetilde W_j(y):=W_j({\bf e}_j+y),
$$
up to passing to a subsequence with ${\bf e}_j\to{\bf e}\in \mathbb S^{n-1}$, we conclude that
$$
\widetilde W_j\to\widetilde W
\qquad\text{in }C^1_{\rm loc}(B_{s/2}).
$$ 
Then the bound \eqref{eq:uniform-bj}, the contact value
\eqref{eq:contact-value-limit}, and a finite Harnack chain starting at ${\bf e}_j$ yield the following: For every $r<r_*/2$, there are constants $0<c_r\le C_r<\infty$, independent of $j$, such that
$$
c_r\le W_j\le C_r \quad\text{in }B_r({\bf e}_j).
$$
Hence $b_jW_j^{p-1}$ is uniformly bounded on every compact subset of
$B_{r_*/2}({\bf e}_j)$. Applying
Lemma~\ref{lem:uniform-C1-compactness} to
$$
\widetilde W_j(y):=W_j({\bf e}_j+y)
$$
and using a diagonal subsequence, we obtain
$$
\widetilde W_j\to\widetilde W
\quad\text{in }C^1_{\loc}(B_{r_*/2}).
$$
On the nonempty open set where $|{\bf e}+y|<1$, let $K$ be a compact subset of
$$
\{y\in B_{r_*/4}:|{\bf e}+y|<1\}.
$$
Then, for all sufficiently large $j$, ${\bf e}_j+K$ lies in a compact subset of $B_1\setminus\{0\}$.  Hence
\eqref{eq:Wj-punctured-limit} and \eqref{eq:exact-pole-limit} give
$$
\widetilde W(y) =\kappa_\tau|{\bf e}+y|^{-\alpha} \qquad \text{ for any }\ y\in K.
$$
Continuity at $y=0$ therefore gives
$$
W_j({\bf e}_j)=\widetilde W_j(0) \to \widetilde W(0)=\kappa_\tau.
$$
This contradicts \eqref{eq:contact-value-limit}, since
$\vartheta>1$.  Thus the proposition follows.
\end{proof}

\begin{rem}
\label{rem:no-preliminary-neck}
Let us explain the precise difference between the preceding
first-contact argument and the conditional first-crossing argument of Qin--Zhang in \cite{QZ2026}.

In the proof \cite[Lemma~3.7]{QZ2026}, the preliminary neck estimate
$$
v_j(y)\le C_s|y|^{-a},    \qquad    a=\frac{n-p}{p},
$$
is used to justify an annular Harnack chain at the first-crossing scale.  Indeed, if $r_j$ is the first-crossing radius and
$$
\widetilde w_j(x):=r_j^\alpha v_j(r_jx),    \qquad \alpha=\frac{n-p}{p-1},
$$
then
$$
-\Delta_p\widetilde w_j = b_j(x)\widetilde w_j^{p-1},
$$
where the critical upper growth gives
$$
b_j(x) \le \Lambda r_j^{-p'} \widetilde w_j(x)^{q-p+1}.
$$
On the larger annulus
$$
 D:=\{x\in\R^n:1/3<|x|<3\},
$$
the preliminary estimate implies
$$
\widetilde w_j(x)  \le   C_s r_j^{\alpha-a}|x|^{-a}.
$$
Since $(\alpha-a)(q-p+1)=p',$
the powers of $r_j$ cancel and $b_j$ is uniformly bounded on
$D$.  This is what permits the Harnack chain used in that proof.
The first-crossing estimate itself controls the original variables only up to radius $2r_j$, whereas the larger annulus $D$ reaches radius $3r_j$; it therefore does not by itself supply the outward coefficient control required by that Harnack argument.

The present proof does not attempt to establish such a preliminary bound on the whole neck.  Instead, the first-contact blow-down at the beginning produces punctured compactness and, by the Pohozaev sign together with the comparison and strict-comparison arguments, it identifies the inward limit as the exact pole
$$
W(x)=\kappa_\tau|x|^{-\alpha}.
$$
Near the moving contact point ${\bf e}_j$, the lower 
Wolff-potential estimate \cite[Theorem~1.6]{KM1994} yields
$$
\mu_j(B_r({\bf e}_j))\le Cr^{n-p}.
$$
If the intrinsic quantity
$$
\Theta_j:=  \epsilon_j^{1/p}W_j^{1/a}
$$
is unbounded there, then a distance weighted point selection together with intrinsic rescaling produces  $m_*>0$.  For a sufficiently small fixed radius this would
contradict the earlier Morrey estimate.  Hence
$$
    \Theta_j\le C \quad \text{ near }\ {\bf e}_j,
$$
and consequently
$$
\frac{\widehat g_j(W_j)}{W_j^{p-1}} \le \Lambda\epsilon_jW_j^{q-p+1} =\Lambda\Theta_j^p \le C.
$$
Only after this coefficient bound has been derived do we apply
Harnack across the moving contact point.  Thus the coefficient
control formerly supplied by the preliminary neck estimate is here obtained from a de-concentration argument based on potential theory rather than
assumed.
\end{rem}

\subsection{Proof of the Harnack inequality and the full Liouville classification}\label{sec:final-proof}
\begin{proof}[Proof of Theorem~\ref{thm:main}]
Assume that the theorem is false.  Up to translating the centers, there are radii $R_j>0$ and nonnegative weak solutions
$$
-\Delta_pu_j=g(u_j) \quad\text{in }B_{3R_j}
$$
such that, by denoting
$$
M_j:=\max_{\overline B_{R_j}}u_j,
\qquad m_j:=\min_{\overline B_{2R_j}}u_j,
$$
one has
\begin{equation}\label{eq:RMm-infty}
R_j^{n-p}M_jm_j^{p-1}\to \infty.
\end{equation}
The regularity results of Section~\ref{sec:regularity} justify the existence of the maxima and minima.  Since a nontrivial nonnegative solution is positive in the connected ball, we conclude that $m_j>0$.  In fact, solutions which vanish identically already satisfy the desired estimate.

Recalling \eqref{eq:RMm-infty} and applying Proposition~\ref{prop:point-selection} to the sequence,  we obtain sequences
$x_j,A_j,\delta_j,\sigma_j,\Gamma_j,v_j$, and
\begin{equation}\label{eq:defn-Xij}
\Xi_j:=A_j\sigma_j^{n-p}m_j^{p-1}\to\infty
\quad \text{ and } \quad \Gamma_j^{n-p}\ge\Xi_j.
\end{equation}
Then, after passing to a subsequence,
$v_j\to U_\tau$ in $C^1_{\loc}(\R^n)$,
and Theorem~\ref{thm:frozen-classification} identifies $U_\tau$ with an explicit bubble.

To apply Proposition~\ref{prop:first-contact}, we take $g_j=g_{A_j}$ and verify its hypotheses.  The two-sided critical  bounds on $g_j$ follow from $L\le h\le\Lambda$, the (scaled) Pohozaev sign is proved in \eqref{eq:pohozaev-sign}, the local convergence 
$$
g_j\to g_\tau
\quad\text{locally uniformly on }(0,\infty).
$$
is given by Lemma~\ref{lem:frozen-convergence}, and \eqref{eq:vj-normalization} together with \eqref{eq:vj-limit} supplies the uniform bound and normalization.  Thus every hypothesis of Proposition~\ref{prop:first-contact} has been verified, and we are allowed to apply it.

Define
\begin{equation}\label{eq:defn-etaj}
\eta_j:=\Xi_j^{-1/[2(n-p)]}, \qquad \widehat\Gamma_j:=\eta_j\Gamma_j.
\end{equation}
Then $\eta_j\to0$ and, by \eqref{eq:defn-Xij}
\begin{equation}\label{eq:defn-hatXij}
\widehat\Gamma_j^{n-p} =\Xi_j^{-1/2}\Gamma_j^{n-p}
\ge\Xi_j^{1/2}\to \infty \quad \text{ and }
\quad 2\widehat\Gamma_j=o(\Gamma_j).
\end{equation}
Apply Proposition~\ref{prop:first-contact} with any fixed
$\vartheta>1$ and $S_j=2\widehat\Gamma_j$.  For every large $j$, we have
\begin{equation*}
v_j\le\vartheta U_\tau
\quad\text{in }B_{2\widehat\Gamma_j}.
\end{equation*}
Then the estimate \eqref{eq:estimate-U} on the explicit bubble and \eqref{eq:bubble-tail} give, on
$|y|=\widehat\Gamma_j$,
\begin{equation}\label{eq:adaptive-sphere-tail}
v_j(y)\le C(n,p,\kappa_\tau)\widehat\Gamma_j^{-\alpha}.
\end{equation}

Recall   $\Gamma_j=\sigma_j/\delta_j$ in \eqref{eq:defn-Gammaj}. Then by \eqref{eq:defn-hatXij}
\begin{equation}\label{eq:physical-adaptive-radius}
\delta_j\widehat\Gamma_j =\eta_j\delta_j\Gamma_j =\eta_j\sigma_j.
\end{equation}
Recall also that the selected point has distance
$D_j=8\sigma_j$ from $\partial B_{2R_j}$.  Since $\eta_j<1$ for large $j$ according to \eqref{eq:defn-etaj}, the entire sphere with radius
\eqref{eq:physical-adaptive-radius} lies inside $B_{2R_j}$.  Therefore the definition of $m_j$, the definition of $v_j$ in \eqref{eq:defn-vj} and \eqref{eq:adaptive-sphere-tail} imply
\begin{equation*}
m_j\le C(n,p,\kappa_\tau) A_j\widehat\Gamma_j^{-\alpha}
= C(n,p,\kappa_\tau)  A_j\eta_j^{-\alpha}\Gamma_j^{-\alpha}.
\end{equation*}
Raising this inequality to the power $p-1$, and using the fourth identity in \eqref{eq:exponent-identities} together with
\begin{equation*}
\Gamma_j^{n-p}=\sigma_j^{n-p}A_j^p,
\end{equation*}
which follows from \eqref{eq:defn-Gammaj}, we obtain
$$
m_j^{p-1} \le C A_j^{p-1}\eta_j^{-(n-p)} \sigma_j^{-(n-p)}A_j^{-p}
=C\eta_j^{-(n-p)}\sigma_j^{-(n-p)}A_j^{-1}.
$$
Multiplication by $A_j\sigma_j^{n-p}$ and using \eqref{eq:defn-Xij} yield
\begin{equation*}
\Xi_j\le C\eta_j^{-(n-p)}=C\Xi_j^{1/2}.
\end{equation*}
This contradicts $\Xi_j\to\infty$.  Hence
\eqref{eq:main-harnack} holds and we conclude the theorem.
\end{proof}

\begin{proof}[Proof of Corollary~\ref{cor:full-liouville}]
Recall that $ \alpha:=\frac{n-p}{p-1}.$
By the local regularity established in
Section~\ref{sec:regularity}, $u$ is continuous.  Let
$m_1:=\min_{\partial B_1}u>0$.  For $S>1$, define
$$
 \psi_S(x):=
 m_1\frac{|x|^{-\alpha}-S^{-\alpha}}{1-S^{-\alpha}}
 \qquad\text{in }B_S\setminus\overline{B_1}.
$$
The function $\psi_S$ is $p$-harmonic, equals $m_1$ on $\partial B_1$, and vanishes on $\partial B_S$.  

Since $-\Delta_p u>0$, applying the comparison principle gives
$u\ge\psi_S$.  Letting $S\to\infty$ and also using the positive minimum of $u$ on $\overline{B_1}$, we obtain a constant $c_0>0$ such that
\begin{equation}\label{eq:entire-lower-bound}
 \inf_{B_{2R}}u\ge c_0R^{-\alpha}  \qquad\text{for every }R\ge1.
\end{equation}
Applying Theorem~\ref{thm:main} in $B_{3R}$ and using
$\alpha(p-1)=n-p$, we get
$$
 \sup_{B_R}u \le C R^{p-n}\left(\inf_{B_{2R}}u\right)^{1-p}
 \le Cc_0^{1-p}.
$$
Hence
$$
 M:=\sup_{\R^n}u<\infty.
$$
The same Harnack estimate, together with $\sup_{B_R}u\ge u(0)>0$, gives first
$$
\inf_{B_{2R}}u\le C_uR^{-\alpha}.
$$
Then replacing $R$ by $R/2$ with $R\ge2$, and enlarging $C_u$ if necessary, we obtain the sharp infimum bound  
\begin{equation}\label{eq:entire-sharp-infimum}
 \inf_{B_R}u\le C_u R^{-\alpha}  \qquad\text{for every }R\ge1.
\end{equation}

Choose $x_j\in\R^n$ with $u(x_j)\to M$, set
$$
 \rho:=M^{-p/(n-p)}, \qquad V_j(y):=M^{-1}u(x_j+\rho y),
$$
and observe that
$$
 0<V_j\le1,  \qquad V_j(0)\to1,
 \qquad  -\Delta_pV_j=V_j^qh(MV_j)=g_M(V_j)
 \quad\text{in }\R^n.
 $$
Then for every fixed $R>0$,
$$
0<V_j\le1, \qquad \|g_M(V_j)\|_{L^\infty(B_{2R})}\le\Lambda.
$$
Therefore, up to relabeling the sequence,  Lemma~\ref{lem:uniform-C1-compactness}  
gives that
$$
V_j\to V\qquad\text{in }C^1_{\rm loc}(\R^n).
$$
Passing the weak formulation to the limit also gives
$$
-\Delta_pV=g_M(V), \qquad 0<V\le1, \qquad V(0)=1. 
$$  Applying
Theorem~\ref{thm:frozen-classification} at the finite amplitude $\tau=M$ yields
$$
 h(s)=h(M)\qquad \text{ for every } \ 0<s\le M.
$$
Since $0<u\le M$, the original equation therefore reduces globally to
$$
 -\Delta_pu=h(M)u^q.
$$
Define
$$
\widetilde u:=h(M)^{1/(q-p+1)}u.
$$
Then
$$
-\Delta_p\widetilde u=\widetilde u^q.
$$ 
The function $\widetilde u$ is bounded and satisfies the sharp infimum estimate corresponding to \eqref{eq:entire-sharp-infimum}.  Therefore all the conditions of
Ciraolo--Gatti \cite[Theorem~1.2]{CG2026} are satisfied; see also  \cite[Theorem~1.1]{O2025} for the case $n=2$ and $1<p<2$.
Consequently, $\widetilde u$, and thus $u$, coincides with an Aubin–Talenti profile. By matching its maximal value and the coefficient $h(M)$, we obtain the identity \eqref{eq:full-liouville-profile}.
\end{proof}

\appendix
\section{A weak local Pohozaev identity}

\begin{proof}[Proof of Proposition~\ref{prop:local-Pohozaev}]
We use the compactly supported variational identity of
Degiovanni--Musesti--Squassina
\cite[Lemma~1, formula~(3)]{DMS2003}. See also the alternative proof by
direct domain variations due to Yan and Zhou
\cite[Theorem~1.1]{YZ2026}. For
$$
\mathcal L(s,\xi)=\frac1p|\xi|^p-F(s),
$$
the map $\xi\mapsto\mathcal L(s,\xi)$ is strictly convex for $p>1$.
Since $u\in C^1(B_R)$, it is locally Lipschitz. By taking the compactly supported inner variation generated by $H\in C_c^1(B_R;\mathbb R^n)$, we obtain that
\begin{equation}\label{eq:local-PS-identity}
\int_{B_R} \left[(\operatorname{div}H)
\left(\frac1p|\nabla u|^p-F(u)\right)
-|\nabla u|^{p-2}u_i u_j\partial_iH_j
\right]\dd x=0.
\end{equation}

Fix $r\in(0,R)$.  Choose $\epsilon>0$ with $r+\epsilon<R$, and choose $\zeta\in C^\infty(\mathbb R)$ such that
$$
\zeta(t)=1\ \text{ when } \ t\le 0,\qquad
\zeta(t)=0\ \text{ when } \ t\ge 1,\qquad \zeta'\le 0,
\qquad \int_0^1\zeta'(t)\dd t=-1.
$$
Also define
$$
\chi_{\epsilon,r}(x) :=\zeta\left(\frac{|x|-r}{\epsilon}\right),\qquad H_{\epsilon,r}(x):=\chi_{\epsilon,r}(x)x.
$$
Then
$$
\operatorname{div}H_{\epsilon,r}
=n\chi_{\epsilon,r}+|x|\partial_\rho\chi_{\epsilon,r},
\qquad \partial_i(H_{\epsilon,r})_j
=\chi_{\epsilon,r}\delta_{ij}
+\partial_\rho\chi_{\epsilon,r}\frac{x_ix_j}{|x|},
$$
and substitution in \eqref{eq:local-PS-identity} gives
\begin{align}
0={}&
\int_{B_R}\chi_{\epsilon,r}
\bigl[a|\nabla u|^p-nF(u)\bigr]\dd x \notag \\
&\quad +\int_{B_R}|x|\partial_\rho\chi_{\epsilon,r}
\left[ \frac1p|\nabla u|^p-F(u) -|\nabla u|^{p-2}u_\nu^2 \right]\dd x. \label{eq:inner-test}
\end{align}
Letting  $\epsilon\to 0$,
the first integral in \eqref{eq:inner-test} tends to
$$
\int_{B_r}\bigl[a|\nabla u|^p-nF(u)\bigr]\dd x. 
$$
As for the second integral, the coarea formula gives
\begin{align*}
&\int_{B_R}|x|\partial_\rho\chi_{\epsilon,r}
\left[\frac1p|\nabla u|^p-F(u)
-|\nabla u|^{p-2}u_\nu^2 \right]\dd x\\
=&\int_0^1(r+\epsilon s)\zeta'(s)
\int_{\partial B_{r+\epsilon s}}
\left[\frac1p|\nabla u|^p-F(u)-|\nabla u|^{p-2}u_\nu^2 \right]\dd \mathcal H^{n-1}\dd s.
\end{align*}
Since $u\in C^1$ and $F(u)$ is continuous, when $\epsilon\to 0$, this tends to
$$
-r\int_{\partial B_r}\left[\frac1p|\nabla u|^p-F(u)
-|\nabla u|^{p-2}u_\nu^2\right]\dd \mathcal H^{n-1}.
$$
Thus
\begin{equation}\label{eq:pohozaev-boundary-first}
 \int_{\partial B_r}
\left[ r|\nabla u|^{p-2}u_\nu^2-\frac rp|\nabla u|^p+rF(u)
\right]\dd \mathcal H^{n-1} = \int_{B_r}\bigl[nF(u)-a|\nabla u|^p\bigr]\dd x.
\end{equation}

We next test the weak equation with
$$
\varphi_{\epsilon,r}:=u\chi_{\epsilon,r}.
$$
Since $u\in C^1(B_R)\cap W^{1,p}(B_R)$ and $\chi_{\epsilon,r}$ is Lipschitz with compact support in $B_R$, this test can be obtained by smoothing $\chi_{\epsilon,r}$ and passing to the limit in $W^{1,p}$.  Then the weak equation gives
$$
\int_{B_R}\chi_{\epsilon,r}|\nabla u|^p\dd x +\int_{B_R}u|\nabla u|^{p-2}\nabla u\cdot\nabla\chi_{\epsilon,r}\dd x
= \int_{B_R}\chi_{\epsilon,r}u f(u)\dd x.
$$
The first integral  on the left-hand side and the right-hand side  integral converge, respectively, to the corresponding integrals over $B_r$ as $\epsilon\to 0$.  As for the second term on the left-hand side, via coarea formula again
\begin{multline*}
\int_{B_R}u|\nabla u|^{p-2}\nabla u\cdot\nabla\chi_{\epsilon,r}\dd x=
\int_0^1\zeta'(s) \int_{\partial B_{r+\epsilon s}} u|\nabla u|^{p-2}u_\nu\,\dd \mathcal H^{n-1}\,\dd s\\
\to -\int_{\partial B_r}u|\nabla u|^{p-2}u_\nu\dd \mathcal H^{n-1}.
\end{multline*}
Therefore
\begin{equation}\label{eq:energy-ball-with-boundary}
\int_{B_r}|\nabla u|^p\dd x
=\int_{B_r}u f(u)\dd x
 +\int_{\partial B_r}u|\nabla u|^{p-2}u_\nu\dd \mathcal H^{n-1}.
\end{equation}
Combining \eqref{eq:pohozaev-boundary-first} and \eqref{eq:energy-ball-with-boundary} we arrive at \eqref{eq:local-Pohozaev}.  

As for the annular identity \eqref{eq:annular-Pohozaev}, let $0<r<s<R$. Choose a radial cutoff equal to one on
$B_s\setminus\overline B_r$, with a transition layer near each connected component of the boundary. The outer layer gives $\Pscr_u(s)$, while the inner layer gives $-\Pscr_u(r)$ since the outward normal of the annulus on
$\partial B_r$ is $-\nu$. Passing to the limit yields
\eqref{eq:annular-Pohozaev} with the stated signs.
\end{proof}

\section{A concise outline for the case \texorpdfstring{$p=2$ and $h\equiv1$}{p=2 and h=1}}
\label{appen:B}

Let us explain the main idea for the case when $p=2$ and $h\equiv 1$.

\subsection{The essential argument for Theorem~\ref{thm:frozen-classification}}
We first consider $w=V^{\frac{2}{2-n}}$. Then according to the assumption of Theorem~\ref{thm:frozen-classification},
$$
w(0)=1, \qquad \nabla w(0)=0, \qquad w\ge 1, 
$$
and the equation $-\Delta V=V^{\frac{n+2}{n-2}}$ gives
$$
w\Delta w = \frac n 2 |Dw|^2 + \frac{2}{n-2}. 
$$
By setting 
$$
c_n=\frac {2}{n-2}\quad \text{ and } \quad P=\Delta w,
$$
we arrive at
\begin{equation}\label{eq:p2-equ1}
    wP=\frac n 2|Dw|^2 + c_n
\end{equation}
with $P(0)=c_n$. This corresponds to  \eqref{eq:transformed-equation}.

Now define 
$$
\mathsf E:= D^2w-\frac P n\Id. 
$$
By differentiating \eqref{eq:p2-equ1} one has
$$
P\nabla w + w\nabla P = n D^2wDw.
$$
Thus
\begin{equation}\label{eq:p2-equ2}
    w\nabla P= n\mathsf  E\nabla w;
\end{equation}
see \eqref{eq:Ou-form}. 
Moreover, note that
$$
{\rm div} (\mathsf E)= \nabla \Delta w - \frac 1 n \nabla P=\frac{n-1}{n}\nabla P,
$$
and, since $\mathsf E$ is trace-free, 
$$
\mathsf E_{ij} w_{ji}= |\mathsf E|^2  + \frac P n {\rm tr} \mathsf E= |\mathsf E|^2.
$$
Therefore, 
\begin{equation}\label{eq:p2-equ3} 
{\rm div} (\mathsf E\nabla w)={\rm div}(\mathsf E)\cdot \nabla w + \mathsf E_{ij} w_{ji}=|\mathsf E|^2+ \frac{n-1}{n} \nabla P\cdot\nabla w;
\end{equation}
see \eqref{eq:div-EX}.

Observe that \eqref{eq:p2-equ2} yields
$$
w^{2-n}\nabla P= nw^{1-n} \mathsf E \nabla w.
$$
Consequently, by \eqref{eq:p2-equ2}  and \eqref{eq:p2-equ3}, 
\begin{align*}
{\rm div}(w^{2-n}\nabla P) 
={} &  n \, {\rm div}(w^{1-n}\mathsf E \nabla w)\\
={} &  n w^{1-n}\, {\rm div}(\mathsf E \nabla w) + n(1-n) w^{-n} \nabla w\cdot \mathsf E\nabla w\\
={} & n w^{1-n}|\mathsf E|^2 + (n-1)w^{1-n}(\nabla P\cdot \nabla w) + (1-n)w^{1-n}(\nabla P\cdot \nabla w),
\end{align*}
where the last two terms cancel each other. Thus we arrive at 
\begin{equation}\label{eq:p2-bochner}
\Lscr_w P:= {\rm div}(w^{2-n}\nabla P)= nw^{1-n}|\mathsf E|^2\ge 0;
\end{equation}
see \eqref{eq:bochner}. 

Now we need to prove the essential global bound\footnote{Currently the author does not see any simpler argument for this part, even when $p=2$. It would be of interest to determine whether there exists a more elementary proof for this which does not rely on the method of moving spheres.}
\begin{equation}\label{eq:p2-Pbound}
    P\le c_n;
\end{equation}
comparing with \eqref{eq:Qw-upper}.
Fix $\ell>c_n$ and $\mathfrak m>0$, let
$$
J_\ell(t)=\ell-\frac{c_n} t,\quad Y_{\ell,\mathfrak m}(t)= \mathfrak m t^{-\frac{n-2}{2}} \left(\frac{J_\ell(t)}{\ell-c_n}\right)^\frac n 2,\quad \theta_{\ell, \mathfrak m}(t) =\frac{n-1}{1+Y_{\ell,\mathfrak m}(t)}.
$$
Then the corresponding  multiplier is 
$$
\phi_{\ell,\mathfrak m}(t)= \exp\left(- \int_1^t\frac{\theta_{\ell, \mathfrak m}(s)}{s}\,ds\right).
$$

Consider
$$
\Mscr_{\ell,\mathfrak m}:=P\phi_{\ell,\mathfrak m}(w).
$$
Following the argument of Proposition~\ref{prop:bochner-lower}, we
obtain from \eqref{eq:p2-bochner}
\begin{multline}\label{eq:p2-lowerbound}
  {\rm div}(w^{2-n}\nabla \Mscr_{\ell, \mathfrak m})- {\mathbf b}_{\ell,\mathfrak m}\cdot \nabla \Mscr_{\ell, \mathfrak m}\\
  \ge \phi_{\ell, \mathfrak m}(w) w^{1-n}\left[\frac{\zeta_{\ell,\mathfrak m}^2}{n-1} + \mathfrak a_{\ell,\mathfrak m}(w) P (P-\ell)\right]  \quad \text{ on } \{\Mscr_{\ell, \mathfrak m} > \ell\}
\end{multline}
where $\mathfrak a_{\ell,\mathfrak m}(w)>0$. Consequently, $\Mscr_{\ell, \mathfrak m}$  cannot attain an interior maximum strictly larger than $\ell$; otherwise $P-\ell>0$ and then 
$${\rm div}(w^{2-n}\nabla \Mscr_{\ell, \mathfrak m})>0$$
at the interior maximum point, which leads to a contradiction. 
 
The remaining issue is a possible escape to infinity. Note that, the standard local gradient estimates on $V$ give
$$
|Dw|\le C(n) w,
$$
and then $P=O(w)$ by \eqref{eq:p2-equ1}.  Now for fixed 
$\mathfrak m$, one has
$$
Y_{\ell,\mathfrak m}(t)\to 0,\qquad \theta_{\ell,\mathfrak m}(t)\to n-1 \quad \text{as } \ t\to \infty.
$$
Thus $\phi_{\ell,\mathfrak m}\sim t^{-(n-1)}$ as $t\to \infty$ by \eqref{eq:defn-thetalm}. Hence, since $P=O(w)$,
$$
P\phi_{\ell,\mathfrak m}(w)=O(w^{-(n-2)})\to 0 \quad \text{ along every sequence on which } w\to \infty. 
$$
Now if a maximizing sequence has bounded $w$-values strictly larger than $\ell$, translating the solution and applying local elliptic compactness, one obtains an entire limiting solution, which satisfies the same equation as $w$ on which the maximum is attained. Then equation \eqref{eq:p2-lowerbound} contradicts the local maximum principle. Hence
$$
P\phi_{\ell,\mathfrak m}(w)\le \ell. 
$$
Observe that for every $t<\infty$ fixed, one has
$$
\phi_{\ell,\mathfrak m}(t)\to 1 \quad \text{ as} \ \mathfrak m\to \infty.
$$
Thus $P\le \ell$, and  letting $\ell\to (c_n)^+$ yields \eqref{eq:p2-Pbound}. 

We now have 
$$
P\le c_n \quad \text{ and } P(0)=c_n.
$$
Unlike the general $p$-case, the operator 
$$
\Lscr_w  := {\rm div}(w^{2-n}\nabla \cdot)
$$
does not degenerate at points where $\nabla w=0$. 
On every bounded set its coefficient $w^{2-n}$ is bounded above and below by positive constants.
By \eqref{eq:p2-bochner}, since $P$ attains its global maximum at the interior point $0$, the strong maximum principle immediately gives $P\equiv c_n$. This is a major simplification over the general proof, where the critical point analysis in Lemma~\ref{lem:unique-tangent}, the punctured-ball barrier in
Lemma~\ref{lem:local-Z-vanishing}, and the final open-and-closed argument in Proposition~\ref{prop:global-rigidity} are no longer needed. Plugging this into \eqref{eq:p2-bochner}, we conclude that $\mathsf E\equiv 0$, and the rest is standard. 

\subsection{The essential argument for Theorem~\ref{thm:main}}

Suppose the estimate fails. Then the normalization argument in
Proposition~\ref{prop:point-selection} gives
$v_j\to V$ locally, where $V$ satisfies the hypotheses of
Theorem~\ref{thm:frozen-classification} and hence is a Talenti bubble $U$.

We now explain the proof of Proposition~\ref{prop:first-contact} when $p=2$. We need to show that, for every $\vartheta>1$ and every
$S_j=o(\Gamma_j)$, one has
$$
v_j< \vartheta U\quad \text{ in  } \ B_{S_j}
$$
for $j$ sufficiently large, where $\Gamma_j\to \infty$ is defined in Proposition~\ref{prop:point-selection}. 

Assume instead that the first contact occurs at radius $\rho_j\to \infty$. Then 
$$
v_j\le \vartheta U\quad \text{ in  } \ B_{\rho_j},\qquad v_j(\rho_j{\mathbf e}_j)= \vartheta U(\rho_j), \quad {\mathbf e}_j\in \mathbb S^{n-1}. 
$$
Since $\rho_j=o(\Gamma_j)$, define the blow-down sequence
$$
W_j(x)= \rho_j^{n-2} v_j(\rho_jx),\qquad \epsilon_j=\rho_j^{-2}
$$
Then
$$
-\Delta W_j =\epsilon_j W_j^{\frac{n+2}{n-2}} 
$$
with 
$$
W_j(\mathbf{e}_j) \to\vartheta \kappa \quad \text{ and } \quad \kappa=[n(n-2)]^{\frac {(n-2)} 2}
$$
as $r^{n-2}U(r)\to \kappa$. 

Moreover, on the punctured ball $B_1\setminus\{0\}$, 
$W_j\to W$ for some harmonic function $W$ satisfying
$$
0<W(x)\le \vartheta\kappa|x|^{2-n}.
$$
The source measures $\mu_j:=\epsilon_j W_j^{\frac {n+2}{n-2}}\dd x$ converge weakly to the concentrated bubble mass $\mathcal M\delta_0$, where
$$\mathcal M=\int_{\mathbb R^n} U^{\frac{n+2}{n-2}}\dd x$$
Then by  the classical B\^ocher theorem,
\begin{equation}\label{eq:expression-W-p2}
    W(x) = \kappa |x|^{2-n} +H(x), 
\end{equation}
where $H$ is harmonic in $B_1$. 

In the pure critical power case, for every $j$, the critical homogeneity gives $\Pscr_{W_j}(r)=0$ for every admissible $r$. Passing to the limit on a fixed sphere gives $\Pscr_{W}(r)=0$. Now by the expression of $W$ in  \eqref{eq:expression-W-p2}, 
Proposition~\ref{prop:pole-expansion} with $p=2$ yields 
$$\Pscr_W(r)=-\frac{n-2} 2\mathcal M H(0),$$
and then $H(0)=0$ as $\mathcal M>0$. On the other hand, a comparison with truncated fundamental solutions gives $W\ge \kappa|x|^{2-n}$, which tells $H\ge 0$. As $H$ is harmonic and $H(0)=0$, one gets indeed $H\equiv 0$. Thus the  blow-down limit is exactly 
\begin{equation}\label{eq:p2-pole}
    W(x)= \kappa |x|^{2-n}. 
\end{equation}

However, the convergence above is only known away from the origin and does not initially include the moving point $\mathbf{e}_j\in\mathbb S^{n-1}$. 
For $p=2$, the required source de-concentration in Step 5 in the proof of Proposition~\ref{prop:first-contact} follows from an elementary Green-function estimate: Since
$$
-\Delta W_j = \mu_j\ge 0,
$$
the Green representation $G_{B_{2r}(\mathbf{e}_j)}$ in the ball $B_{2r}(\mathbf{e}_j)$ yields
$$
W_j(\mathbf{e}_j)\ge \int_{B_{2r}(\mathbf{e}_j)}G_{B_{2r}(\mathbf{e}_j)}(\mathbf{e}_j,y) \,d\mu_j(y)\ge c(n) r^{2-n}\mu_j(B_r(\mathbf{e}_j)). 
$$
Then, since $W_j(\mathbf{e}_j)$ stays bounded
\begin{equation}\label{eq:p2-morrey}
\mu_j(B_r(\mathbf{e}_j))\le C_{KM} r^{n-2}, 
\end{equation}
where the constant $C_{KM}$ is independent of $j$ and $r$.

On the other hand, define the intrinsic height
$$\Theta_j(x)=\epsilon_j^{\frac 1 2} W_j^{\frac 2{n-2}}.$$
If $\Theta_j$ is unbounded near $\mathbf{e}_j$, a suitable selection produce points $z_j$, heights $H_j=W_j(z_j)$ and intrinsic radii
$$r_j^{\rm int}=\epsilon_j^{-\frac 1 2}H_j^{-\frac 2{n-2}} =\Theta_j(z_j)^{-1}\to 0.$$
Then the normalized functions 
$$
V_j(y)=H_j^{-1}W_j(z_j+r_j^{\rm int} y )
$$ 
satisfy 
$$-\Delta V_j = V_j^{\frac{n+2}{n-2}},\qquad V_j(0)=1, $$
and are uniformly bounded from above on a fixed ball. Harnack inequality then gives a positive lower bound near the origin. Consequently, the original source measure $\mu_j$ contains a fixed positive mass
$$
\mu_j(B_{cr_j^{\rm int}}(z_j))\ge m_*>0.
$$
This contradicts the Morrey estimate \eqref{eq:p2-morrey} after choosing the radius $r$ sufficiently small. Hence $\Theta_j$ is bounded near $\mathbf {e}_j$, and now one can apply Harnack and interior elliptic estimates to give compactness through the moving point. As \eqref{eq:p2-pole} yields
$$W_j(\mathbf{e}_j)\to \kappa$$
while the contact condition implies
$$W_j(\mathbf{e}_j)\to \vartheta\kappa,$$
we arrive at a contradiction and conclude the first-contact proposition. Once first-contact proposition is available, the remaining argument is basically algebraic calculation.


\begin{thebibliography}{99}

\bibitem{AD1990}
L.~Ambrosio and G.~Dal Maso,
\emph{A general chain rule for distributional derivatives},
Proc. Amer. Math. Soc. \textbf{108} (1990), 691--702.

\bibitem{ACF2023}
C.~A. Antonini, G.~Ciraolo, and A.~Farina,
\emph{Interior regularity results for inhomogeneous anisotropic quasilinear equations},
Math. Ann. \textbf{387} (2023), 1745--1776.

\bibitem{BE1991}
G. Bianchi, H. Egnell, \emph{A note on the Sobolev inequality}. J. Funct. Anal.
\textbf{100} (1991), 18--24.

\bibitem{CGS1989}
L. A. Caffarelli, B. Gidas, J. Spruck, 
\emph{Asymptotic symmetry and local behavior of semilinear elliptic equations with critical Sobolev growth}.
Comm. Pure Appl. Math. \textbf{42} (1989), no. 3, 271--297.

\bibitem{CMR2023}
G.~Catino, D.~D. Monticelli, and A.~Roncoroni,
\emph{On the critical $p$-Laplace equation},
Adv. Math. \textbf{433} (2023), Paper No.~109331, 38 pp.

\bibitem{CDGL2026}
L.~Chen, W.~Dai, C.~Gui, and Y.~Luo,
\emph{Liouville theorems for $p$-Laplacian equations in convex cones without finite-energy condition},
arXiv:2605.29281 (2026).

\bibitem{CL1991}
W.~Chen and C.~Li,
\emph{Classification of solutions of some nonlinear elliptic equations},
Duke Math. J. \textbf{63} (1991), no.~3, 615--622.

\bibitem{CFMP2009}
A. Cianchi, N. Fusco, F. Maggi, A. Pratelli, \emph{The sharp Sobolev inequality in quantitative form}. Journal of the European Mathematical Society
\textbf{11} (2009), 1105--1139. 

\bibitem{CFP2025}
G.~Ciraolo, A.~Farina, and C.~C. Polvara,
\emph{Classification results, rigidity theorems and semilinear PDEs on
Riemannian manifolds: a $P$-function approach},
J. Eur. Math. Soc. (2025), published online first.

\bibitem{CFM2018}
G. Ciraolo, A. Figalli, F. Maggi, \emph{A quantitative analysis of metrics on $\mathbb R^n$ with almost constant positive scalar curvature, with applications to fast diffusion flows}. Int. Math. Res. Not. IMRN 2018, no. \textbf{21}, 6780–6797.

\bibitem{CFR2020}
G.~Ciraolo, A.~Figalli, and A.~Roncoroni,
\emph{Symmetry results for critical anisotropic $p$-Laplacian equations in convex cones},
Geom. Funct. Anal. \textbf{30} (2020), no.~3, 770--803.


\bibitem{CG2026}
G.~Ciraolo and M.~Gatti,
\emph{Classification results for bounded positive solutions to the critical $p$-Laplace equation},
Nonlinear Anal. \textbf{272} (2026), Article 114190. 

\bibitem{CG20262}
G.~Ciraolo and M.~Gatti,
\emph{On the stability of the critical p-Laplace equation},
J. Funct. Anal. 291 (2026), no. 7, Paper No. 111575.

\bibitem{DMMS2014}
L.~Damascelli, S.~Merch\'an, L.~Montoro, and B.~Sciunzi,
\emph{Radial symmetry and applications for a problem involving the
$-\Delta_p(\cdot)$ operator and critical nonlinearity in $\mathbb R^N$},
Adv. Math. \textbf{265} (2014), 313--335.

\bibitem{DMS2003}
M.~Degiovanni, A.~Musesti, and M.~Squassina,
\emph{On the regularity of solutions in the Pucci--Serrin identity},
Calc. Var. Partial Differential Equations \textbf{18} (2003), 317--334.

\bibitem{DSW2025}
B. Deng, L. Sun, J. Wei, \emph{Sharp quantitative estimates of Struwe's decomposition}. Duke Math. J.
\textbf{174} (2025), 159--228.

\bibitem{D1983}
E.~DiBenedetto,
\emph{$C^{1+\alpha}$ local regularity of weak solutions of degenerate elliptic equations},
Nonlinear Anal. \textbf{7} (1983), no.~8, 827--850.


\bibitem{EG2015}
L.~C. Evans and R.~F. Gariepy,
\emph{Measure Theory and Fine Properties of Functions},
revised edition, CRC Press, Boca Raton, 2015.

\bibitem{FG2020}
A. Figalli, F. Glaudo, \emph{On the sharp stability of critical points of the Sobolev inequality}. Arch. Ration. Mech. Anal. 
\textbf{237} (2020), 201--258.

\bibitem{FN2019}
A. Figalli, R. Neumayer, \emph{Gradient stability for the Sobolev inequality: the case $p\geq 2$}. J. Eur. Math. Soc. (JEMS) \textbf{21} (2019), no. 2, 319--354.

\bibitem{FZ2022}
A. Figalli, Y. Zhang, \emph{Sharp gradient stability for the Sobolev inequality}. Duke Math. J.  \textbf{171}
(2022), 2407–2459.

\bibitem{GT2001}
D.~Gilbarg and N.~S. Trudinger,
\emph{Elliptic Partial Differential Equations of Second Order},
Classics in Mathematics, Springer, Berlin, 2001.

\bibitem{HKM06}
J.~Heinonen, T.~Kilpel{\"a}inen, and O.~Martio, 
\emph{Nonlinear potential theory of degenerate elliptic equations}, Dover Publications, Mineola, NY, 2006.

\bibitem{KV1986}
S.~Kichenassamy and L.~V\'eron,
\emph{Singular solutions of the $p$-Laplace equation},
Math. Ann. \textbf{275} (1986), 599--615.

\bibitem{KV1987}
S.~Kichenassamy and L.~V\'eron,
\emph{Erratum: Singular solutions of the $p$-Laplace equation},
Math. Ann. \textbf{277} (1987), 352.

\bibitem{KM1994}
T.~Kilpel\"ainen and J.~Mal\'y,
\emph{The Wiener test and potential estimates for quasilinear elliptic equations},
Acta Math. \textbf{172} (1994), 137--161. 

\bibitem{KM2012}
T. Kuusi, G. Mingione,
\emph{Universal potential estimates}.
J. Funct. Anal. 262 (2012), no. 10, 4205--4269.

\bibitem{L1999}
Y.~Y. Li,
\emph{A Harnack type inequality: the method of moving planes},
Comm. Math. Phys. \textbf{200} (1999), no.~2, 421--444.

\bibitem{LZ2003}
Y.~Y. Li and L.~Zhang,
\emph{Liouville-type theorems and Harnack-type inequalities for semilinear elliptic equations},
J. Anal. Math. \textbf{90} (2003), 27--87.

\bibitem{L2016}
P. Lindqvist, \emph{A Remark on the Kelvin Transform for a Quasilinear Equation}, 
arXiv:1606.02563 (2016).


\bibitem{LZ2025}
G. Liu, Y. R.-Y. Zhang, \emph{Sharp stability for critical points of the Sobolev inequality in the absence of bubbling}, arXiv:2503.02340 (2025).


\bibitem{MW2024}
X.-N.~Ma and T.~Wu,
\emph{The application of the invariant tensor technique in the
classification of solutions to semilinear elliptic and sub-elliptic partial differential equations} (in Chinese),
Sci. Sin. Math. \textbf{54} (2024), no.~10, 1627--1648.


\bibitem{O1971}
M.~Obata,
\emph{The conjectures on conformal transformations of Riemannian manifolds},
J. Differential Geometry \textbf{6} (1971/72), 247--258.


\bibitem{O2025}
Q. Ou,
\emph{On the classification of entire solutions to the critical $p$-Laplace equation},
Math. Ann. \textbf{392} (2025), no. 2, 1711--1729.


\bibitem{QZ2026}
G.~Qin and Y.~R.-Y.~Zhang,
\emph{A Pohozaev-type neck proof of a conditional Harnack inequality in the critical $p$-Laplacian setting},
arXiv:2606.05990 (2026).

\bibitem{SZ1996}
R.~Schoen and D.~Zhang,
\emph{Prescribed scalar curvature on the $n$-sphere},
Calc. Var. Partial Differential Equations \textbf{4} (1996), no.~1, 1--25.

\bibitem{S2016}
B.~Sciunzi,
\emph{Classification of positive $\mathcal D^{1,p}(\mathbb R^N)$-solutions to the critical $p$-Laplace equation in $\mathbb R^N$},
Adv. Math. \textbf{291} (2016), 12--23.

\bibitem{SW2025}
L. Sun, Y. Wang, \emph{Critical quasilinear equations on Riemannian manifolds}, arXiv:2502.08495 (2025).


\bibitem{T1984}
P.~Tolksdorf,
\emph{Regularity for a more general class of quasilinear elliptic equations},
J. Differential Equations \textbf{51} (1984), 126--150.


\bibitem{V1984}
J.~L. V\'azquez,
\emph{A strong maximum principle for some quasilinear elliptic equations},
Appl. Math. Optim. \textbf{12} (1984), 191--202.

\bibitem{V2016}
J.~V\'etois,
\emph{A priori estimates and application to the symmetry of solutions for critical $p$-Laplace equations},
J. Differential Equations \textbf{260} (2016), no.~1, 149--161.

\bibitem{V2024}
J.~V\'etois,
\emph{A note on the classification of positive solutions to the critical
$p$-Laplace equation in $\mathbb R^n$},
Adv. Nonlinear Stud. \textbf{24} (2024), no.~3, 543--552.

\bibitem{YZ2026}
S.~Yan and H.~Zhou,
\emph{Domain variations and Pohozaev identities for weak solutions of $p$-Laplacian equation},
Acta Math. Sci. Ser. B (Engl. Ed.) \textbf{46} (2026), no.~2, 519--528.

\end{thebibliography}
\end{document}